\documentclass{amsart}
\usepackage[utf8]{inputenc}

\usepackage{amsmath,amsthm,amssymb}
\usepackage[margin=3cm]{geometry} 
\usepackage{patchcmd}
\usepackage{url}
\usepackage{mathtools}
\usepackage{enumitem}
\usepackage{graphics}
\usepackage[normalem]{ulem}
\usepackage{tikz}
\usepackage{lineno}
\usetikzlibrary{arrows,arrows.meta,snakes}

 \usepackage{hyperref}
  \hypersetup{
colorlinks=true,
linkcolor=cyan,
citecolor=cyan
}

\setitemize[1]{leftmargin=*,labelindent=.5em}
\setenumerate[1]{leftmargin=*,labelindent=.5em}

\newcommand{\lt}[1][2]{t$^{#1}$}

\title{Arc-wise analytic t-stratifications}

\author[Pablo Cubides Kovascics]{Pablo Cubides Kovacsics}
\address{Pablo Cubides Kovacsics, Mathematisches Institut der Heinrich-Heine-Universit\"at D\"usseldorf, 
Universit\"atsstr. 1, 40225 D\"usseldorf, Germany. }

\author[Immanuel Halupczok]{Immanuel Halupczok}
\address{Immanuel Halupczok, Mathematisches Institut der Heinrich-Heine-Universit\"at D\"usseldorf, 
Universit\"atsstr. 1, 40225 D\"usseldorf, Germany. }
\thanks{This paper was funded by the individual research grant No.~426488848 ``\emph{Archimedische und nicht-archimedische Stratifizierungen höherer Ordnung}'' of the second author, funded by the Deutsche Forschungsgemeinschaft (DFG, German Research Foundation). The first author has been fully financed by this grant.
In addition, both authors were part of the research training group
\emph{GRK 2240: Algebro-Geometric Methods in Algebra, Arithmetic and Topology}, also funded by the Deutsche Forschungsgemeinschaft.
}

\newcommand{\valring}{\mathcal{O}} 
\newcommand{\maxid}{\mathcal{M}} 

\newcommand{\affdir}{\operatorname{affdir}}

\newcommand{\eq}{^{\mathrm{eq}}}
\newcommand{\1}{^{-1}}

\newcommand{\family}{_{\bullet,\bullet}}

\newbox\removebox
\long\def\bigsout#1{%
\par\setbox\removebox=\vbox{#1}%
\vbox{%
\vbox to0pt{\hbox{\tikz\draw[thick] (0,0) -- (\wd\removebox,-\ht\removebox)  (\wd\removebox,0) -- (0,-\ht\removebox);}}
\box\removebox
}
}

 \theoremstyle{plain}
 \newtheorem{thm}{Theorem}[subsection]
 \newtheorem{cor}[thm]{Corollary}
 \newtheorem{lem}[thm]{Lemma}
 \newtheorem{fact}[thm]{Fact}
 \newtheorem{prop}[thm]{Proposition}
  \newtheorem{lemdef}[thm]{Lemma-Definition}
 \newtheorem{cla}[thm]{Claim}
 
\theoremstyle{definition}
 \newtheorem{defn}[thm]{Definition}

\theoremstyle{remark}
 \newtheorem{rem}[thm]{Remark}
 \newtheorem{ques}[thm]{Question}

 \newtheorem{nota}[thm]{Notation}
 \newtheorem{conv}[thm]{Convention}
  
 \newtheorem{exam}[thm]{Example}

 \numberwithin{equation}{section}

\newcommand{\C}{\mathbb{C}}
\newcommand{\N}{\mathbb{N}}
\newcommand{\Q}{\mathbb{Q}}
\newcommand{\R}{\mathbb{R}}

\DeclareMathOperator{\RV}{RV}

 \DeclareMathOperator{\id}{id}

 \DeclareMathOperator{\ac}{\overline{ac}}

 \DeclareMathOperator{\tsp}{tsp}

\DeclareMathOperator{\GL}{GL}
\DeclareMathOperator{\res}{res}  

\def\XXint#1#2#3{{\setbox0=\hbox{$#1{#2#3}{\int}$}
\vcenter{\hbox{$#2#3$}}\kern-.5\wd0}}

\newcommand{\dd}{{\mathbf{d}}}
\newcommand{\rado}{\operatorname{rad}_{\mathrm{o}}}
\newcommand{\radc}{\operatorname{rad}_{\mathrm{c}}}

\newcommand{\lenew}{\leqslant^*}
\newcommand{\lsnew}{<^*}

\renewcommand{\phi}{\varphi_{\text{replace phi by varphi}}}
\newcommand{\Lx}{\mathcal{L}}

\newcommand{\Lval}{\Lx_{\mathrm{val}}}
\newcommand{\Lomin}{\Lx_{\mathrm{orig}}}

\DeclareMathOperator{\rv}{rv}

\DeclareMathOperator{\val}{v}

\DeclareMathOperator{\Span}{Span}

\DeclareMathOperator{\dir}{dir}
\DeclareMathOperator{\Arc}{Arc}
\DeclareMathOperator{\arc}{arc}

\DeclareMathOperator{\Crit}{crit}
\DeclareMathOperator{\Tub}{Tub}
\DeclareMathOperator{\cS}{{\mathcal{S}}}

\newbox\gnBoxA
\newdimen\gnCornerHgt
\setbox\gnBoxA=\hbox{$\ulcorner$}
\global\gnCornerHgt=\ht\gnBoxA
\newdimen\gnArgHgt

\def\code #1{%
	\setbox\gnBoxA=\hbox{$#1$}%
	\gnArgHgt=\ht\gnBoxA%
	\ifnum \gnArgHgt<\gnCornerHgt
		\gnArgHgt=0pt%
	\else
		\advance \gnArgHgt by -\gnCornerHgt%
	\fi
	\raise\gnArgHgt\hbox{$\ulcorner$} \box\gnBoxA %
		\raise\gnArgHgt\hbox{$\urcorner$}}

\definecolor{private}{RGB}{120,0,120}

\long\def\private#1{
\bgroup
\color{private}
\par
Private Remark: #1
\par
\egroup
}

\begin{document}



\maketitle

\begin{abstract}
We introduce two new notions of stratifications in valued fields: \lt-stratifications and arc-wise analytic t-stratifications. We show the existence of arc-wise analytic t-stratifications in algebraically closed valued fields with analytic structure in the sense of R.~Cluckers and L.~Lipshitz. We prove that arc-wise analytic t-stratifications are \lt-stratifications and, moreover, that \lt-stratifications are valuative Lipschitz stratifications as defined by the second author and Y.~Yin (the latter ones being closely related to Lipschitz stratifications in the sense of Mostowski). Finally, we introduce a combinatorial invariant associated to a t-stratification which we call the \emph{critical value function}. We explain how the critical value function of arc-wise analytic t-stratifications can be used to formulate programatic conjectural bounds for the Nash-Semple conjecture. 
\end{abstract}

\setcounter{tocdepth}{1}
{
  \hypersetup{linkcolor=black}
  \tableofcontents
}

\section{Introduction}

Stratifications are a classical tool to control singularities of certain (e.g.\ algebraic) subsets of $\R^n$ or $\C^n$. Recently, similar notions have been developed over valued fields and an interesting interplay between the classical and the valued fields variants emerged. In \cite{halupczok2014a}, the second author introduced the notion of \emph{t-stratifications} over valued fields. He proved their existence over various classes of valued fields and showed how they can be used to induce Whitney stratifications over $\R$ and $\C$, by working in non-standard models (of $\R$ and $\C$). In a similar vein, a valuative analogue of Mostowski's Lipschitz stratifications \cite{mostowski} was introduced in \cite{halup-yin-18} by the second author and Y. Yin. Analogously, they established their existence in a variety of valued fields and showed how they yield classical Lipschitz stratifications.

The aim of this paper is to further improve and develop stratifications in valued fields, exploiting the additional possibilities provided by the valuation, with the idea in mind that this then yields corresponding notions and results over $\R$ and $\C$. Our main motivation is that
current classical notions (e.g.\ Whitney, Verdier, Lipschitz) do not seem to be robust and/or strong enough to deal with certain problems in singularity theory, and Lipschitz stratifications (one of the strongest classical notions) are considerably technical to define, making it difficult to further strengthen that notion using classical methods. By working in valued fields instead, we find a quite simple and natural notion of stratifications (which we call \emph{\lt-stratifications}) which is a common refinement of the t-stratifications and the valuative Lipschitz stratifications mentioned in the first paragraph. In particular, if one is willing to work in valued fields, they can serve as a much more accessible replacement for classical Lipschitz stratifications. Figure \ref{f.strat} shows the known logical connections between these notions of stratifications (and some more).

\begin{figure}
\begin{tikzpicture}[
strat/.style={draw,align=center,fill=white,thick},
nstrat/.style={strat,fill=blue!20,ultra thick},
imply/.style={->,double equal sign distance,>=Implies},
y=1.3cm,
corr/.style={<-,snake=snake,line before snake=2mm,segment amplitude=.5mm}]


\def\xal{-.5}
\def\xar{2.5}
\def\xnl{6.5}
\def\xnr{10}

\node[strat] at (\xal,4) (aPP) {arc-wise\\analytic\ stratification};
\node[strat] at (\xar,2.5) (aL) {Lipschitz\\stratification};
\node[strat] at (\xal,1.5) (av) {Verdier\\stratification};
\node[strat] at (\xal,0) (aw) {Whitney\\stratification};

\draw[imply] (aPP) -- (av); 
\draw[imply] (av) -- (aw); \draw[imply] (aL) -- (av);


\node[strat,label={270:\footnotesize}] at (\xnl,2.5) (nL) {valuative\\Lipschitz\\ stratification};

\node[nstrat,label={290:\footnotesize }] at (\xnr,4) (nti) {arc-wise\\analytic t-stratification};
\node[nstrat,label={290:\footnotesize}] at (\xnr,2.5) (nt2) {\lt-stratification};
\node[strat,label={270:\footnotesize}] at (\xnr,1) (nt1) {t-stratification};

\draw[imply] (nti) -- (nt2);
\draw[imply] (nt2) -- (nt1);

\draw[imply,->] (nt2) -- (nL);


\draw[corr] (aw) -- (nt1);
\draw[corr] (aL) -- (nL);

\node at (1,5) {Stratifications over $\R$ and $\C$};
\node at (8,5) {Stratifications over valued fields};
\draw (4.6,-.5) -- (4.6,5.2);

\end{tikzpicture}
\caption{$A \Rightarrow B$ means: Every $A$-stratification is in particular a $B$-stratification.
$A \rightsquigarrow B$ means: an $A$-stratification in a nonstandard model of $\R$ or $\C$ induces a $B$-stratification in $\R$ or $\C$. The bold blue boxes are newly introduced in this paper.}
\label{f.strat}
\end{figure}
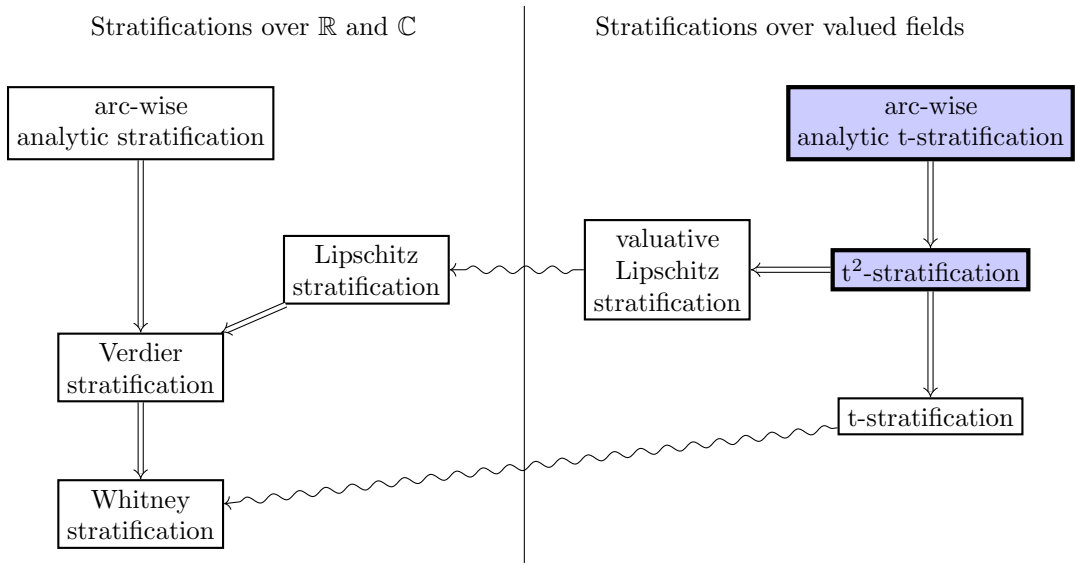

Having these \lt-stratifications at hand, it becomes very easy and natural to come up with further strengthenings; we do so by introducing \emph{arc-wise analytic t-stratifications}. We prove their existence in algebraically closed valued fields (but not yet e.g. in non-standard models of $\R$).
Our hope is that those have the desired robustness properties to which we alluded above. Concretely, we hope that arc-wise analytic t-stratifications can be helpful to prove the Nash-Semple Conjecture (apparently first suggested in \cite{semple1954} by J. G. Semple, see also \cite{perez-teissier, spivakovsky}). That conjecture
predicts that there is a very simple and natural way to construct a resolution of singularities of an arbitrary algebraic variety over $\C$, namely as a finite number of Nash modifications (which is a variant of a blow-up). The Nash-Semple conjecture is widely open, but it has been predicted by Parusinski that one should be able to prove it using a notion of stratification which 
is a common refinement of (a) Lipschitz stratifications and (b) the (archimedean) \emph{arc-wise analytic stratifications} introduced in \cite{PP.arcZariski};
The authors of the present article share this belief. The newly introduced arc-wise analytic t-stratifications do satisfy (a) (when applied in non-standard models of $\C$). We do not know whether they formally satisfy (b), since our notion of arc-wise analyticity is inherently a valued field notion, and it is not clear at all whether it yields some form of arc-wise analyticity in the archimedean sense (over $\C$). However, arc-wise analytic t-stratifications might nevertheless be a suitable ingredient for the proof of the Nash-Semple conjecture, namely if one is willing to carry out the entire proof within valued fields.

Although at this stage this is merely speculative, we present in Section~\ref{sec:nash}
some ideas on how such a proof might work. A potential key ingredient is a combinatorial invariant which we associate to any t-stratification in Section \ref{sec:exponent} (the \emph{critical value function}), which simultaneously captures various kinds of asymptotic distances (e.g.\ between two strata) near all singularities. We prove that even though the information it captures is rather general, it consists of the graphs of only finitely many $\Q$-monomial functions. Our hope is that the number of Nash modifications needed to resolve the singularities of an algebraic variety over $\C$ can be bounded in terms of the (finitely many) rational numbers describing the critical value function of an arc-wise analytic t-stratification of the variety; we formulate this as a conjecture in Section \ref{sec:nash}.
In any case, the critical value function seems to be of independent interest. In particular, it seems to be related to known invariants such as (certain) Newton polygons and probably also to the poles of motivic Poincaré series (though verifying this is left to follow-up work).

\subsection{Summary of the main new concepts and results}

In the entire paper, we work in a valued field $K$ of equicharacteristic $0$, considered as a structure in a suitable first order language.
With the Nash-Semple conjecture in mind, one might mostly be interested in algebraic subsets of $K^n$ and in the case where $K$ is algebraically closed. However, even then, our notions of stratifications require that one also considers more general sets, namely sets definable in the valued field language (i.e., the ring language together with a predicate for the valuation ring). Moreover, we believe that our stratifications will also have other applications, e.g.\ to geometry in valued fields, so large parts of the paper are formulated in rather big generality, namely allowing the valued field $K$ to be henselian, and for arbitrary languages satisfying a tameness property which is natural in this context, namely the notion of Hensel minimality from \cite{clu-hal-rid}. (To be precise, we use $1$-h-minimality.)

The central newly introduced objects in the paper are the following:
\begin{enumerate}
    \item \emph{\lt-stratifications} (Definition~\ref{def:sts}):
    This is a slight strengthening of the definition of t-stratifications from \cite{halupczok2014a}: As for t-stratifications, the only regularity condition we impose is that on any valuative ball $B$, the stratification is ``roughly translation invariant'' in the direction of the lowest dimensional stratum intersecting $B$. For t-stratifications, the notion of ``roughly translation invariant'' allows for a linear error term. For \lt-stratifications, we impose a quadratic error term along arcs. (This point of view in terms of the error term is made more precise in Section~\ref{sec:hierarchy}, where we speculatively also discuss t$^r$-stratifications for arbitrary integers $r \geqslant 1$.) 
    \item \emph{Arc-wise analytic t-stratifications} (Definition~\ref{def:arc-wise-strats}):
    Again, the definition builds on the notion of t-stratifications, but this time, we impose that the arcs are analytic. There is no notion of analyticity in arbitrary Hensel minimal fields $K$, so in this part of the paper, we restrict to the (still quite general) notion of \emph{field with analytic structure} from \cite{CLip}. Moreover, even though analytic structures exist on henselian valued fields $K$, our definition of arc-wise analytic t-stratification works properly only under the additional assumption that $K$ is algebraically closed.
    \item\label{it:cvf} The \emph{critical value function} (Definition~\ref{def:critpt}):
    The most basic information captured by the critical value function is the asymptotic distance between two strata $X_1$ and $X_2$ near a singularity $y$: if $x_1$ runs along $X_1$, then
    the critical value function specifies 
    the (valuative) distance of $x_1$ to $X_2$ as a function of the (valuative) distance of $x_1$ to $y$. The full notion of critical value function improves upon this in the following ways: the ``distance from $x_1$ to another stratum'' is replaced by a more general concept which we call \emph{critical value at $x_1$} and which also captures e.g.\ the diameter of $X_1$ near $x_1$ if $X_1$ is a kind of cone as in Figure~\ref{fig:trumpet} on p.~\pageref{fig:trumpet}; instead of expressing these critical values only in terms of the distance from $x_1$ to some fixed singularity $y$, it is expressed in terms of the distances from $x_1$ to all (lower-dimensional) strata.
\end{enumerate}
Large parts of the paper are devoted to setting up a whole toolbox to work with those notions of stratifications, building on and improving results from \cite{halupczok2014a} about t-stratifications. The main results we obtain (using that toolbox) are the following:
\begin{enumerate}[resume]
 \item
 \emph{Every \lt-stratification is in particular a valuative Lipschitz stratification in the sense of \cite{halup-yin-18}}  (Theorem~\ref{thm:sts-to-Lipstrats}). More precisely, the notion of valuative Lipschitz stratifications only makes sense in a suitable non-standard setting and when moreover the stratification is definable in a language not involving the valuation. The above implication holds under this assumption.
 
 In particular, this means that if one takes the method (from \cite{halupczok2014a}) which turns a t-stratification into a Whitney stratification, and applies it to a
 \lt-stratification, then one obtains a Lipschitz stratification in the original sense of Mostowksi; this is explained in Remark \ref{rem:t2-to-lip}.
  \item\label{it:arc-existence} \emph{In algebraically closed valued fields, arc-wise analytic t-stratifications exist} (Theorem~\ref{thm:arc-existence}).
 We prove this not only for the pure valued field structure on $K$, but more generally for any analytic structure in the sense of \cite{CLip}, i.e., given a set $X \subseteq K^n$ definable in such a structure, there exists a stratification of $K^n$ ``\emph{reflecting}'' $X$. Note that just asking $X$ to be a union of connected components of the strata would be a very weak condition, given that $K$ is totally disconnected. The notion of reflection (introduced in \cite{halupczok2014a} for t-stratifications) fixes this problem in a natural way.
 \item \emph{Every arc-wise analytic t-stratification is in particular a \lt-stratification} (Theorem~\ref{thm:aats-to-lipts}). We do not give a separate existence proof of \lt-stratifications, so their existence only follows indirectly from this implication and Result (\ref{it:arc-existence}) above (Corollary \ref{cor:t2-exist}). In particular, this means that we currently know the existence only in algebraically closed valued fields. (We believe that they exist in arbitrary $1$-h-minimal henselian valued fields.)

\item\label{it:cvf2} The critical value function ``consists of finitely many monimial functions'' (Theorem~\ref{thm:crit} and Corollary~\ref{cor:crit}). More precisely, picking up notation from (\ref{it:cvf}) if we consider all points $x_1$ at given fixed valuative distances $\lambda_0,\dots, \lambda_{d}$ to each of the (lower-dimensional) strata and take all critical values at all those points, we only obtain finitely many different values. Moreover, those values are piecewise monomial in the $\lambda_i$, i.e., of the form $\lambda_0^{r_0}\cdots\lambda_d^{r_d}$ for some $r_i \in \Q$. In addition, the pieces themselves are defined by finitely many monimial conditions, so that in total, the critical value function can be described by finitely many rational numbers.
\end{enumerate}

\subsection{Structure of the paper}
In Section \ref{sec:setting}, we set up the framework of valued fields we will work with, and prove a variety of preparatory lemmas needed throughout. In particular, Subsection~\ref{sec:val-la} is dedicated to interaction of linear algebra with valuations, and Subsection~\ref{sec:modth} fixes the model theoretic setting and conventions.

In Section \ref{sec:tsts} we recall the definition of a t-stratification and revisit some of its main properties. In particular, there is an entire toolbox of little results that are useful in this context and that will be needed later.

Section \ref{sec:lt} focuses on the new notion of \lt-stratification. They are introduced in Subsection \ref{sec:SUT} (the main result of that section is Proposition~\ref{prop:SU-equivalence}, which gives an alternative characterization), and in Section \ref{sec:lipschitz} we show that they are valuative Lipschitz stratifications (Theorem \ref{thm:sts-to-Lipstrats}). We also recall the relation between valuative Lipschitz stratification and (classical) Lipschitz stratifications (Propositions~\ref{prop:liplip} and \ref{prop:liplipacf}) and explain how to put all those things together so that \lt-stratifications can serve as a replacement for Lipschitz stratifications (Remark~\ref{rem:t2-to-lip}).

Section~\ref{sec:analytic} is devoted to the even stronger notion of arc-wise analytic t-stratitications. We start by fixing the model theoretic context for this, in particular imposing that $K$ is algebraically closed and recalling the notion of field with analytic structure from \cite{CLip} (Section~\ref{sec:an_struct}), we define arc-wise analytic t-stratitications and prove their existence (Theorem \ref{thm:arc-existence}) in Section~\ref{sec:an_exist}, and we conclude by showing that they are in particular \lt-stratifications (Theorem \ref{thm:aats-to-lipts} in Section~\ref{sec:aats-to-lipts}).

The next section (Section~\ref{sec:exponent}) is independent from Sections~\ref{sec:lt} and \ref{sec:analytic}: We introduce the critical value function of a t-stratification (Definition \ref{def:critpt}), and we show, under some (natural) more restrictive assumptions on the language than $1$-h-minimality, the properties explained above in (\ref{it:cvf2}).

Finally, we gather in Section \ref{sec:discussion} some less formal discussions about potential follow-up work, in particular
about a hierarchy of ``t$^r$-stratifications''
for integers $r \geqslant 1$ (Section~\ref{sec:hierarchy}) and about the relation to the Nash-Semple conjecture (Section~\ref{sec:nash}). We also list some questions that we stumbled over while working on this paper.

\subsection{Advice for the impatient reader}

Think of $K$ as being an algebraically closed non-trivially valued field in the pure valued field language (a setting in which all results of this paper hold), note that we use multiplicative notation for the valuation (but still denote it by $v\colon K \to \Gamma = \Gamma^\times \cup \{0\}$), and read Definitions~\ref{def:straighened} and \ref{def:-t-strat} if you are not yet familiar with t-stratifications. After that, you should be able to roughly understand most new definitions and results later in the paper (assuming some familiarity with model theory of valued fields).

\section{Setting}\label{sec:setting}

We will closely follows the terminology settled in \cite{halupczok2014a} with a few minor changes but one major caveat: in contrast to \cite{halupczok2014a}, in the present article we chose to use multiplicative notation for the valuation. One of the reasons for this choice is that the multiplicative point of view might be better to quickly develop a geometric intuition of the objects here considered, and it is such an intutition which helps in understainding most of the arguments  presented here. For this reason, and also to be relatively self-contained, we start by recalling most definitions from \cite{halupczok2014a} that will be needed, but now written in multiplicative notation.

\subsection{Notation and terminology of valued fields}
\label{sec:notnVF}

In the entire paper, $(K,v)$ is a non-trivially valued field of equi-characteristic $0$. As stated above, we use multiplciative notation so the valuation is a map $v\colon K\to \Gamma$ where 
\begin{itemize}
    \item $\Gamma\coloneqq \Gamma^\times \cup \{0\}$, 
    \item $(\Gamma^\times,\cdot,1,<)$ is an ordered abelian group, 
    \item $0$ is an element such that $0<\Gamma^\times$, 
\end{itemize}
and elements $x,y\in K$ satisfy 
\begin{itemize}
    \item $v(x)=0$ if and only if $x=0$,
    \item $v(xy)=v(x)v(y)$,
    \item $v(x+y)\leqslant \max\{v(x),v(y)\}$.
\end{itemize}    
As usual, the group~$\Gamma^\times$ is called the value group of~$(K,v)$. Note that we do not assume that $\Gamma^\times$ is isomorphic to a subgroup of $(\R_{>0},\cdot,1,<)$, that is, we work with general Krull valuations. 

We let~$\valring$ be the valuation ring of~$(K,v)$, $\maxid$ its maximal ideal, $\bar K$ be the residue field and $\res\colon \valring\to \bar K$ be the quotient map. 

Abusing of notation, for $a=(a_1,\ldots,a_n)\in K^n$, we set $v(a)\coloneqq \max_{i} v(a_i)$.
Given $\gamma\in \Gamma^\times$, the set 
\begin{align*}
B(a, < \gamma) &\coloneqq \{x \in K^n \mid v(x - a) < \gamma\} \qquad \text{resp.}\\
B(a, \leqslant \gamma) &\coloneqq \{x \in K^n \mid v(x - a) \leqslant \gamma\}
\end{align*}
is the open resp.\ closed ball centered at~$a$ with radius~$\gamma$. Additionally, we consider $K^n$ as an open ball of radius $\infty$. (However, we do not consider individual points as closed balls.)
By a ball, we mean either an open or a closed ball. Given an open (resp. closed) ball~$B$,
we use the notation~$\rado(B)$ (resp.~$\radc(B)$) for its radius (and set $\rado(K^n) = \infty$). Note that when $\Gamma^\times$ is a discrete value group, then any closed ball $B$ is also an open ball, but~$\rado(B)$ and~$\radc(B)$ differ.

We equip $K$ with the topology generated by the basis of open balls. It is easy to see that the topology on~$K^n$ generated by the open balls (in~$K^n$) coincides with the product topology.

Given subsets $X, Y\subseteq K^n$, we write $v(X-Y)$ for the infimum in $\Gamma$ (when it exists) of the set $\{v(x-y) \mid x\in X, y\in Y\}$. For a point $x\in K^n$, we set $v(x-Y)$ as $v(\{x\}-Y)$. Note that if $Y$ is closed, $v(x-Y)$ always exists in $\Gamma$ and, moreover, there is $y_0\in Y$ such that $v(x-Y)=v(x-y_0)$.

Given $a=(a_1,\ldots, a_n)\in \valring^n$, we write $\res(a)$ for the coordinate-wise map 
\[
a\mapsto (\res(a_1),\ldots,\res(a_n))\in \bar K^n.
\]
For a variety $X$ defined over $\valring$, the map $\res\colon X(\valring) \to X(\bar K)$ is defined similarly.

We denote the ``leading term structure'' of $K$ by $\RV$ and also its higher dimensional analogue introduced in \cite{halupczok2014a} (which is not just a cartesian power of $\RV$). Those are defined as follows:

\begin{defn}\label{defn:RVn}
Define $\RV^{(n)} \coloneqq K^n/\mathord{\sim}$, where 
\[
x \sim y \iff (\val(x - y) < \val(x) \,\vee\, x = y = 0).
\] 
We write $\rv\colon K^n \to \RV^{(n)}$ for the canonical quotient map. Instead of~$\RV^{(1)}$, we simply write~$\RV$. 
\end{defn}

Note that if $x, y \in K^n$ satisfy $v(x) = v(y) = 1$, then we have $\rv(x) = \rv(y)$ if and only if $\res(x) = \res(y)$.

For a coordinate free definition of the leading term structures see \cite[Definition 2.5]{halupczok2014a}. When no confusion arises, we also use the letter~$v$ for the map $v\colon \RV^{(n)}\to \Gamma$ satisfying $(v \circ \rv)(a) = v(a)$ for every $a\in K^n$.

\begin{defn}\label{def:risometry} Let $X$ and $ Y$ be subsets of $K^n$. A bijection $\varphi\colon X\to Y$ is a \emph{risometry} if it preserves the leading terms of differences, i.e., if for all~$x,y\in X$ we have
\[
\rv(x-y)=\rv(\varphi(x)-\varphi(y)).  
\]
\end{defn}
In other words, a bijection $\varphi\colon X\to Y$ is a risometry if 
for all distinct~$x,y\in X$, we have
\[
v(\varphi(x)-\varphi(y)-(x-y))<v(x-y). 
\]
Most of the time, we will be interested in risometries $\varphi\colon B\to B$ for a given ball $B\subseteq K^n$.

\begin{rem}\label{rem:finite_riso}
It is worthy to note that by \cite[Lemma 2.15]{halupczok2014a}, when $X$ and $Y$ are finite sets, there is at most one risometry between them.
\end{rem}

We will use the following notation concerning projections. Let $\pi\colon K^n \to K^d$ be a coordinate projection (i.e., sending $(x_1, \dots, x_n)$ to $(x_{i_1}, \dots, x_{i_d})$ for some $1 \leqslant i_1 <  \dots < i_d \leqslant n$). When no confusion arises, we also write $\pi\colon \bar K^n \to \bar K^d$ for the corresponding projection on $\bar K^n$. We let $\pi^\vee\colon K^n \to K^{n-d}$ denote the projection to the complementary set of coordinates compared to $\pi$.
By identifying a fiber $\pi\1(x)$ of a coordinate projection (for $x \in K^d$) with $K^{n-d}$ via $\pi^\vee$, definitions made for $K^{n-d}$ can also be applied to such fibers. As an example, by a ball in $\pi\1(x)$, we mean a subset $B \subseteq \pi\1(x)$ such that $\pi^\vee(B)$ is a ball in $K^d$.

\subsection{Valued linear algebra}
\label{sec:val-la}

\begin{defn}
Given a $K$-vector subspace $U$ of $K^n$, we let $\res(U)$ be the $\bar K$-vector space 
\[
\res(U)\coloneqq \{\res(a) : a\in U\cap \valring^n\}.
\]
If $\bar V \subseteq \bar K^n$ is a $\bar K$-vector space, then any vector space
$V \subseteq K^n$ with $\res(V) = \bar V$ is called a \emph{lift}
of $V$. 
\end{defn}

Note that $\bar V$ and $V$ (as given in the previous definition) are vector spaces of the same dimension. 
Often, we will denote $\res(U)$ by $\bar U$, but sometimes, we just have a $\bar{K}$-vector space $\bar U \subseteq \bar K^n$, without a choice of a lift $U$.

\begin{fact}\label{fact:res-vs-rv}
For $K$-vector spaces $U, W\subseteq K^n$ of the same dimension, we have $\res(U) = \res(W)$ if and only if
for all~$u \in U$, there is~$w \in W$ such that $\rv(u) = \rv(w)$. In addition, it suffices to check the previous condition for all~$u\in U$ of a fixed given valuation. \qed
\end{fact}

\begin{defn}\label{def:dir} Let $x\in  K^n$.  The \emph{direction of $x$} is defined as the zero or one-dimensional subspace $\dir(x)\coloneqq \res(K \cdot x)$ of $\bar K^n$.
\end{defn} 

Note that if $x, y \in K^n$ satisfy $\rv(x) = \rv(y)$, then $\dir(x) = \dir(y)$.

\begin{defn}\label{def:sub-aff} Let $C$ be a subset of $K^n$. The \emph{affine direction space of $C$}
is the $\bar K$-vector sub-space $\affdir(C) \subseteq \bar K^n$ generated by $\dir(x - x')$, where $x, x'$ run through $C$. 
\end{defn}

Directly from the definition of risometry and affine direction space one obtains: 

\begin{fact}\label{fact:riso-affdir}
Let $\varphi\colon X\to Y$ be a risometry for $X, Y \subseteq K^n$. Then $\affdir(X)=\affdir(\varphi(X))$.\qed 
\end{fact}

The following facts follows from \cite[Lemma 2.8]{halupczok2014a} (see also \cite[Remark 2.13]{halupczok2014a}). 

\begin{fact}\label{fact:GL_n-riso}
A matrix $M \in \GL_n(K)$ is an isometry if and only if $M\in \GL_n(\valring)$. It is a risometry if and only if $M \in \GL_n(\valring)$ and $\res(M) = I_n\in \GL_n(\bar K)$. \qed
\end{fact}

\begin{fact}\label{fact:vs-GL_n} Let~$U$ and~$W$ be two $K$-vector spaces in~$K^n$ of dimension~$d$. If~$\res(U)=\res(W)$, then there is a risometry $M\in\GL_n(\valring)$ sending $U$ to $W$.\qed
\end{fact}

\begin{defn}\label{def:exhibition}
Let $\bar U \subseteq \bar K^n$ be a sub-vector space. A coordinate projection $\pi\colon K^n \to K^d$ is called an \emph{exhibition of $\bar{U}$} if the corresponding projection $\pi\colon \bar K^n\to \bar K^d$ induces an isomorphism $\pi_{|\bar{U}}\colon \bar{U}\to \bar K^d$. 
\end{defn}

\begin{rem}\label{rem:exhibition}
Given $\bar U \subseteq \bar K^n$ and a lift $U \subseteq K^n$, a coordinate projection $\pi\colon K^n \to K^{d}$ for $d = \dim(U)$ is an exhibition of $\bar U$ if and only if for all $x \in U$, we have $v(x)=v(\pi(x))$. In particular, if $\pi$ is an exhibition, then $\pi_{|U}\colon U \to K^d$ is an isomorphism (since $\pi_{|U}$ is injective). 
\end{rem}

The following corollary follows form Facts \ref{fact:GL_n-riso} and \ref{fact:vs-GL_n}. 

\begin{cor}\label{fact:proj-unique}
Suppose $\pi$ is an exhibition of $\bar{U}$ and let $U,U'$ be two lifts of $\bar{U}$. Then, there is a unique matrix $M \in \GL_n(K)$ sending $U$ to $U'$ and satsifying that $\pi \circ M = \pi$. In addition, $M$ is a risometry. \qed
\end{cor}

\begin{lem}\label{lem:dir} Let $\bar{U}$ be a $\bar K$-vector subspace of $\bar K^n$ and $\pi\colon K^n\to K^d$ be an exhibition of $\bar{U}$. Suppose that $x,x'\in K^n$ are such that 
$\dir(x), \dir(x') \subseteq \bar{U}$ and $\rv(\pi(x)) =\rv(\pi (x'))$. Then $\rv(x) = \rv(x')$.
\end{lem}

\begin{proof} Since $\pi$ is an exhibition of $\bar{U}$ and $\dir(x), \dir(x') \subseteq \bar{U}$, we have $v(x) = v(\pi(x))=v(\pi(x')) = v(x')$. Possibly multiplying $x$ and $x'$ by the same element of $K^\times$, we may suppose $v(x)=v(x')=1$ (the case $x=x'=0$ is trivial). It suffices to show that $\res(x)=\res(x')$. Since $\rv(\pi(x))= \rv(\pi(x'))$, we deduce that $\res(\pi(x)) = \res(\pi(x'))$, and therefore that $\pi(\res(x)) = \pi(\res(x'))$. Now, $\pi$ is injective on $\bar{U}$, so since $\res(x), \res(x') \in \bar{U}$, we obtain that $\res(x) = \res(x')$.
\end{proof}

\begin{lem}\label{lem:riso-proj}
Let $X \subseteq K^n$ be a non-empty set and $\pi\colon K^n \to K^d$ be an exhibition of $\affdir(X)$.
Then $\pi_{|X}$ is a bijection from $X$ to $\pi(X)$
and for any map
$\varphi\colon X \to Y \subseteq K^n$, the following are equivalent:
\begin{enumerate}
    \item $\varphi$ is a risometry.
    \item $\affdir(Y) = \affdir (X)$ and the induced map
    $\psi\colon \pi(X) \to \pi(Y)$ (i.e., satisfying $\pi(\varphi(x)) = \psi(\pi(x))$ for all $x \in X$) is a risometry.
\end{enumerate}
\end{lem}

\begin{proof}
Since $\pi$ is an exhibition of $\affdir(X)$, for any $x, x' \in X$, we have $v(\pi(x)-\pi(x'))=v(x-x')$. This shows $\pi_{|X}$ is injective, hence a bijection to $\pi(X)$. 

(1) $\Rightarrow$ (2)
The first part of (2) is Fact~\ref{fact:riso-affdir},
so it remains to show that for $x, x' \in X$, we have $\rv(\pi(x) - \pi(x')) = \rv(\pi(\varphi(x)) - \pi(\varphi(x')))$.
Since $\pi$ is an exhibition of $\affdir(X) = \affdir(Y)$, we have 
\[v(\pi(x)-\pi(x'))=v(x - x')
\overbrace{=}^{\text{by (1)}}
v(\varphi(x)-\varphi(x')) =v(\pi(\varphi(x))-\pi(\varphi(x'))).
\]
Without loss, suppose that all these valuations are equal to $1$. Then,  
\[
\res(\pi(x)-\pi(x')) = \pi(\res(x-x')) \overbrace{=}^{\text{by (1)}} \pi(\res(\varphi(x)-\varphi(x'))) =  \res(\pi(\varphi(x))-\pi(\varphi(x'))),
\]
which implies $\rv(\pi(x)-\pi(x'))=\rv(\psi(\pi(x))-\psi(\pi(x')))$. 

(2) $\Rightarrow$ (1)
Given $x, x' \in X$, we have $\rv(\pi(x - x')) = \rv(\pi(\varphi(x) - \varphi(x')))$ (since $\psi$ is a risometry). Combining with $\dir(x - x'), \dir(\varphi(x) - \varphi(x')) \subseteq \affdir(X) = \affdir(Y)$, we can apply Lemma~\ref{lem:dir} (to $x - x'$ and $\varphi(x) - \varphi(x')$) to deduce $\rv(x - x') = \rv(\varphi(x) - \varphi(x'))$.
\end{proof}

\begin{lem}\label{lem:dir2} Let $\pi\colon K^n\to K^d$ be a coordinate projection
and suppose that $x\in K^n$ satisfies $v(x)=v(\pi(x))$.
Then $\pi(\dir(x))=\dir(\pi(x))$.  
\end{lem}

\begin{proof}
If $x = 0$, the statement is clear, so assume now without loss that $v(x) = 1$.
In that case, $\dir(x) = \bar K\cdot\res(x)$,
and since by assumption we have $v(\pi(x)) = v(x)$, we also have $\dir(\pi(x)) = \bar K\cdot\res(\pi(x))$. Thus
\[
\pi(\dir(x))=
\pi(\bar K\cdot\res(x))
= \bar K\cdot\res(\pi(x))
= \dir(\pi(x)).\qedhere
\]
\end{proof}

\begin{lem}\label{lem:dirsum}
Suppose that $x_1, x_2 \in K^n$
satisfy $\dir(x_1) \cap \dir(x_2) = 0$.
Then $\dir(x_1+x_2) \subseteq \dir(x_1) + \dir(x_2)$. 
\end{lem}

\begin{proof} Without loss we may suppose that $x_1\neq 0$, $x_2\neq 0$. If $v(x_1) > v(x_2)$, then $\rv(x_1 + x_2) = \rv(x_1)$ implies  $\dir(x_1 + x_2) = \dir(x_1)$, so we are done, and similarly if $v(x_2)>v(x_1)$. After rescaling, we may thus assume $v(x_1) = v(x_2) = 1$.
Moreover, from $\dir(x_1) \ne \dir(x_2)$,
we deduce $v(x_1 + x_2) = 1$.
Thus $\dir(x_1 + x_2)  = \bar K\cdot\res(x_1 + x_2) \subseteq \bar K\cdot\res(x_1) + \bar K\cdot\res(x_2) =  \dir(x_1) + \dir(x_2)$.
\end{proof}

\begin{lemdef}
If $U$ is a sub-vector space of $K^n$, then we define the valuation $v(b)$ for $b \in K^n/U$ as the minimum $\min\{v(w) : w \in b\}$.
\end{lemdef}

\begin{proof}
The minimum in the previous definition always exists. Indeed, by applying a matrix from $\GL_n(\valring)$, we may assume that $U = K^m \times \{0\}^{n-m}$. In that case, $b = U + (0, \dots, 0, b_{m+1}, \dots, b_n)$, and 
\[
\min\{v(w): w\in b\}= \max_{m < i \leqslant n} v(b_i). \qedhere
\]
\end{proof}

\begin{lemdef}\label{def:VSdistance} Let $U, W\subseteq K^n$ be two $K$-vector spaces of the same dimension. The \emph{distance between $U$ and $W$}, denoted by $\Delta(U,W)$, is defined as the minimum of all $\gamma\in \Gamma$ for which the following holds: for every $u\in U$ there is $w\in W$ such that $v(u-w)\leqslant v(u)\cdot\gamma$. 
\end{lemdef}

\begin{proof}
Let us show that the above minimum exists. Firstly, since $W$ is closed, for each $u \in U$ there exists a $w \in W$ satisfying $v(u - w) = v(u - W)$, so the above minimum is also the supremum of $v(u-W)/v(u)$, where $u$ runs over $U \setminus \{0\}$. We claim that this supremum does not only exist, but is even a maximum.
We may without loss assume that $U = K^d \times \{0\}^{n-d}$ (by applying a matrix from $\GL_n(\valring)$). Then we claim that the maximum
is taken on one of the standard basis vectors $e_1, \dots, e_d$, i.e.,
the above supremum is equal to $\gamma \coloneqq \max_{1 \leqslant i \leqslant d} v(e_i - W)$.
Indeed, pick $w_1, \dots, w_d \in W$ such that $v(e_i - w_i) = v(e_i - W)$. Then, 
given an arbitrary $u = \sum_{i \leqslant d} r_i e_i$ in $U$, for $w \coloneqq \sum r_i w_i$, we obtain $v(u - w) \leqslant v(u) \cdot \gamma$, as an easy computation shows.
\end{proof}

The following is left as an exercise. 

\begin{lem}\label{lem:distance} Let $U,W\subseteq K^n$ be two $K$-vector spaces of the same dimension. Then, 
\begin{enumerate}
    \item $\res(U)=\res(W) \iff \Delta(U,W)<1$, 
    \item If $M\in \GL_n(\valring)$, then $\Delta(U,W) = \Delta(M(U),M(W))$. \qed
\end{enumerate}
\end{lem}

\begin{lem}\label{lem:nice-dist} Let $\{e_1,\ldots,e_n\}$ be the standard basis of $K^n$ and let $W = \Span_K(w_1, \dots, w_d)$ and $W' = \Span_K(w'_1, \dots, w'_d)$ be two sub-vector spaces of $K^n$ of dimension $d$ satisfying
\[
w_{j}, w'_{j} \in e_j +
(\{0\}^{d} \times \maxid^{n-d})
\]
for $j = 1, \dots, d$ (recall that $\maxid$ is the maximal ideal). Then $\Delta(W, W') = \max_j\{v(w_j - w'_j)\}$.
\end{lem}

\begin{proof}
Without loss, $w_j = e_j$ (apply a suitable matrix from $\GL_n(\valring)$). Then, for each $j$, $w'_j$ is an element of $W'$ minimizing $v(w_j-W')$. The result is obtained as in Lemma-Definition \ref{def:VSdistance}. 
\end{proof}

We finish this section with the following lemma. 

\begin{lem}\label{lem:average} Let $x_1,\ldots,x_m, x_1',\ldots,x'_m\in K^n$ and $x,x'\in K^n$ be their correponding averages, i.e.,
\[
x=1/m\sum_{i=1}^m x_i \text{ and } x'=1/m\sum_{i=1}^m x_i'.
\]
If there exists a $\xi \in \RV^{(n)}$ such that
$\rv(x_i-x_i') = \xi$ for all $1\leqslant i\leqslant n$, then we also have $\rv(x-x')=\xi$. 
\end{lem}

\begin{proof} 
We show that $v(x_1-x_1'-(x-x'))<v(x_1-x_1'))$: 
\begin{align*}
    v(x_1-x_1' - (x-x'))     & = v(x_1-x_1' - 1/m\sum_{i=1}^m (x_i-x_i')) & \\ 
                            & = v(m(x_1-x_1') - \sum_{i=1}^m (x_i-x_i')) & \\
                            & = v(\sum_{i=1}^m ((x_1-x_1')-(x_i-x_i'))) & \\
                            & \leqslant \max_i v((x_1-x_1')-(x_i-x_i')) & \text{ (by the ultrametric inequality) } \\
                            & < v(x_1-x_1') & \text{ (since $\rv(x_1-x_1')=\rv(x_i-x_i')$).} 
 \end{align*}
\end{proof}

\subsection{Model-theoretic setting}
\label{sec:modth}

We now fix a language on our valued field $K$.
The language we are mainly interested in is the pure valued field language:

\begin{defn}\label{defn:LHen}
Let~$\Lval$ be the language consisting of the ring language and a predicate for the valuation ring. 
\end{defn}

However, most of this article also works in various ``well-behaved'' expansions of $\Lx_0 \supseteq \Lval$, namely under the hensel minimality assumtion introduced in \cite{clu-hal-rid}. More precisely, for the entire paper, we fix the following.

\begin{defn}\label{defn:L0}
Let~$\Lx_0$ be an expansion of the valued field language $\Lval$ on $K$ such that the $\Lx_0$-theory of $K$ is $1$-h-minimal in the sense of
\cite[Definition~2.3.3]{clu-hal-rid}.
\end{defn}

Here are some examples of languages on $K$ satisfying this condition (by \cite[Section~6]{clu-hal-rid}):
\begin{enumerate}
    \item $\Lval$ itself;
    \item $T$-convex languages obtained from power-bounded o-minimal theories as in \cite{DL.Tcon1}; see Section~\ref{sec:lipschitz} for details;
    \item languages obtained from analytic structures on $K$ in the sense of \cite{CLip}; see Section~\ref{sec:analytic} for details.
\end{enumerate}

To simplify things, we add some imaginary sorts to the language:

\begin{defn}\label{defn:L}
Let $\Lx$ be the following expansion of $\Lx_0$: Firstly, add the sort $\RV$ together with the canonical quotient map $\rv\colon K \to \RV$. Secondly add each sort of the form 
$X/\mathord{\sim}$, where
$X \subseteq \RV^k$ is $\emptyset$-definable and~$\sim$ is a $\emptyset$-definable equivalence relation on~$X$; also add the canonical quotient maps $X \twoheadrightarrow X/\mathord{\sim}$ to the language.
\end{defn}

Note that $\bar K$, $\Gamma$, and $\RV^{(n)}$ for all $n$ are sorts of $\Lx$.

\begin{nota}
We denote the collection of all newly added sorts of $\Lx$ (i.e., all sorts except $K$) by $\RV\eq$.
As in \cite{halupczok2014a}, we will often write ``$\chi \colon K^n \to \RV\eq$'' to we mean a map whose target is an arbitrary auxiliary sort in $\RV\eq$, and similarly for definable subsets $Q \subseteq \RV\eq$.
\end{nota}

\begin{conv}
If $\Lx'$ is some language on $K$, then by $\Lx'$-definable, we mean definable in $\Lx'$ without parameters; if we want to allow parameters, we instead write $\Lx'(K)$-definable.
In the entire paper, if we just write ``definable'', we mean $\Lx(K)$-definable, where for the moment, $\Lx$ is as in Definition~\ref{defn:L}, but in some sections, we will put more restrictions on $\Lx$.
\end{conv}

The $1$-h-minimality assumption in particular implies the following:
\begin{enumerate}
    \item (Stable emeddedness) The sort~$\RV$ is stably embedded, i.e., every definable subset of~$\RV^{n}$ is already definable with parameters from~$\RV$.  
    \item (Definably spherically complete) For every definable family $(B_q)_{q \in Q}$ of balls $B_q \subseteq K$ which form a chain with respect to inclusion, the intersection $\bigcap_{q \in Q} B_q$ is non-empty.
    \item (Dimension of definable sets) The topological dimension is well-behaved on definable subsets of $K^n$ (for every $n>0$). Recall that for a subset $X\subseteq K^n$, the topological dimension of $X$, denoted $\dim(X)$, is the largest integer $0\leqslant k\leqslant n$ for which there is a coordinate projection $\pi\colon K^n\to K^k$ such that $\pi(X)$ has non-empty interior. Here, by ``well-behaved'' we mean that it satisfies standard properties of dimension as listed in \cite[Proposition 5.3.4]{clu-hal-rid}. In particular, the valued field sort eliminates the $\exists^\infty$-quanfier.  
\end{enumerate}

For a definable set $X$, we let~$\code{X}$ denote any element (possibly imaginary) fixed by precisely those automorphisms (of a sufficiently saturated and sufficiently homogeneous elementary extension of~$K$) fixing~$X$ globally. Although~$\code{X}$ is not uniquely defined, any two such elements are interdefinable. For a definable family~$(X_q)_{q \in Q}$, one may find and $\Lx\eq$-definable function $g_X\colon Q\to K\eq$ such that $g_X(q)$ is fixed precisely by the automorphisms fixing~$X_q$ globally. Abusing of notation we let~$\code{X_q}$ denote $g_X(q)$, noting that $\code{X_q}$ does not only depend on the set $X_q$, but is uniformly defined on all fibers $X_q$. 
A consequence of condition (1) above is that if $X_q$ lives in $\RV\eq$, then~$\code{X_q} = g_X(q)$ can be chosen in~$\RV\eq$ as well.

\begin{lem}[Banach fixed point for product of balls]\label{lem:banach} Let $E$ be a product of balls in $K^n$. Let $f\colon E\to E$ be a definable contracting function (i.e., for any $x_1$, $x_2 \in E$ with $x_1\neq x_2$, $v(f(x_1) - f(x_2))<v(x_1-x_2)$. Then $f$ has (exactly) one fixed point.   
\end{lem}

\begin{proof}
We proceed almost exactly as in \cite[Lemma 2.32]{halupczok2014a}. Suppose that $f(x) \ne x$ for all $x\in E$. For $x\in E$ define $B_x \coloneqq B(x, \leqslant v(x - f(x))) \cap E$. As in \cite[Lemma 2.32]{halupczok2014a}, one obtains that those $B_x$ form a chain under inclusion. Applying \cite[Lemma 2.31]{halupczok2014a} to the projection of the chain to each coordinate, one finds that the intersection of the entire chain is non-empty.
\end{proof}

\section{t-Stratifications revisited}\label{sec:tsts}

In this section we recall the notion of t-stratification introduced in \cite{halupczok2014a}. Its definition and main properties are given in Subsection \ref{sec:defprop}. We will also revisit a variant of the origianl definition in Subsection \ref{sec:translaters}. In the entire section, we work in an arbitrary $1$-h-minimal language (more precisely, in a language $\Lx$ as in Definition~\ref{defn:L}).

In this section, unless otherwise stated, we
fix a natural number $n \geqslant 1$ and an open ball
$\mathbf{B} \subseteq K^n$, which will serve as the ambient space of the t-stratifications. Recall that we allow $\mathbf{B}=K^n$ as an open ball.

\subsection{Definition and main properties}\label{sec:defprop}
Informally, a $t$-stratification of $\mathbf{B}$ is a partition of $\mathbf{B}$ into definable subsets $S_0,\ldots, S_n$, called the \emph{strata}, such that (if non-empty) the strata $S_d$ has dimension $d$ and on a neighbourhood of a point of $S_d$ things are ``roughly translation invariant in $d$-dimensions''. Let us now give precise definitions of all these concepts.

\begin{defn}\label{def:straighened} Let $\chi\colon \mathbf{B} \to \RV\eq$ be a definable function, $B\subseteq \mathbf{B}$ be a ball, $\varphi\colon B\to B$ be a risometry and $U\subseteq K^n$ be a $K$-vector space. 
\begin{enumerate}
    \item The map $\chi$ is \emph{$U$-translation invariant on $B$} if for all $x,y\in B$ satisfying $x-y\in U$, we have $\chi(x)=\chi(y)$ (or, in other words, if $\chi$ factors over the canonical map $B \to B/U$). 
    \item
    A definable risometry $\varphi\colon B\to B$
    \emph{$U$-trivializes $\chi$ (on $B$)} if $\chi\circ\varphi\colon B \to \RV\eq$ is $U$-translation invariant.
    Given a sub-vector space $\bar U \subseteq \bar K^n$, we also say that \emph{$\varphi$ $\bar U$-trivializes $\chi$}, if it $U$-trivializes $\chi$ for some lift $U$ of $\bar U$.
    \item 
    We say that $\chi$ is \emph{$\bar U$-trivial} on $B$, for some sub-vector space $\bar U \subseteq \bar K^n$, if there exists a $\varphi\colon B \to B$ $\bar U$-trivializing $\chi$ on $B$.
    We write $\tsp_B(\chi)$ for the maximal sub-vector space of $\bar K^n$ such that $\chi$ is $\tsp_B(\chi)$-trivial. Note that such a maximal sub-vector space exists by \cite[Lemma~3.3]{halupczok2014a}.
    \item Given $d\leqslant n$, we say that $\chi$ is \emph{$d$-trivial on $B$} if $\dim( \tsp_B(\chi))\geqslant d$, i.e., if $\chi$ is $\bar U$-trivial for some sub-vector space $\bar U \subseteq \bar K^n$ of dimension $d$.
\item
Given a tuple $(X_1, \dots, X_{m'}, \chi_1,\ldots, \chi_{m})$ of sets $X_i \subseteq \mathbf B$ and maps $\chi_i\colon \mathbf{B}\to \RV\eq$, we say that $(X_1, \dots, X_{m'},\chi_1,\ldots, \chi_m)$ is $U$-translation invariant (resp.\ $U$-trivialized by $\varphi$, $\bar{U}$-trivialized by $\varphi$, $\bar{U}$-trivial on $B$, or $d$-trivial on $B$ if the map
$x \mapsto (\chi_{X_1}(x), \dots, \chi_{X_{m'}}(x), \chi_1(x), \dots, \chi_m(x))$ is, where $\chi_{X_i}\colon \mathbf B \to \RV$ is the characteristic function of $X_i$ (sending $x$ to $1\in \RV$ if $x\in X$ and to $0\in \RV$ if $x\notin X$).
We define $\tsp_B(X_1, \dots, X_{m'}, \chi_1,\ldots, \chi_{m})$ accordingly.
\item
Let $\pi_U\colon K^n \to K^n/U$ be the quotient map and let $B/U$ be the image of $B$ under $\pi_U$.
Given $b \in B/U$, we define
\[
\Arc_b(U,\varphi)\coloneqq \{x\in B : \pi_U(\varphi^{-1}(x))=b\}. 
\]
We call such a set a $(U,\varphi)$-arc.  
\end{enumerate}
\end{defn}

Note that the $(U,\varphi)$-arcs form a partition of $B$. In addition, a map $\chi$ is $U$-trivialized by $\varphi$ if it is constant on each $(U,\varphi)$-arc, and a set $X$ is $U$-trivialized by $\varphi$ if each $(U,\varphi)$-arc is either contained in $X$ or disjoint from $X$.

\begin{rem}
While it is true that a tuple $(\chi_1, \dots, \chi_m)$ is $U$-trivialized by $\varphi$ if and only if each $\chi_i$ is, note that each individual $\chi_i$ being $d$-trivial on $B$ is a weaker condition than the entire tuple being $d$-trivial on $B$, since different maps $\varphi$ might be needed for different $\chi_i$.
\end{rem}

\begin{exam}\label{exam:square} Let $K$ be a real closed valued field and $X\subseteq K^2$ be the graph of $x^2=y$. Let $a\in K$ be a positive element with $v(a)>1$.
\begin{enumerate}
    \item 
Let $B\subseteq K^2$ be the ball $B((a,a^2),< v(a))$ and consider the function $\varphi\colon B\to B$ sending $(x,y)\in B$ to 
$(x-a+\sqrt{y}, y)$. We let the reader verify that $\varphi$ is a risometry on $B$. Moreover, letting $\bar{U}=\{0\}\times \bar{K}$ and $U=0\times K$ be a lift of $\bar{U}$, we have that $X$ is $U$-trivialized by $\varphi$.
\item Now let $B' \supset B$ be the ball $B((a, a^2), <v(a^2)) = B((0, a^2), <v(a^2))$. We can extend the above $\varphi$ to a risometry $U$-trivializing $X$ on $B'$ using a case distinction on $\rv(x)$: In the case $\rv(x) = \rv(a)$, use the same definition as in (1), in the case $\rv(x) = \rv(-a)$ do something similar for the negative branch, and otherwise, just let $\varphi$ be the identity  (see Figure \ref{fig:riso}).
\end{enumerate}

\begin{figure}
    \centering
    \includegraphics[scale =0.8]{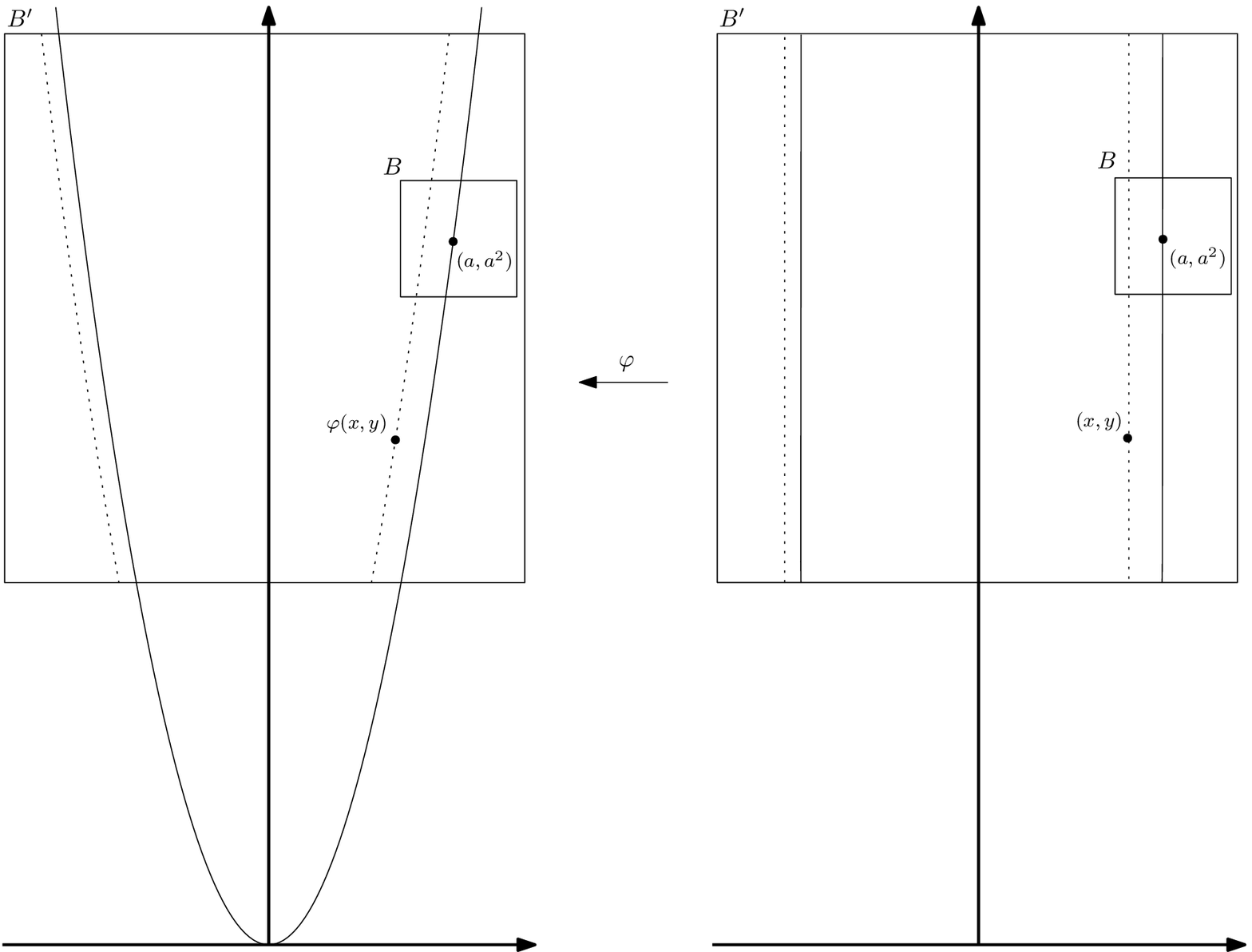}
    \caption{}
    \label{fig:riso}
\end{figure}
\end{exam}

\begin{defn}\label{def:exhibition-respect} Let $U\subseteq K^n$ be a $K$-vector space and suppose $\pi\colon K^n\to K^n$ is an exhibition of $\bar{U} \coloneqq \res(U)$. We say that a map \emph{$\varphi\colon B \to B$ respects $\pi$} (for $B \subseteq K^n$) if $\pi\circ\varphi=\pi$.
\end{defn}

\begin{lem}\label{lem:respect-fibers} Let $B\subseteq \mathbf{B}$ be a ball, $U\subseteq K^n$ be a $K$-vector space, $\pi\colon K^n \to K^d$ an exhibition of $\bar U := \res(U)$, and $\varphi\colon B \to B$ a definable risometry. Then there exists a definable risometry $\varphi'\colon B \to B$ such that the $(U,\varphi')$-arcs are the same as the $(U,\varphi)$-arcs, and moreover, 
$\varphi'$ respects $\pi$. 
In particular, given a definable map $\chi\colon B\to \RV\eq$,
$\varphi$ $U$-trivializes $\chi$ if and only if $\varphi'$ does.
\end{lem}

\begin{proof}
This is more or less 
\cite[Part 1 of Lemma 3.6]{halupczok2014a}. More precisely, the statement of that lemma of \cite{halupczok2014a} is weaker, but the proof in reality yields the above, stronger statement. 
\end{proof}

\begin{fact}\label{fact:exhibition_val}
Using Lemma~\ref{lem:respect-fibers}, one easily obtains that in that situation, for any arc $X \coloneqq \Arc_b(U, \varphi)$,
$\pi$ induces a bijection $\pi_{|X}\colon X \to \pi(B)$. Moreover, the assumption that $\pi$ is an exhibition of $\bar U$ implies that for $x, x' \in X$, we have $v(x - x') = v(\pi(x) - \pi(x'))$. \qed
\end{fact}

\begin{rem}\label{rem:lift-choice}
Definition \ref{def:straighened} is independent of the choice of the lift $U$ of $\bar U$ in the following sense. Given $\varphi$, $\bar{U}$ and $U$ as in Definition \ref{def:straighened}, $U'$ another lift of $\bar U$, choose an exhibition $\pi$ of $\bar U$ and let $M\in \GL_n(\valring)$ be the unique matrix (which is a risometry) sending $U$ to $U'$ and satisfying $\pi \circ M = \pi$ (which exists by Fact~\ref{fact:proj-unique}). Let $\psi'\colon B \to B$ be a risometry obtained by composing $M$ with a suitable translation (sending $M(B)$ to $B$), and define $\varphi' \coloneqq \psi \circ \varphi$. Then, the set of $(U,\varphi)$-arcs is equal to the set of $(U',\varphi')$-arcs. In addition, if $\varphi$ respects $\pi$, then so does $\varphi'$.
\end{rem}

\begin{rem}\label{rem:fibers_riso} Let $B, B'\subseteq K^n$ be two balls and $\chi\colon B\to \RV\eq$, $\chi'\colon B'\to \RV\eq$ be two definable maps.
Suppose that $\chi$ is $U$-trivialized on $B$ (by some risometry $\psi\colon B \to B$) for some sub-vector space $U \subseteq K^n$ and let $\pi\colon K^n \to K^{\dim(U)}$ be an exhibition of $\bar{U}\coloneqq\res(U)$ Assume that $F \coloneqq \pi^{-1}(t) \cap B$ and $F' \coloneqq \pi^{-1}(t') \cap B'$ are $\pi$-fibers, for some $t \in \pi(B), t' \in \pi(B')$. Then the following are equivalent:
\begin{enumerate}
    \item There exists a risometry $\varphi\colon B\to B'$
    such that $\chi = \chi' \circ \varphi$
    \item
    $\chi'$ is $U$-trivialized on $B$ (by some risometry $\psi'\colon B' \to B'$) and
    there exists a risometry $\theta\colon F \to F'$ 
    such that $\chi_{|F} = \chi'_{|F'} \circ \psi$.
\end{enumerate}
This implication ``(1) $\Rightarrow$ (2)'' is essentially Part (2) of \cite[Lemma 3.6]{halupczok2014a}. For the other direction, note that by Lemma~\ref{lem:respect-fibers} we may assume $\psi$ and $\psi'$ to preserve $\pi$-fibers, so that the risometry $\theta\colon F \to F'$
induces a risometry
$\theta'\colon \psi^{-1}(F) \to (\psi')^{-1}(F')$
sending $(\chi\circ \psi)_{|\psi^{-1}(F)}$ to $(\chi'\circ \psi')_{|(\psi')^{-1}(F)}$; by $U$-triviality we may extend $\theta'$ to a risometry $\theta'\colon B \to B'$ sending $\chi\circ \psi$ to $\chi'\circ \psi'$.
Now set $\varphi \coloneqq \psi'\circ\theta'\circ\psi\1$.
\end{rem}

\begin{nota}\label{nota:strat}
Given a partition $\mathcal{S}=\{S_0,\ldots,S_n\}$ of $\mathbf B$ into definable sets $S_i$ satisfying $\dim(S_i)\leqslant i$ for each $i\in\{0,\ldots,n\}$, we set 
\[
S_{<m}\coloneqq \bigcup_{i=0}^{m-1} S_i, 
\]
and $S_{<0} \coloneqq \emptyset$. We call the partition $A$-definable, for some parameter set $A$, if each $S_i$ is $A$-definable.
\end{nota}

We are ready to define t-stratifications. 

\begin{defn}\label{def:-t-strat}
A partition $\mathcal{S}=(S_i)_{0 \leqslant i\leqslant n}$ of $\mathbf{B}$ is called a \emph{t-stratification} if each $S_i$ is definable of dimension at most $i$, and for each $d\leqslant n$ and each ball $B \subseteq S_{\geqslant d}$ (open or closed) $(S_0,\ldots, S_n)$ is $d$-trivial on $B$.

Given a definable map $\chi\colon \mathbf{B}\to \RV\eq$, we say that a t-stratification $\mathcal{S}$ \emph{reflects $\chi$} if for each $d\leqslant n$ and each ball $B \subseteq S_{\geqslant d}$ (open or closed)
$(S_0,\ldots, S_n, \chi)$ is $d$-trivial on $B$.
\end{defn}

\begin{rem}\label{rem:refl}
If $\cS$ reflects $\chi$ and $B \subseteq S_{\geqslant d}$, then by Lemma~\ref{lem:refl-charac} below, we do not only have that $(S_0,\ldots, S_n, \chi)$ is $d$-trivial on $B$, but every definable risometry $U$-trivializing $\cS$ on $B$ also $U$-trivializes $\chi$.
This will be useful when later in this paper, we will introduce stronger notions of stratifications by imposing additional conditions (C) on $\varphi$: If $\cS$ can be trivialized by $\varphi$ satisfying (C), then automatically, for any $\chi$ reflected by $\cS$, also $(\cS, \chi)$ can be trivilized by some $\varphi$ satisfying (C) (namely, the same one).
\end{rem}

Recall that t-stratifications exist in our context (by \cite[Theorem 5.5.3]{clu-hal-rid}). More precisely, by 
\cite[Lemma 5.5.4]{clu-hal-rid}, \cite[Hypothesis 2.21]{halupczok2014a} is satisfied, so all results from \cite{halupczok2014a} apply. In particular, we have:

\begin{thm}[{\cite[Theorem 4.12]{halupczok2014a}}]
Let $\chi \colon \mathbf{B} \to \RV\eq$ be $\emptyset$-definable.
Then there exists a $\emptyset$-definable t-stratification $(S_i)_{i \leqslant n}$ reflecting $\chi$.
\end{thm}

Note that even when we require a t-stratification $\cS$ to be $\emptyset$-definable, the risometries $\varphi\colon B \to B$ trivializing $\cS$ on balls $B$ are usually not even $\code{B}$-definable; see Section~\ref{sec:translaters} for a remedy for this.

A crucial (but not difficult to prove) property of t-stratifications is that they restrict well to fibers of exhibitions:

\begin{fact}[{\cite[Lemma 3.16]{halupczok2014a}}]\label{fact:induced-tst}
Let $\cS$ be a t-stratification of $\mathbf{B}$,  $B \subseteq \mathbf{B}$ be a ball and $\pi\colon B\to K^d$ be an exhibition of $\bar{U}=\tsp_B(\cS)$. Let $\mathbf{F}\coloneqq \pi\1(z)\cap B$ be a $\pi$-fiber for some $z\in \pi(B)$. Set $\cS'=(S_i')_{0\leqslant i\leqslant n-d}$ where $S_i' \coloneqq S_{i+d} \cap \mathbf{F}$. Then 
\begin{enumerate}
\item  $\cS'$ is a t-stratification of $\mathbf{F}$;
\item the restriction of $\rho_{\cS}$ to $\mathbf{F}$ is $\rho_{\cS'}$ (\cite[Lemma 4.21]{halupczok2014a});
\item if $\cS$ reflects a definable map $\chi\colon \mathbf{B}\to \RV\eq$, then $\cS'$ reflects $\chi_{|\mathbf F}$. 
\end{enumerate}
\end{fact}

The above fact will mostly be applied to maximal balls $B \subseteq S_{\geqslant d}$ for some $d$; such balls indeed exist:

\begin{fact}[{\cite[Lemma 3.17~(1)]{halupczok2014a}}]\label{fact:max-B}
Given a t-stratification $\cS$ of $\mathbf{B}$ and a point $x \in S_{\geqslant d}$ (for some $d$), there exists a maximal open ball $B \subseteq S_{\geqslant d}$ containing $x$.
\end{fact}

\subsection{Rainbows}

In this section, we recall the notion of the rainbow associated to a t-stratification $\cS$ of $\mathbf B$. The definition looks somewhat technical, but it yields a natural characterisation of which maps are reflected by $\cS$ (Lemma~\ref{lem:refl-charac}) and also its fibers have a natural characterization (Lemma~\ref{lem:rainbow-charac}).

\begin{defn}\label{def:rainbow} Let $\mathcal{S}=(S_i)_i$ be a t-stratification of $\mathbf{B}$. The \emph{rainbow of $\mathcal{S}$} is the map $\rho_\mathcal{S}\colon \mathbf{B}\to \RV\eq$ given by 
\[
x\mapsto (\code{\rv(x-S_0)},\code{\rv(x-S_{1})}, \ldots, \code{\rv(x-S_n))})
\] 
where $\rv(x-S_i)\coloneqq\{\rv(x-y): y \in S_i\}$. By a \emph{rainbow fiber}, we mean a fiber of the rainbow $\rho_\mathcal S$.
\end{defn}

The rainbow $\rho_\mathcal{S}$ is determined up to a choice of the elements in $\RV\eq$ we choose as codes of the sets $\rv(x - S_d)$. (But recall that we assume those elements to depend definably on $x$.) Abusing of terminology, we will fix some codomain for $\rho_\mathcal{S}$ and speak about \emph{the} rainbow of $\mathcal{S}$.

\begin{rem}\label{rem:triv-rain}
It is not difficult to check that for any ball $B \subseteq \mathbf B$, any definable risometry $\varphi\colon B \to B$ and any sub-vector space $U \subseteq K^n$ the following are equivalent:
$\varphi$ $U$-trivializes $\mathcal{S}$ on $B$; 
$\varphi$ $U$-trivializes $\rho_\mathcal{S}$ on $B$; $\varphi$ $U$-trivializes $(\mathcal{S}, \rho_\mathcal{S})$ on $B$. This in particular implies that $\tsp_B(\mathcal S) = \tsp_B(\rho_{\mathcal S})$, and that
another t-stratification $\mathcal{S}'$ reflects $\mathcal{S}$ if and only if it reflects $\rho_\mathcal{S}$.
\end{rem}

\begin{lem}\label{lem:refl-charac}
For a t-stratification $\cS$ and a definable map $\chi\colon \mathbf B \to \RV\eq$, the following are equivalent:
\begin{enumerate}
    \item $\cS$ reflects $\chi$.
    \item $\chi$ factors over the rainbow, i.e., $\chi = h \circ \rho_{\cS}$ for some (necessarily definable) map $h$.
    \item For every definable risometry
    $\varphi\colon \mathbf B \to \mathbf B$ fixing each set $S_i$ setwise, we have $\chi = \chi \circ \varphi$.
    \item For every ball $B \subseteq \mathbf B$ and every sub-vector space $U \subseteq K^n$, if $\varphi\colon B \to B$ is a definable risometry $U$-trivializing $\cS$, then $\varphi$ also $U$-trivializes $(\cS, \chi)$.
\end{enumerate}
\end{lem}

\begin{proof}
The equivalence of (1), (2) and (3) is \cite[Proposition~4.17]{halupczok2014a}, and the implication (4) $\Rightarrow$ (1) is trivial.

To finish, we prove (2) $\Rightarrow$ (4):
Suppose that $\varphi\colon B \to B$ $U$-trivializes $\cS$. Then it $U$-trivializes $\rho_{\cS}$ (by Remark~\ref{rem:triv-rain}), so it also trivializes $\chi = h \circ \rho_{\cS}$.
\end{proof}

The following fact describes rainbow fibers rather precisely in a recursive way; it is very useful for inductive proofs:

\begin{fact}[{By \cite[Lemma 4.21]{halupczok2014a}}]\label{fact:rfiber_induction}
Let $\cS$ be a t-stratification on $\mathbf{B}$ and $C$ be a rainbow fiber. Then either $C$ consists of a single element of $S_0$ or it is entirely contained in a ball $B \subseteq S_{\geqslant 1}$.
In the second case, we moreover have that, for any exhibition $\pi$ of $\tsp_B(\cS)$, for any $z \in \pi(B)$, and for $\mathbf{F} = \pi\1(z) \cap B$ and $\cS'$ as in Fact~\ref{fact:induced-tst}, 
$C \cap \mathbf F$ is a $\rho_{\cS'}$-fiber.
\end{fact}

\begin{rem}
In the case $C \subseteq B \subseteq S_{\geqslant 1}$ of the above fact, $B$ is necessarily maximal among the balls contained in $S_{\geqslant 1}$, since otherwise, we would get a contradiction to $C$ being $1$-trivial on a ball $B' \subseteq S_{\geqslant 1}$ properly containing $B$.
\end{rem}

The next fact contains some more properties of the rainbow $\rho_\mathcal{S}$ (see \cite[Lemma 4.22]{halupczok2014a})

\begin{fact}\label{fact:rainbow} Let $\mathcal{S}=(S_i)_i$ be a t-stratification of $\mathbf{B}$ and $C\subseteq S_d$ be a $\rho_{\mathcal{S}}$-fiber. Then 
\begin{enumerate}
    \item $\dim(\affdir(C))=d$;
    \item If $B$ is a ball such that $B\subseteq S_{\geqslant d}$ and $B\cap C\neq\emptyset$, then $\affdir(C) = \tsp_B(\rho_{\mathcal{S}})$.
\end{enumerate}
\end{fact}

The following lemma gives yet another description of rainbow fibers.

\begin{lem}\label{lem:liso} Let $\mathcal{S}=(S_i)_i$ be a t-stratification of $\mathbf{B}$. For every $\rho_{\mathcal{S}}$-fiber $C$, there exists a map $\psi\colon C \to R \subseteq K^n$, where $R$ is a product of open balls and singletons and such that $\psi$ is a composition of finitely definable maps $\psi_i$ between subsets of $K^n$
each of which is either a risometry or lies in $\GL_n(\valring)$.
\end{lem}

\begin{rem}\label{rem:composition_riso} One can easily check that any such finite composition $\psi$ can be written as the composition of a single risometry and a single map from $\GL_n(\valring)$. (But we will not use this.)
\end{rem}

\begin{proof}[Proof of Lemma \ref{lem:liso}] By Fact \ref{fact:rfiber_induction}, either $C$ consists of a single element in $S_0$ or it is entirely contained in a maximal ball $B\subseteq S_{\geqslant 1}$. If the former holds, there is nothing to prove, so suppose $C \subseteq B$ and let $\pi\colon K^n\to K^d$ be an exhibition of $\tsp_B(\mathcal{S})$. Let $\varphi\colon B\to B$ be a definably risometry
$U$-trivializing $\mathcal{S}$ (and hence also $\rho_{\mathcal{S}}$) on $B$ for some $d$-dimensional $U \subseteq K^n$. Using Lemma~\ref{lem:respect-fibers}, we further assume that $\varphi$ respects $\pi$.

Fix some $q\in \pi(B)$ and consider the corresponding $\pi$-fiber $F_q \coloneqq \pi\1(q) \cap B$. By induction, we have a map $\tilde\psi\colon F_q \to K^{n-d}$ which is a desired kind of composition such that $\tilde\psi(\varphi^{-1}(C) \cap F_q)$ is equal to a product $\tilde R \subseteq K^{n-d}$ of open balls and singletons. Setting $R\coloneqq \tilde{R} \times \pi(B) \subseteq K^n$, the map $\tilde\psi$ induces a map 
\[
\tilde\psi \times \id\colon
(\varphi^{-1}(C) \cap F_q) \times \pi(B)\to R 
\]
of the desired form. Composing this map with a map from $\GL_n(\valring)$, we obtain a map $\varphi^{-1}(C) \to R$, which composed with $\varphi\1$ gives us the desired map $\psi
\colon C \to R$.
\end{proof}

\begin{lem} \label{lem:surjectivity} Let $\mathcal{S}=(S_i)_i$ be a t-stratification of $\mathbf{B}$ and $C$ be a $\rho_{\mathcal{S}}$-fiber. Let $\varphi\colon C\to X$ be a definable risometry, for some $X \subseteq C$.
Then $X = C$. Moreover, if $\pi$ is an exhibition of $\affdir C$, then the same applies for $C$ replaced by $\pi(C)$, namely, if $X\subseteq \pi(C)$ and $\psi\colon \pi(C)\to X$ is a risometry, then $\pi(C)=X$.
\end{lem}

\begin{proof}
Fistly, note that if $\psi\colon C \to C'$ is either a risometry or lies in $\GL_n(\valring)$,
then $\varphi' \coloneqq \psi\1 \circ \varphi \circ \psi$
is still a risometry. By applying this repeatedly and using Lemma~\ref{lem:liso}, we may assume without loss that $C$ is a product of balls. 

Now proceed exactly as in \cite[Lemma 2.33]{halupczok2014a}, namely: To prove that some given $x_0\in C$ lies in the image of $\varphi$, define $f\colon C \to C$ by $f(x) \coloneqq x + x_0 - \varphi(x)$. This map is contracting and therefore, by the Banach Fixed Point Theorem (Lemma \ref{lem:banach}), $f$ admits a fixed point, which is a preimage of $x_0$ under $\varphi$.  

The last assertion follows by noting that the risometry $\psi\colon \pi(C) \to X$ lifts to a risometry $\psi'\colon C \to \pi^{-1}(X) \cap C$ (by Lemma~\ref{lem:riso-proj}) and by applying the first part to $\psi'$.
\end{proof}

We finish this section by providing the following alternative characterization of rainbow fibers of a t-stratification $\cS$ as orbits under the group of definable risometries fixing $\cS$. 

\begin{lem}\label{lem:rainbow-charac} Let $\cS$ be a t-stratification. The following are equivalent:
\begin{enumerate}
    \item $x_1, x_2\in \mathbf{B}$ lie in the same rainbow fiber 
    \item there is a definable risometry from $\mathbf{B}$ to itself sending $x_1$ to $x_2$ and fixing each $S_i$ setwise.  
\end{enumerate}
\end{lem}

\begin{proof}
The implication (2) $\Rightarrow$ (1) is easy: Let $\varphi\colon \mathbf{B}\to\mathbf{B}$ be a definable risometry sending $x_1$ to $x_2$ and fixing $\cS$ setwise. Then, for every $0\leqslant d\leqslant n$, $\rv(x_1-S_d)=\rv(\varphi(x_1)-\varphi(S_d))=\rv(x_2-S_d)$.

We now prove (1) $\Rightarrow$ (2).
Let $C$ be the rainbow fiber containing $x_1, x_2$. We proceed by induction on $\dim(C)$. If 
$\dim(C) = 0$, then $C$ is a singleton and the result is trivial (take $\varphi$ to be the identity). Otherwise, by Fact \ref{fact:rfiber_induction}, $C$ is entirely contained in a maximal ball $B\subseteq S_{\geqslant 1}$.

Set $\bar{U} \coloneqq \tsp_B(\cS)$, let $U \subseteq K^n$ be a lift, $\pi$ be an exhibition of $\bar{U}$, and $\varphi \colon B \to B$ be a risometry $U$-trivializing $\rho_{\cS}$ on $B$. Let $\mathbf{F}_i \subseteq B$ be the $\pi$-fiber containing $\varphi^{-1}(x_i)$ (for $i=1,2$) and let $\psi\colon B \to B$ be the translation in direction $U$ sending $\mathbf{F}_1$ to $\mathbf{F}_2$. Set $x_1' \coloneqq \psi(\varphi^{-1}(x_1)) \in \mathbf F_2$ and let $\cS_2$ be the induced stratification on $\mathbf{F}_2$. Since $x_1'$ and $\varphi\1(x_2)$ lie in the same rainbow fiber with respect to $\cS_2$, by induction, there is a definable risometry $\psi'\colon \mathbf{F}_2\to \mathbf{F}_2$ sending $x'_1$ to $\varphi^{-1}(x_2)$ and fixing $\cS_2$ setwise. Extend $\psi'$ to $B$ ``trivially along $U$'', i.e., write any $y \in B$ as $y' + v$ for some $y' \in \mathbf{F}_2$ and $v \in U$ and set $\psi'(y) \coloneqq \psi(y') + v$. It follows that $\psi'\circ\psi\colon B \to B$ is a risometry preserving $\varphi^{-1}(\cS_{|B})$ setwise and sending $\varphi^{-1}(x_1)$ to $\varphi^{-1}(x_2)$. Thus $\varphi \circ \psi'\circ\psi \circ \varphi^{-1}\colon B \to B$ is a definable risometry preserving $\cS_{|B}$ setwise and sending $x_1$ to $x_2$. Extending such a risometry by the idendity outside of $B$ completes the construction of the desired risoemtry $\mathbf B \to \mathbf B$.
\end{proof}

\subsection{Translaters}\label{sec:translaters}

We will now recall the notion of translater from \cite{halupczok2014a}. Given a definable map $\chi\colon \mathbf B \to \RV\eq$, a ball $B \subseteq \mathbf B$ and a sub-vector space $\bar U \subseteq \bar K^n$, translaters provide a more canonical way to characterize $\bar U$-triviality of $\chi$ on $B$ than via definable risometries $\varphi\colon B \to B$ trivializing $\chi$. Indeed, the definition of such a $\varphi$ often inherently needs additional parameters, whereas for translaters, at least in the case $\bar U = \tsp_B(\chi)$, one can deduce from \cite[Proposition~3.19~(3')]{halupczok2014a} that if $\chi$ and $B$ are $A$-definable for some parameter set $A$, also a translater witnessing $\bar U$-triviality can be chosen $A$-definable.
Since we will not use this canonicity result explicitly, we do not give the details of this argument, but implicitly, this canonicity plays a key role in the proof of Proposition~\ref{prop:an-t-to-arc-t}.

Note that the notion of translater also has a drawback (compared to trivializing maps): It depends on a choice of an exhibition $\pi\colon K^n\to K^d$ of $\bar{U}$, so \emph{a priori}, 
being $\bar U$-trivial in the sense of translaters would not even be a coordinate independent notion.

The definition of translater below slightly differs from \cite[Definition~3.8]{halupczok2014a}: Firstly, we define it more generally than just for ball $B$; and secondly, we left out a condition which is unnecessary (as we shall see below).

\begin{defn}\label{def:translater}
Let $E \subseteq \mathbf{B}$ be a definable set, $\chi\colon \mathbf{B}\to \RV\eq$ be a definable map, $\bar U\subseteq \bar K^n$ be a $\bar K$-vector space and $\pi\colon K^n\to K^d$ be an exhibition of $\bar{U}$. For $q\in\pi(E)$, let $E_q$ denote the fiber $E\cap \pi\1(q)$. A definable family of risometries $\alpha\family = (\alpha_{q,q'})_{q,q'\in \pi(E)}$ with $\alpha_{q,q'}\colon E_q\to  E_{q'}$ is called a \emph{$\bar{U}$-translater on $E$ reflecting $\chi$} if it has the following properties, where $q, q', q''$ run over $\pi(E)$:
\begin{enumerate}
\item $\chi_{|E_{q'}}\circ \alpha_{q,q'} = \chi_{|E_{q}}$;
\item  $\alpha_{q',q''}\circ \alpha_{q,q'} = \alpha_{q,q''}$;
\item $\dir(\alpha_{q,q'}(z)-z)\subseteq \bar{U}$.
\end{enumerate}
(If we do not specify that a translater should reflect a map $\chi$, then Condition (1) is dropped.)
An \emph{$\alpha\family$-arc} is a set of the form
$\{\alpha_{\pi(x),q}(x) \mid q \in \pi(E)\}$
for $x \in E$.
\end{defn}

Note that by (2), the arcs of $\alpha\family$ form a partition of $E$ and that Condition (1) could also be stated as: $\chi$ is constant on each arc.

Also note that the notion of $\bar U$-translater does make sense for $\bar U = 0$, but it becomes trivial, namely: it consists of a single map $\alpha\colon E \to E$ which is the identity.

The following lemma corresponds to \cite[Lemma 3.7]{halupczok2014a}.

\begin{lem}\label{lem:translater} Let $B \subseteq \mathbf B$ be a ball, let $\bar{U}\subseteq \bar K^n$ be a $\bar K$-vector space, let $U \subseteq K^n$ be a lift of $\bar U$, and let $\pi\colon K^n\to K^d$ be an exhibition of $\bar{U}$. Then for any partition of $B$, the following are equivalent:
\begin{enumerate}
    \item There exists a definably risometry $\varphi\colon B\to B$ such that the partition is exactly the partition of $B$ into $(U,\varphi)$-arcs.
    \item There exists a $\bar{U}$-translater $\alpha\family$ on $B$ such that the partition of $B$ is exactly the partition into the $\alpha\family$-arcs.
\end{enumerate}
In particular, a definable map $\chi\colon \mathbf B \to \RV\eq$ is $\bar U$-trivial on $B$ if and only if there exists a $\bar{U}$-translater on $B$ reflecting $\chi$. 
\end{lem}

\begin{proof}
The proof is essentially the same as the proof of \cite[Lemma 3.19]{halupczok2014a}, with two differences:

Firstly, note that \cite[Lemma 3.19]{halupczok2014a} only claims the ``in particular'' part of the above lemma, and does not make the statement about the partition into arcs. However, the proof of \cite[Lemma 3.19]{halupczok2014a} really yields the equality of the partitions.

And secondly, to get from a translater to a risometry,
the proof of \cite[Lemma 3.19]{halupczok2014a}) uses a fourth compatibility condition which is not part of Definition~\ref{def:translater}. This fourth condition states the following: 
\begin{enumerate}
    \item[(4)] for every $\delta \in \pi(B-B)$ (where $B-B$ is the translate of $B$ containing the origin), the map $B \to B, x \mapsto \alpha_{\pi(x),\pi(x)+\delta}(x)$ is a risometry. 
\end{enumerate}
We will show that (4) is actually unnecessary, i.e., that it follows from Definition \ref{def:translater} (1)--(3). Fix some $\delta\in\pi(B-B)$ and let $x_1,x_2\in B$ be given. Set $q_i=\pi(x_i)$ and $q_i'=q_i+\delta$. We need to show that 
\begin{equation*}\label{eq:rv}
\rv(x_1-x_2)=\rv(\alpha_{q_1,q_1'}(x_1)-\alpha_{q_2,q_2'}(x_2)).     
\end{equation*}
Condition (2) of Definition \ref{def:translater} implies that $\alpha_{q,q}$ is the identity on $B\cap\pi\1(q)$ for every $q\in \pi(B)$ (indeed, for $q,q'$, $\alpha_{q,q'}(\alpha_{q,q}(x))=\alpha_{q,q'}(x)$, and since $\alpha_{q,q'}$ is injective, $\alpha_{q,q}(x)=x$). In particular, if $\delta=0$, \eqref{eq:rv} holds trivially. So suppose $\delta\neq 0$.

After adding and substracting $\alpha_{q_2,q_1'}(x_2)-\alpha_{q_2,q_1}(x_2)$ to $\alpha_{q_1,q_1'}(x_1)-\alpha_{q_2,q_2'}(x_2) - (x_1-x_2)$, the ultrametric inequality, yields 
\[
v(\alpha_{q_1,q_1'}(x_1)-\alpha_{q_2,q_2'}(x_2) - (x_1-x_2))\leqslant \max
\begin{cases}
v(\alpha_{q_1,q_1'}(x_1)-\alpha_{q_2,q_1'}(x_2) - (x_1-\alpha_{q_2,q_1}(x_2)))&\\ v(\alpha_{q_2,q_1'}(x_2)-\alpha_{q_2,q_2'}(x_2) - (\alpha_{q_2,q_1}(x_2))-x_2))).& 
\end{cases}
\]
Hence, it suffices thus to show that each member of the above maximum is smaller than $v(x_1-x_2)$. For the first term, since $\alpha_{q_1,q_1'}$ is a risometry, we have
\begin{align*}
v(\alpha_{q_1,q_1'}(x_1)-\alpha_{q_2,q_1'}(x_2) - (x_1-\alpha_{q_2,q_1}(x_2))) & =     v(\alpha_{q_1,q_1'}(x_1)-\alpha_{q_1,q_1'}(\alpha_{q_2,q_1}(x_2)) - (x_1-\alpha_{q_2,q_1}(x_2))) \\
    & <v(x_1-\alpha_{q_2,q_1}(x_2))\leqslant v(x_1-x_2),  
\end{align*}
where the last inequality holds by the triangle inequality, and using $\pi(x_1)=\pi(\alpha_{q_2,q_1}(x_2))$ and
$v(\alpha_{q_2,q_1}(x_2) - x_2) = v(q_1 - q_2)$.

For the second term, set $x=\alpha_{q_2,q_1'}(x_2)-\alpha_{q_2,q_2'}(x_2)$ and 
$x'=\alpha_{q_2,q_1}(x_2)-x_2$. Condition (3) of Definition \ref{def:translater} implies that $\dir(x),\dir(x')\subseteq \bar{U}$. Moreover, $\pi(x)=\pi(x')=\delta$. Therefore, by Lemma 
\ref{lem:dir}, we obtain that $\rv(x)=\rv(x')$, which, unravelling definitions yields
\[
v(\alpha_{q_2,q_1'}(x_2)-\alpha_{q_2,q_2'}(x_2) - (\alpha_{q_2,q_1}(x_2))-x_2)))<v(\alpha_{q_2,q_1}(x_2))-x_2))\leqslant v(x_1-x_2). \qedhere
\]
\end{proof}

We finish this subsection by proving that each rainbow fiber $C$ has a natural tubular neighbourhood $\Tub_C$ on which there exists an $\mathrm{affdir}(C)$-translater reflecting the rainbow $\rho_\mathcal S$. More precisely, this is the content of Proposition~\ref{prop:EC-translater}; before that, we introduce the necessary notation and a more technical ingredient.

\begin{defn}\label{def:EC}
Let $\mathcal{S}=(S_i)_i$ be a t-stratification of~$\mathbf{B}$, and $\rho_\mathcal{S}$ be the rainbow of $\mathcal{S}$. Given any 
$\rho_\mathcal{S}$-fiber $C\subseteq S_{d}$
for $0 \leqslant d\leqslant n$, we set
\begin{equation*}
\lambda_C\coloneqq
\begin{cases}
v(C - S_{<d}) & \text{if } S_{<d} \ne \emptyset\\
\rado(\mathbf B) & \text{if } S_{<d} = \emptyset
\end{cases}
\end{equation*}
and
\begin{equation*}
\Tub_C=\{x\in K^n: v(x-C)<\lambda_C\}.
\end{equation*}
We call $\Tub_C$ the \emph{natural tubular neighbourhood of $C$}. 
\end{defn}

\begin{rem}
Note that $\lambda_C$ is a well-defined element of $\Gamma^\times$ (or equal to $\infty$, if $S_{<d} = \emptyset$ and $\mathbf B = K^n$). More precisely, if $S_{<d} \ne \emptyset$, then since $S_{< d}$ is closed, $v(x - S_{<d})$ is well-defined for every $x \in C$ and, by definition of $\rho_\mathcal{S}$, for any two points $x_1,x_2\in C$, $v(x_1 - S_{<d})=v(x_2 - S_{<d})$.
In other words, for every $x \in C$,
$\lambda_C$ is the maximal element of $\Gamma^\times \cup \{\infty\}$ satisfying $B(x,<\lambda_C) \subseteq S_{\geqslant d}$. (This is true even if $S_{<d} = \emptyset$.)
\end{rem}

\begin{rem}\label{rem:Tub_C}
Pick any exhibition $\pi\colon K^n \to K^d$ of $\mathrm{affdir}(C)$. Then the reader may easily verify that
\begin{equation*}
\Tub_C =\bigcup_{q\in \pi(C)} \pi\1(q)\cap B(c_q,<\lambda_C), 
\end{equation*}
where, for every $q\in\pi(C)$, $c_q$ is the unique element in $\pi\1(q)\cap C$ (uniqueness follows since $\pi$ is an exhibition of $\affdir(C)$). In particular, $\pi(\Tub_C) = \pi(C)$.
\end{rem}

The following technical lemma is only a step in the proof of Proposition~\ref{prop:EC-translater}.

\begin{lem}\label{lem:moving_translater}
Let $\mathcal{S}=(S_i)_i$ be a t-stratification of $\mathbf{B}$ and $C\subseteq S_{d}$ be a $\rho_\mathcal{S}$-fiber for some $d\leqslant n$. Let $\bar{U}=\mathrm{affdir}(C)$,  $\pi\colon K^n\to K^d$ be an exhibition of $\bar{U}$,
let $B \subseteq \mathbf{B}$ be any ball containing $C$ and let $\varphi\colon B\to B$ be any definable risometry. Suppose there exists a translater $(\beta_{q,q'})_{q,q' \in \pi(\varphi^{-1}(\Tub_C))}$ on $\varphi^{-1}(\Tub_C)$ reflecting $\rho_{\mathcal{S}} \circ \varphi$. Then, there also exists a translater $(\alpha_{q,q'})_{q,q' \in \pi(\Tub_C)}$ on $\Tub_C$ reflecting $\rho_{\mathcal{S}}$.
\end{lem}

\begin{proof}
Given $q,q' \in \pi(\Tub_C)$ and $x \in \pi^{-1}(q) \cap \Tub_C$, we need to define $\alpha_{q,q'}(x)$. Set $\tilde C \coloneqq \varphi^{-1}(C)$, $\tilde x \coloneqq \varphi^{-1}(x)$, $\tilde q \coloneqq \pi(\tilde x)$ and $\beta\family$-arc
\[
\tilde D \coloneqq \{\beta_{\tilde q, \tilde q''}(\tilde x) \mid \tilde q'' \in \pi(\tilde C)\}.
\]
\begin{figure}
    \centering
    \includegraphics[scale =0.8]{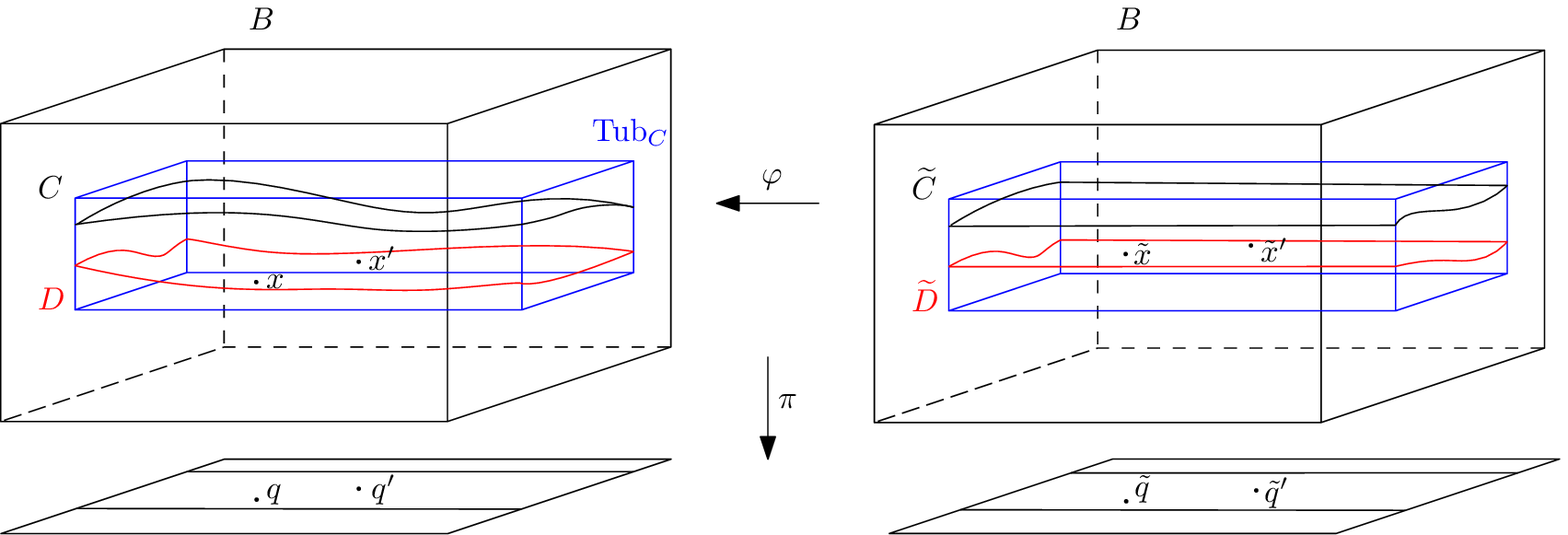}
    \caption{}
    \label{fig:translater}
\end{figure}
By the Fact \ref{fact:riso-affdir}, $\affdir(\tilde C) = \affdir(C) =\bar{U}$, so by the definition of translater, we also have $\affdir(\tilde D)= \bar{U}$.
Using that $\pi$ is an exhibition of $\bar U$, we deduce that the restriction $\varphi_{|\tilde D}\colon \tilde D \to \varphi(\tilde D) \eqqcolon D$ induces (using Lemma~\ref{lem:riso-proj}) a risometry $\varphi_D\colon \pi(\tilde D)\to \pi(D)$. Also, the restriction $\varphi_{|\tilde C}\colon \tilde C \to C $ induces (again by Lemma~\ref{lem:riso-proj}) a risometry $\varphi_C\colon \pi(\tilde C)\to \pi(C)$. Since $\pi(\tilde D) = \pi(\tilde C)$, these two maps can be composed to a
risometry $\psi \coloneqq \varphi_D\circ \varphi_C^{-1}\colon\pi(C) \to \pi(D)$. Note that we have $\pi(D) \subseteq \pi(\Tub_C) = \pi(C)$, so
by Lemma \ref{lem:surjectivity}, the image of $\psi$ is equal to $\pi(C)$. 

Now we can define $\alpha_{q,q'}(x)$, namely as the (unique) element of $\pi^{-1}(q') \cap D$. Setting $\tilde{q}'\coloneqq \varphi_D\1(q')$ (see Figure \ref{fig:translater}), $\alpha_{q,q'}$ is nothing else but 
\[
\alpha_{q,q'}(x) \coloneqq \varphi(\beta_{\tilde q,\tilde{q}' }(\tilde x)). 
\]
It remains to show $\alpha_{q,q'}$ is indeed satisfies the three conditions of translaters; this is a routine computation:

For property (1), since $\beta\family$ reflects $\rho_{\mathcal{S}}\circ\varphi$, we obtain that 
\[
\rho_{\mathcal{S}}\circ \alpha_{q,q'}(x) =  \rho_{\mathcal{S}}(\varphi(\beta_{\tilde q,\tilde{q}' }(\tilde x))) = \rho_{\mathcal{S}}(\varphi(\tilde x)) = \rho_{\mathcal{S}}(\varphi(\varphi\1(x))) = \rho_{\mathcal{S}}(x).  
\]

For property (2), given $q,q',q''\in\pi(C)$ and letting $\tilde{q},\tilde{q}'$ and $\tilde{q}''$ be as above (with $\tilde{q}''\coloneqq \varphi_D(q'')$) we have
\begin{align*}
\alpha_{q',q''}(x)\circ\alpha_{q,q'}(x) & = \varphi(\beta_{\tilde q',\tilde{q}'' }(\varphi\1(\alpha_{q,q'}(x)))) & \\
& = \varphi(\beta_{\tilde q',\tilde{q}'' }(\varphi\1(\varphi(\beta_{\tilde q,\tilde{q}'}(\tilde x))))) & \\
& = \varphi(\beta_{\tilde q',\tilde{q}'' }(\beta_{\tilde q,\tilde{q}'}(\tilde x))) & \text{(since $\beta\family$ is a translater)} \\
& = \varphi(\beta_{\tilde q,\tilde{q}'' }(\tilde x)) = \alpha_{q,q''}(x) & \\
\end{align*}

Property (3) follows from
$\affdir(D) = \affdir(\tilde D) = \bar U$ (where the first equality uses Fact \ref{fact:riso-affdir} once more).
\end{proof}

\begin{prop}\label{prop:EC-translater}
Let $\mathcal{S}=(S_i)_i$ be a t-stratification of $\mathbf{B}$ and $\rho_\mathcal{S}$ be its rainbow. Let $C\subseteq S_{d}$ be a $\rho_\mathcal{S}$-fiber for some $d\leqslant n$. Let $\bar{U}=\mathrm{affdir}(C)$ (so $\dim(\bar{U})=d$) and $\pi\colon K^n\to K^d$ be an exhibition of $\bar{U}$. Then there exists a $\bar U$-translater on $\Tub_C$ reflecting $\rho_\mathcal{S}$. 
\end{prop}

\begin{proof}
We proceed by induction on $d$.
The case $d = 0$ is trivial, so suppose that $d \geqslant 1$.
By Fact \ref{fact:rfiber_induction}, $C$ then is entirely contained in a maximal ball $B\subseteq S_{\geqslant 1}$.
Set $\bar{U}' \coloneqq \tsp_B(\cS)$, $d' \coloneqq \dim(\bar{U}')$, choose an exhibition $\pi'\colon K^n \to K^{d'}$, and let $U'$ be a lift of $\bar{U}'$.   

We have $\bar{U}' \subseteq \bar{U}$. Indeed, for any one-dimensional $W \subseteq \bar{U}'$, we can find $c, c' \in C$ satisfying $W = \dir(c - c')$ by picking $c \in C$ arbitrarily, setting $r \coloneqq \pi'(c)$, choosing some $r' \in \pi'(B)$ satisfying $\dir(r - r') = \pi'(W)$ and then setting $c' \coloneqq \beta_{r,r'}(c)$, where
$\beta\family$ is a $\bar U'$-translater obtained by applying Lemma~\ref{lem:translater} to $B$.

We deduce that $d' \leqslant d$ and that we can pick our exhibition $\pi'$ of $\bar{U}'$ in such a way that it satisfies $\pi' = \theta\circ \pi\colon K^n \to K^{d'}$ for some coordinate projection $\theta\colon K^d \to K^{d'}$; from now on, we assume that $\pi'$ is picked like this.

Note that $C \subseteq B$ implies $\Tub_C \subseteq B$.
Indeed, if $S_{<d'}$ is empty, then $B = \mathbf B$ so that this is trivial, and otherwise, $\rado(B) = v(C, S_{<d'}) \geqslant v(C, S_{<d})  = \lambda_C$. (As a side remark, note that if $d'=d$, then $\Tub_C=B$ and the result follows from Lemma~\ref{lem:translater}.)

Let $\varphi\colon B \to B$ be a definable risometry respecting $\pi'$ and $U'$-trivializing $\rho_{\cS}$ (use Lemma~\ref{lem:respect-fibers}). Fix any $t \in \pi'(B)$ once and for all (for the remainder of the proof) and consider the fiber $\mathbf F \coloneqq (\pi')\1(t)\cap B$. The t-stratification $\mathcal{S}$ induces a t-stratification $\tilde{\mathcal{S}}$ on $\mathbf{F}$, namely $(\varphi^{-1}(S_{i+d'} \cap \mathbf{F}))_{0 \leqslant i \leqslant n-d'}$ (by Fact~\ref{fact:induced-tst}). By Fact \ref{fact:rfiber_induction}, $\tilde C \coloneqq \varphi^{-1}(C \cap \mathbf{F})$ is a rainbow fiber of $\tilde{\mathcal{S}}$. We apply induction to it (note that $\dim( \tilde C) = \dim(C) - d' < \dim(C)$) and obtain a $\bar{V}$-translater $\tilde{\alpha}_{q,q'}$ reflecting $\rho_{\tilde{\mathcal{S}}} = (\rho_{\mathcal S} \circ \varphi)_{|\mathbf{F}}$, where $q, q'$ run over $\pi(\tilde{C})$, and $\bar{V} = \affdir(\tilde{C}) = \bar U \cap \ker \bar\pi'$.

By Lemma \ref{lem:moving_translater}, it suffices to show there is a translater $(\alpha_{q,q'})_{q,q'\in \pi(\varphi\1(\Tub_C))}$ on $\varphi\1(\Tub_C)$ respecting $\rho_{\mathcal{S}}\circ \varphi$. We define such a translater by ``extending
$\tilde{\alpha}_{\tilde q,\tilde q'}$
trivially in the direction of $U'$'', i.e., $\alpha\family$ is the unique translater satisfying
$\alpha_{q,q'} = \tilde{\alpha}_{q,q'}$ if $q, q' \in \pi(\tilde{C})$ and $\alpha_{\pi(x),\pi(x')}(x) = x'$ if $x - x' \in \varphi^{-1}(C) \cap U'$.

To make this formal, we introduce some (abuse of) notation:
Given  $t',t''\in \pi'(B)$ and $x'\in B \cap (\pi')\1(t')$, we write $x' + (t''-t')$ for the
unique element $x''\in B \cap (\pi')\1(t'')$ satisfying $x'' - x' \in U'$, and we also write $\pi(x') + (t''-t')$ for $\pi(x'')$.
In other words, we identify $K^{d'}$ with $U'$ or with $\pi(U')$.

Consider $q',q'' \in \pi(\varphi\1(\Tub_C))$ and $x\in \pi\1(q')\cap \varphi(\Tub_{C})$, let $t'=\theta(q')$ and $t''=\theta(q'')$. We defne 
\[
\alpha_{q', q''}(x) \coloneqq \tilde{\alpha}_{q'+(t-t'),q''+(t-t'')}(x+(t-t')) + (t''-t).
\]
and we claim that this is a $\bar U$-translater reflecting $\rho_\mathcal{S}\circ\varphi$. For property (1) of Definition \ref{def:translater}, 
\begin{align*}
    \rho_\mathcal{S}\circ\varphi \circ\alpha_{q', q''}(x) & = \rho_\mathcal{S}\circ  \varphi (\tilde{\alpha}_{q'+(t-t'),q''+(t-t'')}(x+(t-t')) + (t''-t)). &  \\
    & = \rho_\mathcal{S}\circ  \varphi(\tilde{\alpha}_{q'+(t-t'),q''+(t-t'')}(x+(t-t'))) & \text{($\varphi$ $U'$-trivializes $\rho_\mathcal{S}$)} \\ 
    & = \rho_{\tilde{\mathcal{S}}}(\tilde{\alpha}_{q'+(t-t'),q''+(t-t'')}(x+(t-t'))) &  \\
    & = \rho_{\tilde{\mathcal{S}}}(x+(t-t'))    
    & \text{($\tilde{\alpha}\family$ is a $\bar{V}$-translater)} \\
    & = \rho_\mathcal{S}\circ\varphi(x) & \text{($\varphi$ $U'$-trivializes $\rho_\mathcal{S}$).}
\end{align*}

Condition (2) is an easy computation and is left to the reader. 

For (3), take distinct $q',q''\in
\pi(C)$ and $x\in \pi\1(q')\cap \varphi(\Tub_{C})$.

Set $x_1 \coloneqq x + (t - t')$ and $x_2 \coloneqq \tilde\alpha_{q'+(t-t'),q''+(t-t'')}(x_1)$,
so that $\alpha_{q',q''}(x) = x_3 \coloneqq x_2 + (t''-t)$.
Then we have
$x_1 - x \in U'$, $x_3 - x_2 \in U'$ and
$x_2 - x_1 \in V$. Hence $\dir(x_1 -x + x_3-x_2) \subseteq \bar{U}'$ and $\dir(x_2 - x_1) \subseteq \bar{V}$. By Lemma~\ref{lem:dirsum}, using that $\bar V \cap \bar U' = 0$, it follows that $\dir(x_3 - x) = \dir((x_1 -x + x_3-x_2) + (x_2 -x_3)) \subseteq \bar{U}$.
\end{proof}

\section{\texorpdfstring{\lt}{t²}-Stratifications}
\label{sec:lt}

In this section, we introduce a strengthening of the notion of t-stratifications which we call \lt-stratifications.
The name refers to the error term allowed in the triviality condition: While, in some sense, $d$-triviality allows for error terms of linear order,  \lt-stratifications impose error terms of quadratic order along arcs.
In Section~\ref{sec:hierarchy}, we will provide some examples suggesting that one can define an entire hierarchy of \lt[r]-stratifications for $r \in \N$.

Our main motivation to consider \lt-stratifications is that they are tightly related to the Lipschitz stratifications introduced by Mostowski \cite{mostowski} and to their valuative variant introduced in \cite{halup-yin-18}. Indeed, we will show in Section \ref{sec:lipschitz} that every \lt-stratification is a valuative Lipschitz stratification as defined in \cite{halup-yin-18}. The main benefit of working with \lt-stratifications is that they are significantly less technical than the definition of (classical or valuative) Lipschitz stratifications, in particular completely avoiding the concept of ``chains'' needed there. Our hope is that this new approach will in this way provide a much better understanding of Lipschitz stratifications, for example opening the road towards further strengthenings of that notion, as we indeed do in Section~\ref{sec:analytic}. The existence of \lt-stratifications will be proven in Section~\ref{sec:analytic}. 

\subsection{\lt-stratifications}\label{sec:SUT}

We continue working in a $1$-h-minimal language as in Definition~\ref{defn:L}. Let us start by fixing what exactly we mean by a definable (embedded) manifold. Those work best if one uses strict differentiability, a notion we recall.

\begin{defn}[Strict differentiability and $C^1$-bijection]
A map $\psi\colon B \to K^m$ (for some set $B \subseteq K^n$) is called \emph{strictly differentiable} at a point $a$ in the interior of $B$, with total derivative $L \in\mathrm{Mat}_{m\times n}(K)$, if
\[
\lim_{\substack{x, x' \to a,\\x\ne x'}}
\frac{v(\psi(x) - \psi(x') - L(x-x'))}{v(x-x')} = 0.
\]
We denote the total derivative of $\psi$ at $a$ by $D\psi(a)$.
We call a map $\psi\colon B \to B'$ between open sets $B, B'$ \emph{strictly $C^1$}, if it is strictly differentiable at every point of $B$. We say it is a \emph{strictly $C^1$-bijection} if it is a bijection and both, $\psi$ and $\psi\1$ are strictly $C^1$.
\end{defn}

\begin{defn}[Definable manifold and tangent space]\label{def:manifold} For integers $0\leqslant d \leqslant n$,
\begin{enumerate}
    \item a \emph{$d$-dimensional definable $C^1$-submanifold of $K^n$} is a definable set $X \subseteq K^n$ such that each $a \in X$ lies in a definable chart, i.e., there exists an open set $B \subseteq K^n$ containing $a$ and a definable strictly $C^1$-bijection $\psi\colon B \to B' \subseteq K^n$ such that $\psi(X \cap B) = (K^d \times \{0\}^{n-d}) \cap B'$;
\item if $X \subseteq K^n$ is a definable $C^1$-manifold, then 
the \emph{tangent space to $X$ at $a \in X$}
is defined as $T_aX \coloneqq D\psi\1(K^d \times \{0\}^{n-d}) \subseteq K^n$ for some chart $\psi\colon B \to B'$ around $a$.
\end{enumerate}
\end{defn}

\begin{rem}\label{rem:manifold}
\begin{enumerate}
    \item 
As usual, one verifies that the tangent space is well-defined, i.e., does not depend on the choice of the chart.
\item
Any $C^1$-manifold $X$ is locally equal to the graph of a strictly $C^1$-function: Indeed, for $a \in X$, consider the composition $\psi' \coloneqq (D\psi(a))\1 \circ \psi$, which has total derivative $D\psi'(a) = I_n$.
Using that this is a strict derivative, one deduces that after sufficiently shrinking the domain $B$ of $\psi'$, we have $v(\psi(x) - \psi(x') - (x - x')) < v(x - x')$ for all $x, x' \in B$. In other words, $\psi'$ is a risometry, and its inverse $U$-trivializes $X$ on $B$ for $U \coloneqq (D\psi(a))\1(K^d \times \{0\}^{n-d}) = T_aX$.
(Here, we assume that $B$ is a ball.)
Using Section~\ref{sec:tsts} (e.g. Lemma~\ref{lem:respect-fibers}), one obtains that $X$ is indeed the graph of a function on $\pi(B)$, for any exhibition $\pi$ of $\res(U)$.
\item
We can now deduce that we even have a uniformly definable family of charts covering $X$, namely: For each $a \in X$, pick an exhibition of $\res(T_aX)$; let $B$ be the largest convex set\footnote{meaning that if $B$ contains some $x_1, x_2$, then it contains all of $B(x_1, \leqslant v(x_1 - x_2))$.} containing $a$ such that $X\cap B'$ is the graph of a function $f\colon \pi(B) \to \pi^\vee(B)$; and take $\psi\colon B \to \pi(B) \times K^{n-d}, x \mapsto (\pi(x), \pi^\vee(x) - f(\pi(x))$ as a chart.
\end{enumerate}
\end{rem}

For large parts of this section, we fix the following:
\begin{itemize}
    \item 
Fix a ball $B \subseteq K^n$ ($B$ can be either open or closed; later, we will do a case distinction on this; see below Convention~\ref{conv:balls}.) We will often divide by the radius of $B$. In the entire section, we will use the convention that if $B = K^n$, then $\frac{\beta}{\rado(B)} = \frac1{\infty} \coloneqq 0$ for any $\beta \in \Gamma$.
\item
Fix a natural number $d \leqslant n$, a vector space $\bar{U}\subseteq \bar K^n$ of dimension $d$, a lift $U \subseteq K^n$ of $\bar U$, and an exhibition $\pi\colon K^n\to K^d$ of $\bar U$.
We write $\pi_U\colon K^n\to K^n/U$ for the quotient map and $B/U$ is the image of $B$ under $\pi_U$. Recall also that $\pi^\vee\colon K^n\to K^{n-d}$ denotes the complement projection of $\pi$.
\item Finally, fix a definable risometry $\varphi\colon B\to B$.  
\end{itemize}
From now on, we will have to treat the case where $B$ is closed and the case where $B$ is open slightly differently: strict inequalities become weak and vice versa. We will write down everything in the version for closed balls; the following convention explains how to obtain the corresponding version for open balls:

\begin{conv}\label{conv:balls}
Either we assume that the ball $B$ is closed; in that case, the notation closed$^*$, $\lsnew$, $\lenew$, $\radc^*(\cdot)$ just means the usual closed, $<$, $\leqslant$ and $\radc(\cdot)$. Or we assume that $B$ is open; then by closed$^*$, $\lsnew$, $\lenew$, $\radc^*(\cdot)$ we mean open, $\leqslant$, $<$, and $\rado(\cdot)$.
\end{conv}
For example, in both cases, $B$ is a closed$^*$ ball, and for $a \in B$, we have $B = B(a, \leqslant^* \radc^*(B))$.

\begin{defn}[2-Taylor]\label{def:STSU2} The risometry $\varphi$ is \emph{2-Taylor along $U$} if for every $b_0\in B/U$, the set $\Arc_{b_0}(U,\varphi)$ is a definable $C^1$-submanifold of $B$,
and for every $x_0\in \Arc_{b_0}(U,\varphi)$, every $b \in B/U$ and every $x_1,x_2\in \Arc_b(U,\varphi)$ with $x_1 \ne x_2$ there exists $w\in T_{x_0}(\Arc_{b_0}(U,\varphi))$ such that 
\begin{equation}\label{eq:SUT2}
v(x_1-x_2-w)\lsnew\frac{v(x_1 - x_2)\cdot\max\{v(x_0-x_1), v(x_0-x_2)\}}{\radc^*(B)}. 
\end{equation}
\end{defn}

Note that by our conventions, in the case $B = K^n$, \eqref{eq:SUT2} becomes ``$v(x_1 - x_2 - w ) = 0$''.

\begin{defn}\label{def:Lip-trivial}
Let $\mathbf{B}$ be an open ball of $K^n$. Let $\chi\colon \mathbf{B} \to \RV\eq$ be a definable map and $B\subseteq \mathbf{B}$ be a ball. The map $\chi$ is said to be \emph{2-Taylor $d$-trivial on $B$} if there exist a vector space $\bar{U}\subseteq \bar K^n$ of dimension $d$, a lift $U$ of $\bar{U}$ and a definable risometry $\varphi\colon B\to B$ such that
\begin{enumerate}
    \item $\chi$ is $U$-trivialized by $\varphi$ and 
    \item $\varphi$ is 2-Taylor along $U$. 
\end{enumerate}
More generally, we also apply the above definition to tuples of sets $X_i \subseteq \mathbf B$ and maps $\chi_i\colon K^n\to \RV\eq$: 
$(X_1, \dots, X_{m'}, \chi_1, \dots, \chi_m)$ is \emph{2-Taylor $d$-trivial on $B$} if there exists a single $\varphi$ as above which is 2-Taylor along $U$ and which $U$-trivializes the entire tuple.
\end{defn}

\begin{defn}[\lt-stratification]\label{def:sts} Let $\mathbf{B}$ be an open ball of $K^n$. A partition $\mathcal{S}=\{S_0,\ldots,S_n\}$ of $\mathbf{B}$ is a \emph{\lt-stratitication} if for every $d\in \{0,\ldots, n\}$ and every ball $B \subseteq S_{\geqslant d}$, $(S_0,\ldots,S_n)$ is 2-Taylor $d$-trivial on $B$.  
\end{defn}

\begin{lem}\label{lem:manifold} 
If $\mathcal{S}=\{S_0,\ldots,S_n\}$ is a \lt-stratitication of $\mathbf{B}$, then each $S_d$ is a definable $C^1$-manifold. 
\end{lem}

\begin{proof}
Indeed $S_d$ is convered by definable $C^1$-submanifolds of dimension $d$ of the form $\Arc_{b}(U,\varphi)$, for suitable $b$, $U$ and $\varphi$. 
\end{proof}

A \lt-stratification is in particular a t-stratification, so we have a notion of when it reflects a map $\chi\colon \mathbf B \to \RV\eq$, namely, if for every $d$ and every $B \subseteq S_{\geqslant d}$, $(\cS, \chi)$ is $d$-trivial on $B$. It might seem more natural in this case to impose that $(\cS, \chi)$ is 2-Taylor $d$-trivial on $B$; note that indeed this is automatic by Remark~\ref{rem:refl}.

\

In the remainder of Section \ref{sec:SUT} we show an alternative characterisation of being 2-Taylor along $U$, which is seemlingly coordinate dependent (and hence less suitable as the definition) but more handy to work with. To state the alternative characterization, it will be handy to consider the $(U,\varphi)$-arcs as graphs of definable functions and prove a couple of supplementary lemmas about them.

\begin{defn}\label{def:g_b} Given $b \in B/U$, we define the function $\arc_b\colon \pi(B)\to K^{n-d}$ as the map sending $x\in \pi(B)$ to the image under $\pi^\vee$ of the unique element in $\Arc_b(U,\varphi) \cap \pi\1(x)$.
\end{defn}

\begin{lem}\label{lem:S_3U-unique} Suppose $\varphi$ is 2-Taylor along $U$. Let $b_0,b, x_0, x_1, x_2$ be as in Definition \ref{def:STSU2}. Then the unique $w\in T_{x_0}(\Arc_{b_0}(U,\varphi))$ such that $\pi(x_1-x_2)=\pi(w)$ satisfies \eqref{eq:SUT2}.
\end{lem} 

\begin{proof}
By Remark \ref{rem:lift-choice},
we may assume without loss that $U = T_{x_0}(\Arc(U,\varphi))$.

Let $w'\in T_{x_0}(\Arc(U,\varphi))$ satisfy (\ref{eq:SUT2}).
It suffices to show that
\[
v(w - w')\lsnew\frac{v(x_1 - x_2)\cdot \max\{v(x_0-x_1), v(x_0-x_2)\}}{\radc^*(B)}, \]
since the result then follows by the ultrametric inequality. And indeed:
\begin{align*}
v(w - w') &= v(\pi(w) - \pi(w')) &\text{(Remark \ref{rem:exhibition})}\\
&=
v(\pi(x_1-x_2)-\pi(w'))\\
& =v(\pi(x_1-x_2-w'))\\
					& \leqslant v(x_1-x_2-w')					\\
					&\lsnew\frac{v(x_1-x_2)\cdot\max\{v(x_0-x_1), v(x_0-x_2)\}}{\radc^*(B)}.
					&\text{(by (\ref{eq:SUT2}))}
\end{align*}
\end{proof}

\begin{prop}\label{prop:SU-equivalence} The following are equivalent:
\begin{enumerate}[label=(\roman*)]
\item for every $u_0,u_1,u_2\in \pi(B)$ with $u_1 \ne u_2$ and all $b_0,b\in B/U$, $\arc_{b_0}$ is strictly $C^1$ and we have 
\begin{equation}\label{eq:SUT1}
v(\arc_b(u_1)-\arc_b(u_2)-D\arc_{b_0}(u_0)(u_1-u_2))\lsnew\frac{v(u_1-u_2)\cdot\alpha}{\radc^*(B)},
\end{equation}
where $\alpha\coloneqq \max\{v(u_0-u_1), v(u_0 - u_2), v(b_0 - b)\}$; 
    \item same condition as in (i) but with $u_0=u_2$;
    \item $\varphi$ is 2-Taylor along $U$.
    \end{enumerate}
\end{prop} 

\begin{rem}\label{rem:SUT1}
For $b, b' \in B/U$ and $u, u' \in \pi(B)$, we have
\[
\max\{v(u- u'), v(b - b')\} = v((u, \arc_b(u)) - (u', \arc_{b'}(u')));
\]
so in particular, the $\alpha$ in (i) can also be written as
\[
\alpha = \max\{v((u_0, \arc_{b_0}(u_0)) - (u_1, \arc_{b}(u_1))), v((u_0, \arc_{b_0}(u_0)) - (u_2, \arc_{b}(u_2)))\}.
\]
\end{rem}

\begin{proof}[Proof of Proposition \ref{prop:SU-equivalence}] 
$(i) \to (ii)$ is trivial;
let us show $(i)\to (iii)\to (i)$ and $(ii) \to (i)$. By Lemma \ref{lem:respect-fibers} we may suppose that $\varphi$ respects $\pi$.  

$(i)\to(iii):$ Define $\psi\colon B \to \pi(B) \times K^{n-d}$ by 
\[
x \mapsto (\pi(x), \pi^\vee(x) - \arc_{b_0}(\pi(x)).
\]
This is a bijection from $B$
to $\pi(B) \times \pi^\vee(B - B)$
(where $B - B = \{b - b' \mid b, b' \in B\}$ is the translate of $B$ containing $0$), and
by (i), $\psi$ and its inverse are strictly $C^1$, so it is a chart, witnessing that $\Arc_{b_0}(U, \varphi)$ is a definable $C^1$-submanifold of $B$.
Now let $b_0,b\in B/U$, $x_0\in \Arc_{b_0}(U,\varphi)$ and $x_1,x_2\in \Arc_b(U,\varphi)$ be given. Then \eqref{eq:SUT2} follows from \eqref{eq:SUT1} applied to $u_i \coloneqq \pi(x_i)$ (for $i=0,1,2$), using Remark \ref{rem:SUT1} and choosing
\[
w \coloneqq
(u_1 - u_2, D\arc_{b_0}(u_0)(u_1-u_2)).
\]
Note that $w$ is indeed an element of $T_{x_0}(\Arc_{b_0}(U, \varphi))$,
since for $f(u) \coloneqq (u, \arc_{b_0}(u))$, the image of the composition $\varphi \circ f$ lies in $K^d \times \{0\}^{n-d}$, which implies that $w$, since it lies in the image of $D f(u_0)$, gets sent to $K^d \times \{0\}^{n-d}$ by $D\psi(x_0)$.

$(iii)\to(i):$ Fix $b_0 \in B/U$, $u_0 \in \pi(B)$ and set $x_0 \coloneqq (u_0,\arc_{b_0}(u_0))$. Since $\Arc_{b_0}(U,\varphi)$ is a strict $C^1$-manifold, as in Part (2) of Remark \ref{rem:manifold}, one obtains that $\arc_{b_0}$ is strictly $C^1$.
To show $\eqref{eq:SUT1}$, given $b\in B/U$ and $u_1,u_2\in \pi(B)$, let 
\[
w=((u_1-u_2),D{\arc_{b_0}}(u_0)(u_1-u_2))=M_{x_0}(u_1-u_2).
\]
Note that $w$ is the unique element
of $T_{x_0}\Arc_{b_0}(U,\varphi)$ satisfying $\pi(w) = \pi(u_1 - u_2)$. 
Thus, by Lemma \ref{lem:S_3U-unique} and Remark \ref{rem:SUT1}, \eqref{eq:SUT1} follows from \eqref{eq:SUT2} taking $x_i=(u_i,\arc_b(u_i))$ for $i=1,2$.  

\

$(ii)\to (i):$ Assume that $\eqref{eq:SUT1}$ holds in the case $u_2 = u_0$, i.e., that we have
\begin{equation}\label{eq:step1.1}
v(\arc_b(u_1)-\arc_b(u_2)-D\arc_{b_0}(u_2)(u_1-u_2))\lsnew\frac{v(u_1-u_2)\cdot\max\{v(u_1-u_2), v(b-b_0)\}}{\radc^*(B)}.   
\end{equation}
By the triangle inequality, Condition \eqref{eq:SUT1} for $u_0,u_1,u_2$ with $u_0\neq u_2$ follows from \eqref{eq:step1.1} and  
\begin{equation}\label{eq:step1.2}
v((D\arc_{b_0}(u_2)-D\arc_{b_0}(u_0))(u_1-u_2))\lsnew\frac{v(u_1-u_2)\max\{v(u_0-u_2), v(b-b_0)\}}{\radc^*(B)}.    
\end{equation}
To show \eqref{eq:step1.2} assuming $u_0\neq u_2$, apply $\eqref{eq:step1.1}$ once to the pair $(u_0,u_2)$ and once to the pair $(u_2,u_0)$ to obtain   
\begin{align*}
        v(\arc_b(u_0)-\arc_b(u_2)-D\arc_{b_0}(u_2)(u_0-u_2))\lsnew 
        \frac{v(u_0-u_2)\cdot \max\{v(u_0-u_2), v(b-b_0)\}}{\radc^*(B)}\\
        v(\arc_b(u_2)-\arc_b(u_0)-D\arc_{b_0}(u_0)(u_2-u_0))\lsnew \frac{v(u_2-u_0)\cdot \max\{v(u_2-u_0), v(b-b_0)\}}{\radc^*(B)}.
\end{align*}
The inequality \eqref{eq:step1.2} follows directly by applying the triangle inequality to the sum of the left-hand terms.  
\end{proof}

We finish this section with a technical lemma which will ne needed in Section \ref{sec:lipschitz}. Recall (Lemma-Definition~\ref{def:VSdistance}) that $\Delta(U, U')$ denotes a natural notion of distance between sub-vector spaces $V, V' \subseteq K^n$ of the same dimension.

\begin{lem}\label{lem:dist_tangentS2} Suppose $\varphi$ is 2-Taylor along $U$. Fix $b',b''\in B/U$ and let $x'\in \Arc_{b'}(U,\varphi)$, $x''\in \Arc_{b''}(U,\varphi)$. Then 
\begin{equation}\label{eq:dist}
\Delta\big(T_{x'}(\Arc_{b'}(U,\varphi)), T_{x''}(\Arc_{b''}(U,\varphi))\big) \lsnew \frac{v(x' - x'')}{ \radc^*(B)}.
\end{equation}
\end{lem}

\begin{proof}
Given any $w' \in T_{x'}(\Arc_{b'}(U,\varphi))$,
we let $w'' \in T_{x''}(\Arc_{b''}(U,\varphi))$ be the unique vector satisfying $\pi(w') = \pi(w'')$. To prove the lemma, it suffices to prove that
\[
v(w' - w'') \lsnew \frac{v(w') \cdot v(x' - x'')}{\radc^*(B)}
\]
Without loss, we may suppose that $v(w'), v(w'')\leqslant v(x'-x'')$. By Lemma \ref{lem:S_3U-unique} applied to $b_0=b'$, $b=b''$, $x_0=x'$, $x_1=x''$ and $x_2$ the unique element in $\Arc_{b''}(U,\varphi)$ such that $\pi(x''-x_2)=\pi(w')$, we have 
\[
v(x'' - x_2 - w') \lsnew \frac{v(x'' -x_2) \cdot \max\{v(x' - x''), v(x'-x_2)\}}{\radc^*(B)}.
\]
Since we have $v(x'' - x_2) = v(w') \leqslant v(x' - x'')$, the maximum is equal to $v(x' - x'')$.

Proceeding analogously with the same $b$, $x_1$ and $x_2$, but with $b_0 = b''$ and $x_0 = x''$, we obtain
\begin{align*}
v(x'' - x_2 - w'') &\lsnew \frac{v(x'' -x_2) \cdot \max\{v(x'' - x''), v(x''-x_2)\}}{\radc^*(B)}\\
&\leqslant \frac{v(x'' -x_2) \cdot v(x'-x'')}{\radc^*(B)}
\end{align*}
which shows the result. 
\end{proof}

\subsection{Relation to valuative Lipschitz stratifications}\label{sec:lipschitz}

In this section we describe the relation between \lt-stratifications and Mostowski's Lipschitz stratifications.
Instead of directly working with Mostowski's original definition, we use the notion of valuative Lipschitz stratifications from \cite{halup-yin-18}, which is closely related but formulated in terms of valued fields; the precise relation between those tho notions is recalled in Proposition~\ref{prop:liplip}.

The main result of this section (Theorem~\ref{thm:sts-to-Lipstrats}) essentially says that every \lt-stratifiation is in particular a valuative Lipschitz stratification.
However, the notion of valuative Lipschitz stratification assumes that the strata are definable in a language not involving the valuation, so this needs to be additionally imposed.\footnote{One could, in principle, define valuative Lipschitz stratifications in a language with valuation, but if the field is not definably connected, then this definition is not particularly useful.}

We start by fixing a setting where ``language without valuation'' makes sense. We will call such a language the ``original language'' and denote it by $\Lomin$. 

\begin{conv}\label{conv:lip_setting}
In Section~\ref{sec:lipschitz}, we assume that we are in one of the following two settings:
\begin{enumerate}
    \item[(ACF)] $K$ is an algebraically closed non-trivially valued field of equi-characteristic $0$, $K' \subseteq \valring$ is an algebraically closed subfield contained in the valuation ring, $\Lomin$ is the ring language on $K$ (without predicate for the valuation ring), possibly expanded by constants for elements of a subset of $K'$, and $\Lx_0$ is the expansion of $\Lomin$ by a predicate for the valuation ring. (So $\Lx_0$ is the language $\Lval$ from Definition~\ref{defn:LHen}, possibly expanded by some constants).
   \item[(OMIN)] $K$ is a real closed field, $\Lomin$ is a power-bounded o-minimal expansion of the ring language on $K$,
   $\valring$ is the convex closure of a proper elementary substructure $K' \precneqq_{\Lomin} K$ (we assume that such a $K'$ exists),
   and $\Lx_0$ is obtained from $\Lomin$ by adding a predicate for $\valring$.
\end{enumerate}
In both cases, let $\Lx$ be the language obtained from $\Lx_0$ by adding all the sorts $\RV\eq$, as explained in Definition~\ref{defn:L}.
\end{conv}

We need a notion of $\Lomin(K)$-definable $C^1$-manifold. In Setting~(ACF), we cannot expect to have $\Lomin(K)$-definable charts (in the sense of Definition~\ref{def:manifold}), since small neighbourhoods are not definable. Let us therefore use the following definition, which works in all settings:
\begin{defn}\label{def:manifold.v2}
By an \emph{$\Lomin(K)$-definable $C^1$-submanifold} of $K^n$, we mean an $\Lomin(K)$-definable subset of $K^n$ which is a definable $C^1$-manifold in the sense of Definition~\ref{def:manifold} for the language $\Lx$, i.e., such that every $x \in X$ lies in an $\Lx(K)$-definable chart.
\end{defn}
\begin{rem}\label{rem:C1=C1}
Note that this definition agrees with various other notions of manifold. In particular, in Setting (OMIN),
it agrees with the notion of $C^1$-manifold from \cite[Notation~1.1.9]{halup-yin-18}:
One can as well impose that the charts are $\Lomin(K)$-definable (since the order topology and the valuation topology coincide, and since definable maps are locally $\Lomin(K)$-definable); after that,
finitely many charts suffice to cover the manifold (using o-minimal cell decomposition).
\end{rem}

We do not recall what ``power-bounded'' means. Just note that (OMIN) is exactly the setting of \cite{halup-yin-18}, and also note that in both settings (ACF) and (OMIN), the $\Lx$-theory of $K$ is $1$-h-minimal (by \cite[Section~6]{clu-hal-rid}), so in particular, all of Sections~\ref{sec:setting}, \ref{sec:tsts} and \ref{sec:SUT} applies.

We will now recall the notion of valuative Lipschitz stratifications. This needs the following ingredient, which is 
\cite[Definition~1.6.1]{halup-yin-18} (though we do not distinguish between ``plain val-chains'' and ``augmented val-chains'').

\begin{defn}\label{def:valchain} Let $1\leqslant m\leqslant n$ be integers and let $\dd=(d_0,\ldots,d_{m})$ be a sequence of non-negative integers. Let $\mathcal{S}=\{S_0,\ldots,S_n\}$ be a 
partition of $K^n$ into $\Lomin(K)$-definable sets $S_i$ of dimension at most $i$. A sequence of points $(a_0,\ldots,a_{m})\in (K^n)^{m+1}$ is called a \emph{$\dd$-val-chain for $\mathcal{S}$} if 
\begin{enumerate}
\item $n\geqslant d_0 \geqslant d_1>d_2>\cdots>d_{m}$
\item $a_i\in S_{d_i}$, for $0\leqslant i\leqslant m$
\item $\lambda_i\coloneqq  v(a_0-a_i) <v(a_0-S_{<d_i})$ if $S_{<d_i} \ne \emptyset$, for $1\leqslant i\leqslant m$
\item
$v(a_0-a_i) = v(a_0-S_{<d_{i-1}})$ if $d_{i-1} > d_i$, for $1\leqslant i\leqslant m$
\end{enumerate}
We set $\lambda_{m+1}\coloneqq v(a_0-S_{<d_m})$, or $\lambda_{m+1} = \infty$ if $S_{<d_m} = \emptyset$. We call the sequence $(\lambda_1,\ldots, \lambda_{m+1})$ \emph{the distances of $(a_0,\ldots, a_m)$}. 
\end{defn}

Note that in (1) it is intentional that we do allow $d_1 = d_0$ but that we require $d_{i} < d_{i-1}$ for $i > 1$. In particular, the ``if $d_{i - 1} > d_i$'' in (4) can only fail for $i = 1$.

Instead of using the original definition of valuative Lipschitz stratification (\cite[Definition~1.6.5]{halup-yin-18}), we use the equivalent condition given in
\cite[Proposition~1.8.3]{halup-yin-18}.
Also note that whereas in \cite{halup-yin-18}, stratifications of closed definable subsets of $K^n$ are considered, here, we consider stratifications of the entire ambient space $K^n$.

\begin{defn}\label{def:Lip-strat} A \emph{valuative Lipschitz stratification}
of $K^n$ is a partition
$\mathcal{S}=\{S_0,\ldots,S_n\}$ of $K^n$
such that each $S_i$ is an $\Lomin(K)$-definable $C^1$-manifold of dimension $i$, and
for every $1\leqslant m\leqslant n$, for every sequence of non-negative integers $\dd=(d_0,\ldots,d_{m})$ and every $\dd$-val-chain $(a_0,\ldots,a_m)\in (K^n)^{m+1}$ with distances $(\lambda_1, \dots, \lambda_{m+1})$, there are $K$-sub-vector spaces $V_{i,j} \subseteq K^n$ for $0\leqslant i\leqslant j \leqslant m$ such that  

\begin{enumerate}
\item $V_{i,i}= T_{a_i} (S_{d_i})$
for $0 \leqslant i \leqslant m$;
\item $V_{i,j}\subseteq V_{i,j-1}$
for $0 \leqslant i < j \leqslant m$; 
\item $\dim(V_{i,j})=d_j$ for $0 \leqslant i \leqslant j \leqslant m$;
\item
$\Delta(V_{i,j}, V_{i+1,j})\leqslant \frac{\lambda_{i+1}}{\lambda_{j+1}}$
for $0 \leqslant i < j \leqslant m$,
with the convention that if $\lambda_{m+1} = \infty$, then $\frac{\lambda_{i+1}}{\lambda_{m+1}} \coloneqq 0$.
\end{enumerate}
\end{defn}

We can now state the main result of this subsection:

\begin{thm}\label{thm:sts-to-Lipstrats}
Suppose that $\cS=(S_0,\ldots, S_n)$ is a \lt-stratification of $K^n$ such that each $S_i$ is $\Lomin(K)$-definable (where $\Lomin$ is as in Convention~\ref{conv:lip_setting}). Then $\cS$ is also a valuative Lipschitz stratification. 
\end{thm}

We believe that the converse of this result also holds, i.e., that for $\Lomin(K)$-definable partitions of $K^n$, being an \lt-stratification and being a valuative Lipschitz stratification are equivalent (see Question~\ref{q:converse}). 

\

Before proving Theorem~\ref{thm:sts-to-Lipstrats}, let us make precise the relation to
the original notion of Lipschitz stratifications by Mostowksi.
We first explain this in the (OMIN) setting, where this is a result from \cite{halup-yin-18}.

Originally, Mostowksi introduced Lipschitz stratifications over the complex numbers \cite{Mos.bilip}, but it also makes sense over the reals \cite{Par.lipSemi,Par.lipSub}, and more generally over arbitrary real closed fields with a suitable langauge;
see \cite[Definition~1.2.4]{halup-yin-18}. In particular, it does make sense in our setting (OMIN), for the language $\Lomin$.

(In the following proposition, we secretly use Remark~\ref{rem:C1=C1}.)

\begin{prop}[{\cite[Propositions~1.6.11 and 1.8.3]{halup-yin-18}}]\label{prop:liplip}
Suppose that we are in the setting (OMIN) from Convention~\ref{conv:lip_setting}.
Then an $\Lomin$-definable partition $\cS$ of $K^n$ is a Lipschitz stratification in the sense of \cite[Definition~1.2.4]{halup-yin-18} if and only if it is a valuative Lipschitz stratification in the sense of Definition~\ref{def:Lip-strat}.
\end{prop}

Note that for this equivalence to hold, it is important that $\cS$ is definable without additional parameters from $K$. However, we can always add constants for elements from $K'$ to $\Lomin$.

In \cite{halup-yin-18}, only the setting (OMIN) is considered, but all of \cite[Section~1]{halup-yin-18} (and actually, probably all of \cite{halup-yin-18}) also works in the setting (ACF), as follows.

By passing to an elementary extension, we may suppose that our given algebraically closed valued field $K$ is the algebraic closure of a real closed valued field
$K_{\mathrm{rcf}}$ inducing the valuation on $K$ and satisfying the setting (OMIN), where as $\Lomin$, we take the ring-language expanded by constants for all elements of the elementary substructure $K_{\mathrm{rcf}}' \prec K_{\mathrm{rcf}}$ whose convex closure is the valuation ring.

To consider stratifications in the algebraic closure $K$ of $K_{\mathrm{rcf}}$, we interpret $K$ in $K_{\mathrm{rcf}}$. In this way, all the definitions and arguments from \cite[Section~1]{halup-yin-18} make sense in $K$, using the language $\Lomin$ on $K_{\mathrm{rcf}}$ everywhere, with two differences:
Firstly, the sets $S_i$ forming the stratification are definable in the ring-language on $K$ (though the charts witnessing that they are manifolds are definable in the bigger language, so that in the case $K = \C$, we obtain the usual notion of analytic manifolds). And secondly, when we consider sub-vector spaces of $K^n$, in particular in \cite[Propositions~1.8.3]{halup-yin-18}, we mean $K$-sub-vector spaces and not $K_{\mathrm{rcf}}$-sub-vector spaces. (The entire proof of \cite[Propositions~1.8.3]{halup-yin-18} is just ``valuative linear algebra'', without any model theory, so $K_{\mathrm{rcf}}$ plays no role here.) In this way, one obtains the following proposition.

\begin{prop}\label{prop:liplipacf}
Suppose that we are in the setting (ACF) from Convention~\ref{conv:lip_setting}.
Then an $\Lomin$-definable partition $\cS$ of $K^n$ is a Lipschitz stratification in the sense of \cite[Definition~1.2.4]{halup-yin-18} if and only if it is a valuative Lipschitz stratification in the sense of Definition~\ref{def:Lip-strat}.
\end{prop}

\begin{rem}\label{rem:t2-to-lip}
Combining Theorem~\ref{thm:sts-to-Lipstrats} with Propositions~\ref{prop:liplip} and \ref{prop:liplipacf} shows, in any of the above settings:
Any $\Lomin$-definable \lt-stratification $\cS$ of $K^n$ is a Lipschitz stratification.
Since being a Lipschitz stratification is a first order property, it therefore also induces a Lipschitz stratification
$(S_0 \cap (K')^n, \dots, S_n \cap (K')^n)$ of $(K')^n$. In particular, if $K'$ is $\C$, this provides a new way to obtain Lipschitz stratifications in the original sense of Mostowski, namely as follows: Suppose that
$X' \subseteq \C^n$ is $\Lomin$-definable. Let $X \subseteq K^n$ be the set defined by interpreting the same formula in $K \succ \C$. By Corollary~\ref{cor:t2-exist} we find an $\Lx$-definable \lt-stratification $\cS$ reflecting (the characteristic function of) $X$. By proceeding as in the proof of \cite[Proposition~6.2]{halupczok2014a}, we may improve $\cS$ in such a way that $S_{\leqslant d}$ becomes Zariski closed for every $d$ and hence $\Lomin$-definable. (The idea of that proof is to replace $S_{\leqslant d}$ by its Zariski closure $S_{\leqslant d}^{\mathrm{Zar}}$, then repair the stratification using a \lt-stratification reflecting $(X, S_{\leqslant 0}^{\mathrm{Zar}}, \dots, S_{\leqslant n}^{\mathrm{Zar}})$, take the Zariski closures again, etc.; if this is done carefully, then $\dim( \bigcup_d(S_{\leqslant d}^{\mathrm{Zar}} \setminus S_{\leqslant d}))$ decreases in each step, until we have $S_{\leqslant d}^{\mathrm{Zar}} = S_{\leqslant d}$.) 
As explained above, $\cS' \coloneqq (S_0 \cap \C^n, \dots, S_n \cap \C^n)$ is a Lipschitz stratification of $\C^n$. Finally note that the fact that $\cS$ reflects $X$ implies that the set $X'$ we started with is a union of connected components of the $S'_i$, by (the easy) Part (2) of \cite[Theorem~7.11]{halupczok2014a}. 
\end{rem}

We now come to the proof of Theorem~\ref{thm:sts-to-Lipstrats}. 

\begin{proof}[Proof of Theorem~\ref{thm:sts-to-Lipstrats}] Let $\mathcal{S}=\{S_0,\ldots,S_n\}$ be a \lt-stratification of $K^n$.
By Lemma \ref{lem:manifold}, each $S_i$ is a definable $C^1$-manifold, so we just need to check the condition related to chains.
To this end, fix a $\dd$-val-chain
$(a_0,\ldots, a_m)$ for $\mathcal{S}$ with $\dd=(d_0,\ldots,d_m)$ and with distances $(\lambda_1, \dots, \lambda_{m+1})$. For each $j\in \{0,\ldots,m\}$, set $B_j \coloneqq B(a_0, {<\lambda_{j+1}})$. This ball contains $a_0, \dots, a_j$, and it is disjoint from
$S_{<d_j}$.
Note that if $S_{<d_m} = \emptyset$, then by our above conventions, we have $B_m = B(a_0, <\infty) = K^n$.

For each $j$, fix a lift $U_j \subseteq K^n$ of $\tsp_{B_j}(\cS)$ (note that $\dim(U_j) = d_j$) and a definable risometry $\varphi_j\colon B_j\to B_j$ trivializing $\cS$ which is 2-Taylor along $U_j$ (witnessing that $\cS$ is 2-Taylor $d_j$-trivial on $B_j$).

For each $0 \leqslant i\leqslant j\leqslant n$, set $b_{i,j} := \pi_{U_j}(\varphi_j^{-1}(a_i)) \in B/U_j$ (so that
$a_i \in \Arc_{b_{i,j}}(U_j,\varphi_j)$),
and denote the tangent space
$T_{a_i}(\Arc_{b_{i,j}}(U_j,\varphi_j))$ 
by $W_{i,j}$.

\begin{cla}\label{cla:W-prop}The following holds for $0 \leqslant i \leqslant j \leqslant n$:
\begin{enumerate}[label=(\roman*)]
    \item $\dim(W_{i,j}) = d_j$,
    \item $W_{i,i}=T_{a_i}(S_{d_i})$,
    \item $W_{i,j}\subseteq T_{a_i}(S_{d_i})$,
    \item $\Delta(W_{i+1,j}, W_{i,j})\leqslant \frac{\lambda_{i+1}}{\lambda_{j+1}}$ if $i < j$.
\end{enumerate}
\end{cla}

The first property is clear.
Property $(ii)$ follows from $(iii)$, since both $K$-vector spaces have the same dimension. 

For (iii), note that since $S_{d_i}$ is $U_j$-trivialized by $\varphi_j$, from $a_i \in S_{d_i}$ we obtain that $S_{d_i}$ contains the entire $(U_j,\varphi_j)$-arc containing $a_i$, i.e.,
$\Arc_{b_{i,j}}(U_j, \varphi_j) \subseteq S_{d_i}$. This implies
\[
W_{i,j} = T_{a_i}(\Arc_{b_{i,j}}(U_j,\varphi_j))
 \subseteq T_{a_i}(S_{d_i}).
\]

Property $(iv)$ follows from Lemma \ref{lem:dist_tangentS2} applied to the open ball $B_{j}$; see Convention \ref{conv:balls}, and recall also
the convention $\frac{\beta}{\infty} \coloneqq 0$.
\qed(\ref{cla:W-prop})

To finish the proof of Theorem~\ref{thm:sts-to-Lipstrats},
it remains to show that, given vector spaces $W_{i,j}$ satisfying Claim~\ref{cla:W-prop}, we can find $K$-vector spaces $V_{i,j}$ satisfying the same properties and moreover such that $V_{i,j+1}\subseteq V_{i,j}$ for all $i\leqslant j$. This is the content of Lemma \ref{lem:new vectors}.
\end{proof}

\begin{lem}\label{lem:new vectors}
Let $m\leqslant n$. Suppose that we have elements $\lambda_1 < \dots < \lambda_{m+1}$ of the value group, where we allow $\lambda_{m+1} = \infty$, integers
$n \geqslant d_0 \geqslant \dots \geqslant d_m \geqslant 0$ and
$K$-vector spaces $W_{i,j} \subseteq K^n$ for $0 \leqslant i \leqslant j \leqslant m$ with the following properties:
\begin{enumerate}[label=(\roman*)]
 \item $\dim(W_{i,j}) = d_j$.
    \item $W_{i,j}\subseteq W_{i,i}$,
    \item If $i < j$, then $\Delta(W_{i+1,j}, W_{i,j})\le\frac{\lambda_{i+1}}{\lambda_{j+1}}$, with the convention $\frac{\lambda_{i+1}}{\infty} = 0$.
\end{enumerate}
Then there exist $K$-vector spaces $V_{i,j}$ satisfying (i)--(iii)
such that $V_{i,i} = W_{i,i}$, and
such that moreover, we have $V_{i,j+1} \subseteq V_{i,j}$ for every $0 \leqslant i \leqslant j \leqslant n-1$.
\end{lem}
\begin{proof}
By part (1) of Lemma \ref{lem:distance}, and since $\frac{\lambda_{i+1}}{\lambda_{j+1}}<1$, (iii) implies $\res(W_{i,j}) = \res(W_{i',j})$ for any $i, i', j$. Set $\bar W_j\coloneqq \res(W_{i,j})$ for some (any) $0\leqslant i\leqslant m$. Since $W_{j,j+1} \subseteq W_{j,j}$, we have 
\begin{equation}\label{eq:inclusion-W}
\bar W_m \subseteq \dots \subseteq \bar W_0.    
\end{equation}
Let $\{e_1,\ldots,e_n\}$ be the standard basis of $K^n$ and set $\bar{e}_j\coloneqq \res(e_j) \in \bar K^n$. 

\begin{cla}\label{cla:standard_basis} Without loss of generality, we may suppose that 
\[
\bar W_j = \Span_{\bar K}(\bar{e}_1,\ldots,\bar{e}_{d_j}),
\]
\end{cla}

Let $\{b_1,\ldots,b_{d_0}\}$ be a basis of $\bar{W}_0$. By \eqref{eq:inclusion-W}, we may assume that $\{b_1,\ldots,b_{d_j}\}$ is a basis for $\bar{W}_j$ for any $j\geqslant 0$. Choose $\bar{M}\in \GL_n(\bar K)$ sending $b_i$ to $\bar{e}_i$ and choose a lift $M\in \GL_n(\valring)$ of $\bar{M}$ (that is, $\res(M)=\bar{M}$). By Part (2) of Lemma \ref{lem:distance}, the $K$-vector spaces $W_{i,j}'\coloneqq M(W_{i,j})$ satisfy properties $(i)-(iii)$ and the statement in Claim~\ref{cla:standard_basis}. Assuming we can find $K$-vector spaces $V_{i,j}'$ as in the statement of the lemma with respect to the $K$-vector spaces $W_{i,j}'$, the reader may verify that $V_{i,j}\coloneqq M^{-1}(V_{i,j}')$ also satisfy the conclusion of the lemma with respect to the original $K$-vector spaces $W_{i,j}$. \qed(\ref{cla:standard_basis})

\begin{cla}\label{cla:even_more_standard_basis}
$W_{i,j}$ has a basis $(w_{i,j,k})_{k \leqslant d_j}$ satisfying
\[
w_{i,j,k} \in e_k +
(\{0\}^{d_j} \times \maxid^{n-d_j})
\]
\end{cla}

First, choose an arbitrary preimage $w'_k \in W_{i,j}$ of $\bar e_k \in \bar W_j$. Those $w'_k$ form a basis of $W_{i,j}$, and we have
$w'_k \in e_k + \maxid^{n}$.
Let $A$ be the $(d_j\times d_j)$-matrix consisting of the first $d_j$ rows of the matrix $B \coloneqq (w'_1 \mid \dots \mid w'_{d_j})$ (i.e., with columns $w'_k$). Then $v(A - I) < 1$ (where $I$ is the identity matrix); this implies that $A$ is invertible.
The columns of the matrix $BA^{-1}$ are the desired basis $w_{i,j,k}$.
\qed(\ref{cla:even_more_standard_basis})

\begin{cla}\label{cla:vector_distance}
$v(w_{i+1,j,k} - w_{i,j,k}) \leqslant \frac{\lambda_{i+1}}{\lambda_{j+i}}$.
\end{cla}
This follows from Lemma~\ref{lem:nice-dist}.
\qed(\ref{cla:vector_distance})

We partition the basis $\{w_{i,j,1}, \dots w_{i,j,d_j}\}$ of $W_{i,j}$ into sets
\[
A_{i,j,\ell} = \{w_{i,j,d_{\ell+1}+1}, \dots, w_{i,j,d_\ell}\}
\]
for $\ell = j, \dots, m$ (and where $d_{m+1} \coloneqq 0$) and we define, for each $i,j$, the desired vector space $V_{i,j}$ to be generated by
\[
A_{i,m,m} \cup A_{i,m-1,m-1} \cup \dots \cup A_{i,j,j}.
\]
Let us verify that these have all the desired properties.

By definition, we have $V_{i,j+1} \subseteq V_{i,j}$.

Since the first $d_j$ columns of the vectors generating $V_{i,j}$ form a matrix whose residue is the identity matrix, these vectors are in particular linearly independent, and we have $\dim(V_{i,j}) = d_j$.

We have $V_{i,j} \subseteq W_{i,i}$, since
$A_{i,k,k} \subseteq W_{i,k} \subseteq W_{i,i}$, and hence $V_{i,i} = W_{i,i}$ (since they have the same dimension).

Finally, $\Delta(V_{i+1,j}, V_{i,j})\le\frac{\lambda_{i+1}}{\lambda_{j+1}}$ follows from Claim~\ref{cla:vector_distance}, applied to the basis vectors of $V_{i+1,j}$ and $V_{i,j}$.
\end{proof}

\section{Arc-wise analytic t-stratifications}\label{sec:analytic}

In this section, we restrict to algebraically closed valued fields and introduce our strongest notion of stratifications: arc-wise analytic t-stratifications. After proving their existence (Section~\ref{sec:an_exist}), we show that they are in particular \lt-stratifications (Section~\ref{sec:aats-to-lipts}). We will see in Section \ref{sec:hierarchy} that arc-wise analytic t-stratifications probably lie at the top of a whole hierarchy of strenghtenings of \lt-stratifications.

For the notion of arc-wise analytic t-stratification to make sense, one needs a notion of ``analytic function''. We will therefore assume in the entire Section~\ref{sec:analytic} that our valued field $K$ is equipped with an analytic structure in the sense of \cite{CLip}; see Section~\ref{sec:an_struct}. In addition, we will assume that $K$ is algebraically closed. We believe that this last assumption is unnecessary, but the framework of analytic functions is much smoother in the algebraically closed case, and even formulating the ``right'' definition of arc-wise analytic t-stratifications gets more tricky in the non-algebraically closed case.

By passing through arc-wise analytic t-stratifications, we thus obtain that
\lt-stratifications exist in arbitrary algebraically closed valued fields of equi-characteristic $0$ with
analytic structure. (Note that this includes the ``trivial'' analytic structure, where the language ist just the pure valued field language.)
However, we do not only believe that the assumption of algebraically closedness is unnecessary, but more generally that (in contrast to arc-wise analytic t-stratifications) \lt-stratifications exist in arbitrary $1$-h-minimal valued fields of equi-characteristic $0$ (see Question \ref{q:hensel}).

\subsection{Fields with analytic structure}
\label{sec:an_struct}

When $K$ is a complete valued field, one can expand the language $\Lval$ to an ``analytic language'' by adding function symbols for some power series converging on
$\valring^n$. In an elementary extension $K'$ of $K$, although the interpretation of such function symbols does not necessarily yield a power series, they still share many properties that actual power series have. This idea (introduced by L. van den Dries and J. Denef in \cite{denef-vdd-88})  led R.~Cluckers and L.~Lipshitz in \cite{CLip}  to define an abstract notion of ``analytic structure'' on $K$, which yields an ``analytic language'' on $K$ expanding $\Lval$. As already mentioned above, in the entire Section~\ref{sec:analytic}, we fix such an analytic structure $\mathcal A$ on $K$ and we work in the corresponding analytic language.

The definition of analytic structure is quite general but rather technical, so we remit the reader to \cite{CLip} for definitions and details. Here, we only recall the features that we will need, and moreover only under the additional assumption that $K$ is algebraically closed.

\begin{itemize}
\item 
Certain definable subsets of $K^n$ are called \emph{domains} \cite[Definition~5.2.2]{CLip}. In particular, any product of balls $B_i \subsetneq K$ is a domain. There are also notions of \emph{open} and \emph{closed} domains. (A closed domain is essentially an abstract version of a rational domain in the sense of rigid geometry.)
\item If $X$ is either an open domain or a closed domain, then
we have a well-defined \emph{ring of analytic functions} $\mathbf{A}(X)$ consisting of certain definable (with parameters) maps from $X$ to $K$ \cite[Definition~5.2.2 and Corollary~5.2.14]{CLip}.
\item The analytic functions on $\valring^m$ are given as follows:
From the analytic structure $\mathcal{A}$ on the $K$, one obtains certain $K$-algebras of abstract power series $A^{\dagger}_{m,0}(K) \subseteq K\otimes_{\valring} \valring[[\xi_1,\ldots,\xi_m]]$ and injective\footnote{We use the remark above \cite[Definition~4.5.6]{CLip} and \cite[Theorem~4.5.4]{CLip} to assume without loss that the maps $\sigma_{m,n}$ are injective.} $K$-algebra homomorphisms $f\mapsto f^\sigma$ from $A^{\dagger}_{m,0}(K)$ to the ring of $K$-valued maps on $\valring^m$.
The ring $\mathbf{A}(\valring^m)$ is the image of $A^{\dagger}_{m,0}(K)$ under $\sigma$. (This corresponds to $\mathcal{O}_K^\sigma(\varphi)$ in \cite{CLip}, where $\varphi$ is a formula defining $\valring^m$.)
\item Similarly, the ring $\mathbf{A}(\maxid^m)$ of analytic functions on $\maxid^m$ is the image of a $K$-algebra $A^{\dagger}_{0,m}(K) \subseteq K\otimes_{\valring} \valring[[\rho_1,\ldots,\rho_m]]$ under a homomorphism $f \mapsto f^\sigma$.
\end{itemize}

It is worthwhile to note that every henselian field $K$ in the language $\Lval$ constitutes (up to interdefinability) already an example of field with analytic structure. Other interesting examples are gathered in \cite[Section 4.4]{CLip}. 

\

We now summarize some basic facts about analytic functions which are probably clear to anybody working with them, but since they are not written explicitly in \cite{CLip}, we sketch their proofs. (We only mention statements about $\mathbf{A}(\valring^n)$; variants of those could also be obtained for $\mathbf{A}(\maxid^n)$.)

\begin{lem}\label{lem:basic-analyticII} Let $g = f^\sigma\colon \valring^n\to K$ be in $\mathbf{A}(\valring^n)$, where $f=\sum_{i\in\mathbb{N}^n} a_i\xi^i\in A^\dagger_{n,0}(K)$.
Then we have the following:
\begin{enumerate}
\item
We have $\max \{v(g(x)) \mid x \in \valring^n\} = \max \{v(a_i) \mid i \in \N^n\}$; in particular, those maxima exist.
\item
We have good Taylor approximation: For every $r \in \N$, we have
\[
v(g(x) - \sum_{|i| < r} a_i x^i) \leqslant v(x)^r \cdot \max\{v(a_i) : |i| \geqslant r\}.
\]
\item 
For each $j \in \N^n$, the $j$-th partial derivative $g^{(j)}$ of $g$ is well-defined (in the sense that it exists and is independent of the order in which one takes the derivatives with resepct to the different variables),
and we have $a_j = g^{(j)}(0)/j!$.
\item For every $j \in \N^n$, the formal $j$-th partial derivative $f^{(j)}$ of the power series $f$
lies in $A^{\dagger}_{n,0}(K)$, and $(f^{(j)})^\sigma$ is equal to the actual 
$j$-th partial derivative $g^{(j)}$ of $g$.
\end{enumerate}
\end{lem}

\begin{proof}[Sketch of proof]
(1) Existence of $\max \{v(a_i) \mid i \in \N^n\}$ is by definition of 
$A^{\dagger}_{m,0}(K)$ and \cite[Remark~4.1.10]{CLip}. After rescaling, we may assume that it is equal to $1$. Then $g(\valring^n) \subseteq \valring$ (by definition again), so we already obtained ``$\leqslant$''.
Now set $I \coloneqq \{i \in \N^n : v(a_i) = 1\}$ and
$f_0 \coloneqq \sum_{i \in I}a_i\xi^i$. Then by the ``$\leqslant$'' part applied to $f - f_0$, we get $\res(f^\sigma(x)) = \res(f_0^\sigma(x))$ for all $x$. Since $f_0$ is a polynomial (by definition of $A^{\dagger}_{n,0}(K)$), we have
$\res(f_0^\sigma(x)) = \sum_{i \in I}\res(a_i)\res(x)^i$. Since not all $\res(a_i)$ are $0$, there exists an $x \in \valring^n$ where this is not $0$, and hence such that $v(g(x)) = 1$.

(2) Using Weierstraß division repeatedly, we can write $f = \sum_{|i| < r} a_i\xi^i + \sum_{|j|=r} \xi^j\tilde f_j$ for some $\tilde f_j \in A^{\dagger}_{n,0}(K)$, and where each coefficient of each of the power series $\tilde f_j$ is equal to some $a_i$ with $|i| \geqslant r$.
By applying $\sigma$, we obtain that in the Taylor approximation, the error term is bounded by $v(x)^r\cdot \max \{v(\tilde f^\sigma_j(x)) : |j| = r\}$.
Now apply (1) to each $\tilde f_j$ to get the desired bound.

(3)+(4) By a common induction on $|j|$, it suffies to prove (3) and (4) when $|j| \leqslant 1$. In that case, (3) follows 
directly from (2) (using $r = |j|+1$).
The first part of (4) holds since $A^\dagger_{n,0}(K)$ is closed under Weierstraß division, namely (without loss considering $j = (1, 0, \dots, 0)$), we have
\begin{equation}\label{eq:weidiff}
    f(x) - f(x_1 + y, x_2, \dots, x_n) = y\cdot f'(x,y) + f''(x)
\end{equation}
for some $f' \in A^\dagger_{n+1,0}(K)$ and some
$f'' \in A^\dagger_{n,0}(K)$, and $f^{(j)}(x) = f'(x, 0)$. Applying $\sigma$ to the entire Equation (\ref{eq:weidiff}) and then letting $y$ go to $0$ yields the second part of (4)
(using that $(f')^\sigma \in \mathbf{A}(\valring^{n+1})$ is continuous, which follows from the $|j|=1$ case of (3)).
\end{proof}

\begin{rem}\label{rem:an-to-strict}
Note that Lemma~\ref{lem:basic-analyticII} implies that any $g \in \mathbf{A}(\valring^n)$ is strictly $C^1$, namely using that the derivative of $g$ is continuous and that the convergence of the limit defining the derivative is uniform on small balls (by Lemma~\ref{lem:basic-analyticII} (2)).
From this, one deduces that analytic functions $g\colon X \to K$ on arbitrary open or closed domains $X$ are strictly $C^1$, e.g. by composing $g$ with an affine linear map sending $\valring^n$ to a small neighbourhood of any $x \in X$.
\end{rem}

We now list two easy consequences of the above facts, for future reference:

\begin{lem}\label{lem:our5.8-multi}
Let $g\colon \valring^n\to\valring^d$ be an analytic function (i.e., each coordinate is analytic).
\begin{enumerate}
    \item If $\rv(g(x))$ is constant on $\valring^n$, then $v(Dg(0)) < v(g(0))$.
    \item
      If, for any $x, x' \in \valring^n$, $v(Dg(x) - Dg(x')) < 1$, then
    $v(g(x)-g(0)-Dg(0)(x)) < v(x)^2$. 
\end{enumerate}
\end{lem}

\begin{proof}
Both statements can be proved for each coordinate funciton separately, so we may without loss assume $d = 1$. Write $g = f^\sigma$ for $f = \sum a_ix^i$.

(1)
Set $h(x) := g(x) - g(0)$. Since $\rv(g(x)) = \rv(g(0))$ (by assumption), we have $v(h(x)) < v(g(0))$ for all $x$.
Applying Lemma~\ref{lem:basic-analyticII} (1) to $h$ yields
$v(a_i) < v(g(0))$ for all $i \ne \underline{0}$. By Lemma~\ref{lem:basic-analyticII} (3), the valuations of the $a_i$ for $|i| = 1$ are equal to the valuations of the first partial derivatives of $h$. Thus
$v(\nabla g(0)) = v(\nabla h(0)) < v(g(0))$.

(2)
By Lemma~\ref{lem:basic-analyticII} (3), it suffices to show that $v(a_i) < 1$ for $|i| \geqslant 2$.
Each of those $a_i$ appears as a non-constant coefficient of the series corresponding to one of the first partial derivatives $g_j$ of $g$.
By assumption, $v(g_j(x) - g_j(0)) < 1$, so applying Lemma~\ref{lem:basic-analyticII} (1) to $g_j$ yields the desired result.
\end{proof}

\begin{lem}\label{lem:analytic-inverse}
Suppose that $g\colon B \to B'$ is an analytic bijection between two open or between two closed balls in $K^n$ of the same radius, such that $v(Dg) = 1$ on all of $B$ and
$\res(Dg)$ is constant on $B$.
Then the inverse of $g$ is also analytic.
\end{lem}

\begin{proof}
First assume that $B$ and $B'$ are open.
Without loss, $B = B' = \maxid^n$ and $g(0) = 0$. By composing with a matrix from $\GL_n(\valring)$, we may moreover assume that $Dg(0)$ is the identity matrix.
In particular, the $k$-th coordinate $g_k$ of $g$ is of the form $g_k = f^{\sigma}_k$, where
$f_k$ is of the form $f_k(\rho) = \rho_k + \sum_{|i| \geqslant 2} a_i \rho^i$ for some $a_i \in \valring$ (where $\rho = (\rho_1, \dots, \rho_n))$.
In particular, $f_1(\rho) - \rho'_1 \in \valring[[\rho, \rho'_1]]$ is regular in $\rho_1$ of degree $1$ in the sense of \cite[Definition~4.1.1 (ii)]{CLip}, so
by Weierstraß division,
we obtain $\rho_1 = q\cdot (f - \rho'_1) + r_1$ for some $q \in \valring[[\rho'_1, \rho]]$ and $r_1 \in \valring[[\rho'_1, \rho_2, \dots, \rho_n]]$.
For $k \geqslant 2$, set $r_k := f_k(r, \rho_2, \dots, \rho_n)$. Then the analytic function
$\tilde g := (r_1^\sigma, \dots, r_n^\sigma)\colon \maxid^n \to \maxid^n$ has the property that for any $x \in \maxid^n$ and for $y = g(x) \in \maxid^n$, we have $\tilde g(y_1, x_2, \dots, x_n) = (x_1, y_2, \dots, y_n)$. Since $D\tilde g(0)$ is again the idendity matrix (as an easy computation shows), we can repeat the same procedure for all the other variables, so that in the end, we obtain the desired inverse of $g$.

The case where $B$ and $B'$ are closed is very similar, but requires one additional step. We assume that $B = B' = \valring^n$, that $g(0) = 0$ and that $Dg(0)$ is the identity matrix.
As before, we obtain $g_k = f_k^\sigma$ for 
$f_k(\xi) = \xi_k + \sum_{|i| \geqslant 2} a_i \xi^i$ with $a_i \in \valring$ and we want to apply 
Weierstraß division to $f_1(\xi) - \xi'_1$. 
The additional difficulty is that to have regularity in $\xi_1$ of degree $1$, we now need
to have $v(a_i) < 1$ for all $i$ (since $\xi_1$ runs over the valuation ring). So suppose that we have $v(a_i) = 1$ for some $i$. Then also one of the first partial derivatives  $f' \coloneqq \partial f_1/\partial \xi_\ell$
has a coefficient of valuation $1$, so that Lemma~\ref{lem:basic-analyticII} (1) contradicts the assumption that $\res((f')^\sigma)$ is constant.

Now we can apply Weierstraß division and the proof continues as in the open ball case.
\end{proof}

\subsection{Analytic rainbows and arc-wise analytic $t$-stratifcations}
\label{sec:an_exist}

As announced, we assume in this section that $K$ is algebraically closed and endowed with an analytic structure.
We start by definining the notion of arc-wise analytic t-stratification. As usual $\mathbf B \subseteq K^n$ is an open ball, possibly equal to $K^n$.

\begin{defn}\label{def:arc-wise-strats}
An \emph{arc-wise analytic t-stratification} of $\mathbf B$ is a
partition $\mathcal{S}=(S_i)_{0 \leqslant i \leqslant n}$ of $\mathbf B$ into definable sets $S_i$ of dimension at most $i$ such that for every $d\in \{0,\ldots,n\}$ and every ball $B \subseteq S_{\geqslant d}$, there are a risometry $\varphi\colon B\to B$ and a $d$-dimensional sub-vector space $U \subseteq K^n$ such that 
\begin{enumerate}
    \item $\cS$ is $U$-trivialized by $\varphi$ on $B$, and 
    \item for every $b \in B/U$,
    there exists an analytic map $g\colon \pi(B) \to B$, for some coordinate projection $\pi\colon K^n \to K^d$,
    such that the image $g(\pi(B))$ is equal to $\Arc_b(U, \varphi)$, such that $v(Dg(x)) = 1$ for all $x \in \pi(B)$ and such that $\res(Dg(x)) \in \mathrm{Mat}_{n\times d}(\bar K)$ does not depend on $x$.
\end{enumerate}
\end{defn}

In the second condition, the coordinate projection plays no role at all; ``$\pi(B)$'' is just a quick way to write ``some $d$-dimensional ball which is open if and only if $B$ is, and which has the same radius as $B$''.

The first condition above simply states that $(S_0,\ldots, S_n)$ is $d$-trivial on $B$, and therefore, as the name suggests, arc-wise analytic t-stratifications are t-stratifications. In fact, we will later show in Theorem \ref{thm:aats-to-lipts} that they are also \lt-stratifications. 

There are two natural ways to define what it means that an arc-wise analytic t-stratification reflects a map $\chi\colon \mathbf B \to \RV\eq$: On the one hand, we could use the definition coming from t-stratifications, namely, that for every $d$ and every $B \subseteq S_{\geqslant d}$, $(\cS, \chi)$ has to be $d$-trivial on $B$; On the other hand, we could additionally impose that the map $\varphi$ $d$-trivializing $(\cS, \chi)$ satisfies the above Condition~(2). As in the case of \lt-stratification, by Remark~\ref{rem:refl}, these the two definitions are equivalent.  

The following lemma yields an alternative characterization of Condition~(2) of Definition~\ref{def:arc-wise-strats}. It has the disadvantage of being (seemingly) coordinate-dependent, but it is often more handy to work with.

\begin{lem}\label{lem:arc-wise-strats}
Suppose that $\varphi\colon B \to B$ is a definable risometry and $U \subseteq K^n$ is a $d$-dimensional sub-vector space. Let $\pi\colon K^n \to K^d$ be an exhibition of $\bar U := \res(U)$. Then the following are equivalent, for each $b \in B/U$:
\begin{enumerate}
    \item There exists an analytic map $g\colon \pi(B) \to K^n$ whose image $g(\pi(B))$ is equal to $\Arc_b(U, \varphi)$, such that $v(Dg) = 1$ on all of $\pi(B)$ and such that $\res(Dg)$ is constant on $\pi(B)$.
    \item The map $\arc_b\colon \pi(B)\to K^{n-d}$ from Definition~\ref{def:g_b} (whose graph is $\Arc_b(U, \varphi)$) is analytic.
\end{enumerate}
\end{lem}

\begin{proof}
If $\arc_b$ is analytic, we can take $g(x) := (x, \arc_b(x))$.
Indeed, using that $\varphi$ is a risometry, one obtains that for every $x \in \pi(B)$, $v(Dg(x)) = 1$ and $\res(Dg(x))$ is inverse of the restricted projection $\pi_{|\bar U}\colon \bar U \to \bar K^d$.

For the other direction, let $g$ be given. Then the composition $h := \pi \circ g\colon \pi(B) \to \pi(B)$ is analytic.
Using that $v(Dg) = 1$ everywhere and
that $v(\pi(x - x'))= v(x - x')$ for $x, x' \in \arc_b(U, \varphi)$, we deduce that $v(Dh) = 1$ everywhere.
Moreover, since $Dh$ is a sub-matrix of $Dg$, we obtain that $\res(Dh)$ is constant (since $\res(Dg)$ is constant).
We can therefore apply
Lemma~\ref{lem:analytic-inverse}, to obtain that the inverse of $h$ is also analytic.
This implies that $\arc_b = \pi^\vee \circ g\circ h\1$ is analytic.
\end{proof}

In order to show the existence of arc-wise analytic t-stratifications on algebraically closed valued fields with analytic structure, we will exploit that in \cite{halupczok2014a}, a very closely related construction has already been done:

\begin{defn}[Analytic rainbow]\label{def:an-t-strat} We say that a t-stratification $\mathcal{S}=(S_i)_i$ of $\mathbf{B}$ has an \emph{analytic rainbow} if for any $d\in \{0,\ldots,n\}$, for any fiber $C$ of the rainbow $\rho_{\cS}$ of $\mathcal{S}$ with $C \subseteq S_d$ and for any exhibition $\pi\colon K^n \to K^d$ of $\affdir(C)$, there exists an open domain $X\subseteq K^d$
containing $\pi(C)$ such that
$C$ is the graph of the restriction to $\pi(C)$ of an analytic function from $X$ to $K^{n-d}$.  
\end{defn}

The following is essentially Part (1) of \cite[Lemma 5.9]{halupczok2014a}:

\begin{lem}\label{lem:an-t-strat}
Let
$\chi\colon \mathcal B \to \RV\eq$ be an $A$-definable map, for some set of parameters $A \subseteq K \cup \RV\eq$. Then there exists an $A$-definable t-stratification $\mathcal{S}=(S_i)_i$ with analytic rainbow and reflecting $\chi$.
\end{lem}

\begin{proof}
\def\someindex{$_{n-1}$}
We wish to apply \cite[Lemma 5.9]{halupczok2014a}. To this end,
first note that the assumption
\cite[Hypothesis~2.21\someindex]{halupczok2014a} in that lemma is satisfied by
\cite[Proposition 5.12]{halupczok2014a}.
Choose an $A$-definable t-stratification $\cS$ reflecting $\chi$ (using \cite[Theorem~4.12]{halupczok2014a}) and apply \cite[Lemma 5.9]{halupczok2014a} to find an ``improved'' t-stratification 
$\cS'$. Since $\cS'$ reflects $\cS$, it also reflects $\chi$ (by \cite[Corollary 4.19]{halupczok2014a}).
The statement of \cite[Lemma 5.9]{halupczok2014a} does not say that $\cS'$ has an analytic rainbow, but note that Part 1 of its proof does.
\end{proof}

To obtain the existence of arc-wise analytic t-stratifications, we will show:

\begin{prop}\label{prop:an-t-to-arc-t} Every t-stratification with analytic rainbow is an arc-wise analytic t-stratification. 
\end{prop}

The proof of Proposition~\ref{prop:an-t-to-arc-t} relies on the following analytic variant of \cite[Lemma 3.19]{halupczok2014a}. recall that $\Tub_C$ is the natural tubular neighbourhood of a rainbow fiber $C$; see Definition~\ref{def:EC} and its characerization in Remark~\ref{rem:Tub_C}.

\begin{lem}\label{lem:analytic-3.19} Let $\mathcal{S}=(S_i)_i$ be a t-stratification of $\mathbf{B}$ with analytic rainbow. Let $C\subseteq S_{d}$ be a $\rho_\mathcal{S}$-fiber for some $d\leqslant n$. Let $\bar{U}=\mathrm{affdir}(C)$ (so $\dim(\bar{U})=d$), $\pi\colon K^n\to K^d$ be an exhibition of $\bar{U}$. Then there is a $\code{\Tub_C}$-definable $\bar{U}$-translater $(\alpha_{q,q'})_{q,q'\in \pi(\Tub_C)}$ on $\Tub_C$ reflecting $\rho_\mathcal{S}$ such that, in addition, for each $x\in \Tub_C$, the map $g_x\colon \pi(C)\to K^{n-d}$ defined by
\[
q\mapsto \pi^\vee(\alpha_{\pi(x),q}(x))
\]
is analytic. 
\end{lem}

In other words, the lemma imposes that the arcs of $\alpha\family$ are graphs of analytic functions.

Before proving that lemma, we show how it implies Proposition~\ref{prop:an-t-to-arc-t}.

\begin{proof}[Proof of Proposition~\ref{prop:an-t-to-arc-t} using Lemma~\ref{lem:analytic-3.19}]
Let $\mathcal{S}=(S_i)_i$ be a t-stratification with analytic rainbow on $\mathbf{B}$. Let $B\subseteq \mathbf{B}$ be a maximal ball disjoint from $S_{<d}$ and such that $B\cap S_d\neq\emptyset$, set $\bar{U}\coloneqq \tsp_B(\cS)$, and let $\pi\colon K^n\to K^d$ be an exhibition of $\bar{U}$.
Pick a $\rho_{\mathcal{S}}$-fiber $C \subseteq S_d$ satisfying $C \cap B \ne \emptyset$. 
Since the radius of $B$ is equal to $\lambda_C$, we obtain $B = \Tub_C \cap \pi^{-1}(\pi(B))$. 
By Lemma \ref{lem:analytic-3.19} applied to $C$, there exist a $\bar{U}$-translater $(\alpha_{q,q'})_{q,q'\in \pi(C)}$ on $\Tub_C$ reflecting $\rho_{\mathcal S}$ whose
arcs are graphs of analytic functions. Since for $q \in \pi(B)$, $\pi^{-1}(q) \cap B = \pi^{-1}(q) \cap \Tub_C$,
the family $(\alpha_{q,q'})_{q,q'\in \pi(B)}$ is a $\bar U$-translater on $B$.

Finally, we choose a lift $U\subseteq K^n$ of $\bar U$ and we use Lemma~\ref{lem:translater}
to turn the translater $(\alpha_{q,q'})_{q,q'\in \pi(B)}$ into a risometry $\psi\colon B\to B$
such that the $(U,\psi)$-arcs are exactly the arcs of $(\alpha_{q,q'})_{q,q'\in \pi(B)}$.
This in partiular implies that $\psi$ $U$-trivializes $\rho_{\mathcal{S}}$ and that its arcs are graphs of analytic functions, as 
required by Definition~\ref{def:arc-wise-strats}.
\end{proof}

As a direct consequence of Lemma \ref{lem:an-t-strat} and Proposition~\ref{prop:an-t-to-arc-t}, we now obtain the existence of arc-wise analytic t-stratifications:

\begin{thm}\label{thm:arc-existence} 
Suppose $K$ is an algebracally closed fields of equi-characteristic $0$ with analytic structure in the sense of \cite{CLip}. Let
$\chi\colon \mathbf B \to \RV\eq$
be an $A$-definable map, for some open ball $\mathbf{B}\subseteq K^n$ (possibly equal to $K^n$) and some set of parameters $A \subseteq K \cup \RV\eq$. Then there exists an $A$-definable arc-wise analytic t-stratification $\mathcal{S}=(S_i)_i$ reflecting $\chi$. \qed 
\end{thm}

It remains to prove Lemma~\ref{lem:analytic-3.19}. Its proof closely follows the proof of \cite[Lemma 3.19]{halupczok2014a}; one only needs to additionally verify that analyticity of the rainbow of $\cS$ implies analyticity of the map $g_x$ for each $x\in \Tub_C$, and this verification is straightforward. More precisely,
the translation between our lemma and
\cite[Lemma 3.19]{halupczok2014a} is as follows:
Our $n-d$ corresponds to $n$ in \cite[Lemma 3.19]{halupczok2014a}; our sets $S_i\cap \Tub_C$ correspond to $S_{i-d}$ in \cite[Lemma 3.19]{halupczok2014a}); the sets $Q$ and $Q'$ of \cite[Lemma 3.19]{halupczok2014a} coincide in our case and correspond simply to $\pi(C)$. 
We nevertheless give a full proof of Lemma~\ref{lem:analytic-3.19}.

\begin{proof}[Proof of Lemma~\ref{lem:analytic-3.19}]
We do an induction on $n-d$.
When $n=d$ (the base case of the induction), $\pi$ is the identity map on $K^n$. In particular, $\Tub_C=\pi(\Tub_C) = \pi(C) = C$ and each map $\alpha_{q,q'}$ is defined as the map sending $q$ to $q'$. Given $x\in C$, the map $g_x\colon \pi(C)\to K^0$ (where $K^0$ is identified to a point) is trivially analytic. 

For the inductive case, given $q\in\pi(C)$, let $F_q\coloneqq \pi^{-1}(q)\cap \Tub_C$.
By Proposition~\ref{prop:EC-translater}, there exists a translater $(\beta_{q,q'})_{q,q'}$ on $\Tub_C$ reflecting $\rho_{\mathcal{S}}$.
For $q,q'\in \pi(C)$, define $\alpha_{q,q'}\colon F_q\cap S_{d} \to F_{q'}\cap S_{d}$ as the restriction of $\beta_{q,q'}$. Since each arc of $\alpha\family$ is just a rainbow fiber (contained in $S_d$), those arcs are graphs of analytic functions, so it remains to extend the family $\alpha\family$ to  
maps $F_q \to F_{q'}$ and show the desired properties.

Fix some $q_0 \in \pi(C)$, fix a ball $B_{q_0}$ maximally contained in $\pi^\vee(F_{q_0}\cap S_{>d})$ and set $D_{q_0}\coloneqq (\pi^\vee)\1(B_{q_0})\cap F_{q_0}$. For any $q\in \pi(C)$, let 
\[
D_{q}\coloneqq \beta_{q_0,q}(D_{q_0}) \text{ and } B_{q} \coloneqq  \pi^\vee(D_q). 
\]
Since $\beta_{q_0,q}$ is a risometry, and $\pi^\vee$ is injective on $F_q$, $B_q$ is also a ball maximally contained in $\pi^\vee(F_{q}\cap S_{>d})$.
Let $\widetilde \lambda$ be the open radius of $B_q$. Note that it does not depend on $q$,
and moreover, if $v(q - q') < \widetilde \lambda$, then $B_q = B_{q'}$.
Indeed, for $x \in D_q$ and $x' \coloneqq \beta_{q, q'}(x)$,
we have that $D_q$ and $D_{q'}$ are $\pi$-fibers of $B(x, <\widetilde\lambda)$ and $B(x',<\widetilde\lambda)$, respectively; now
$v(x-x') = v(q - q')$
(since $\pi$ is an exhibition of $\affdir C \supseteq \dir(x - x')$),
so that $B(x,<\widetilde\lambda) = B(x',<\widetilde\lambda)$.
In particular, 
\[
B(x,<\widetilde\lambda) = \bigcup_{q' \in B(q,<\widetilde\lambda)} D_{q'}.
\]
Let $\widetilde{d}>d$ be minimal such that $D_q \cap S_{\widetilde{d}} \ne \emptyset$.
This does not depend on $q$.
Choose any rainbow fiber $\widetilde C \subseteq S_{\widetilde d}$ which contains some $x \in D_q$, for some $q$ and set $\widetilde{\bar U} \coloneqq \affdir \widetilde C$. 

Let us first show that $\bar{U}$ is contained in $\widetilde{\bar U}$. Let $W$ be a one dimensional subspace of $\bar{U}$. By the definition of $\affdir(\widetilde{C})$, it suffices to find $c,c'\in \widetilde{C}$ such that $W= \dir(c-c')$. Fix any $c\in \widetilde{C}$ and set $q=\pi(c)$. Choose $q'\in\pi(\bar{U})$ such that $\dir(q-q')=\pi(W)$ and set $c'\coloneqq \beta_{q,q'}(c)$. Note that $c'\in \widetilde{C}$ since $\beta_{q,q'}$ reflects the rainbow $\rho_{\mathcal{S}}$. Using that $v(c-c')=v(q-q')$, Lemma \ref{lem:dir2} implies that
\[
\pi(\dir(c-c'))=\dir(\pi(c-c'))=\dir(q-q')=\pi(W). 
\]
By condition (3) of Definition \ref{def:translater}, $\dir(c-c')\subseteq \bar{U}$, and since $\pi_{|\bar{U}}$ is bijective, we may conclude that $\dir(c-c')=W$.

Let $\widetilde \pi\colon \bar K^n \to \bar K^{\widetilde d}$ be an exhibition of $\widetilde{\bar U}$ such that $\pi = \theta \circ \widetilde \pi$ for some coordinate projection $\theta\colon \bar K^{\widetilde d} \to \bar K^d$.

Let $\lambda_{\widetilde C}$ and $\Tub_{\widetilde C}$ be defined as in Definition~\ref{def:EC}. Note that $\lambda_{\widetilde C} = \widetilde\lambda$. Indeed, $\lambda_{\widetilde C} \leqslant \widetilde\lambda$ by the choice of $\widetilde \lambda$ and $D_q$ being maximal disjoint from $S_d$. For the other direction, pick $x \in \widetilde C \cap D_q$, and use $d$-triviality on
$B(x,<\lambda_C)$ to argue that the distance from $x$ to $S_d$ can be computed within the fiber $F_q$. 

Also note that
\[
\Tub_{\widetilde C}
\supseteq \bigcup_{q\in \pi(C)} \{q\}\times D_q
.
\]
(Actually, one even has equality.)

By induction, there is a $\code{\Tub_{\widetilde C}}$-definable translater $(\widetilde{\alpha}_{\widetilde{q},\widetilde{q}'})_{\widetilde{q},\widetilde{q}' \in \widetilde\pi(\widetilde C)}$
on $\Tub_{\widetilde C}$ such that for $x\in \Tub_{\widetilde{C}}$, the function
$\widetilde g_x\colon \widetilde q\mapsto \widetilde\pi^\vee(\widetilde\alpha_{\widetilde\pi(x),\widetilde q}(x))$
is analytic.

We now define $\alpha_{q,q'}(x)$ for $x \in D_q$ as follows. For $q \in \pi(C)$, denote by $a_q$ the average of all elements of $S_d \cap F_q$ with distance $\widetilde\lambda$ to $D_q$.
Note that the set $a_q - D_q$ does not depend on $q$, i.e., given any $q,q'$ and $x \in D_q$, we have $x - a_q + a_{q'} \in D_{q'}$. Thus we can define
\[
\alpha_{q,q'}(x)
\coloneqq \widetilde\alpha_{\widetilde\pi(x),\widetilde\pi(x - a_q + a_{q'})}(x).
\]
Note that for fixed $x$, the map $h_x\colon q \mapsto \widetilde\pi(x - a_{\pi(x)} + a_{q})$ is analytic. Therefore, so is the composition
\[
g_x\colon q
\mapsto
\pi^\vee( \widetilde\alpha_{\widetilde\pi(x),\widetilde\pi(x - a_{\pi(x)} + a_{q})}(x))
= \pi^\vee (h_x(q) , \widetilde g_x(h_x(q))).
\]
It remains to show that the family $(\alpha_{q,q'})_{q,q'\in\pi(C)}$ is a translater reflecting $\rho_{\mathcal{S}}$. Condition (1) of Definition~\ref{def:translater} follows directly from the definition of $\alpha_{q,q'}$ and the fact that $(\widetilde{\alpha}_{\widetilde{q},\widetilde{q}'})_{q,q' \in \widetilde\pi(\widetilde C)}$ is a translater. Condition (2) is a straight forward computation, namely as follows. Let $q,q'$ and $q''$ be elements in $\pi(C)$, let $x$ be in $F_q$ and set $x'=\alpha_{q,q'}(x)$. Then 
\begin{align*}
\alpha_{q',q''}(\alpha_{q,q'}(x)) & = \widetilde{\alpha}_{\widetilde{\pi}(x'), \widetilde{\pi}(x'-a_{q'}+a_{q''})}(x') \\
 & =  \widetilde{\alpha}_{\widetilde\pi(x - a_q + a_{q'}), \widetilde{\pi}(x'-a_{q'}+a_{q''})}(x')  \\
 & = \widetilde{\alpha}_{\widetilde\pi(x - a_q + a_{q'}), \widetilde{\pi}(x-a_{q}+a_{q''})}(x') \\
  & = \widetilde{\alpha}_{\widetilde\pi(x - a_q + a_{q'}), \widetilde{\pi}(x-a_{q}+a_{q''})}(\alpha_{q,q'}(x)) \\
& = \widetilde{\alpha}_{\widetilde\pi(x - a_q + a_{q'}), \widetilde{\pi}(x-a_{q}+a_{q''})}(\widetilde{\alpha}_{\widetilde{\pi}(x),\widetilde\pi(x - a_q + a_{q'})}(x)) \\
& \overset{(\star)}{=} \widetilde{\alpha}_{\widetilde{\pi}(x),\widetilde\pi(x - a_q + a_{q''})}(x) = \alpha_{q,q''}(x), 
\end{align*}
where the identity $(\star)$ follows from Condition (2) for the family $\widetilde{\alpha}\family$. For condition (3), let $q,q'$ be distinct elements in $\pi(C)$, $x$ be in $F_q$ and set $x'=\alpha_{q,q'}(x)$. Then by Condition (3) for the family $\widetilde{\alpha}\family$
\[
\dir(x'-x)=\dir(\widetilde{\alpha}_{\widetilde{\pi}(x),\widetilde\pi(x - a_q + a_{q'})}(x)-x) \subseteq \widetilde{\bar U}. 
\]
Since $\widetilde{\pi}_{|\widetilde{\bar U}}$ is a bijection and $\bar{U}\subseteq \widetilde{\bar U}$, to obtain $\dir(x' - x) \subseteq \bar U$, it suffices to show that $\widetilde{\pi}(\dir(x'-x))\subseteq \widetilde{\pi}(\bar{U})$. First, note that 
\[
\widetilde{\pi}(\dir(x'-x)) = \dir(\widetilde{\pi}(x'-x)) = \dir(\widetilde{\pi}(a_{q'}-a_q))) =
\widetilde{\pi}(\dir(a_{q'}-a_q)), 
\]
where the first and the last identity follow from Lemma \ref{lem:dir2}, and the second one holds
since $\widetilde{\pi}(x'-x) = \widetilde{\pi}(a_{q'}-a_q)$ by definition of $\alpha_{q,q'}$. It is thus enough to show that $\dir(a_{q'} - a_q) \subseteq \bar U$. Suppose $a_q$ (resp.\ $a_{q'}$) is the average of $r$ points $c_1,\ldots,c_r\in F_q$ (resp.\ $c_1',\ldots, c_r'\in F_{q'}$). Since $\pi(c_i'-c_i)=q-q'$ and $\pi$ is an exhibition of $\bar{U}$, $v(c_i'-c_i)=v(q'-q)$
for all $1\leqslant i\leqslant r$. Moreover, by condition (3) for $\beta_{q',q}$, we have that $\dir(c_i'-c_i)\subseteq \bar{U}$ for each $1\leqslant i\leqslant r$. Thus, by Lemma \ref{lem:dir}, we have that $\rv(c_1'-c_1)=\rv(c_i'-c_i)$ for all $1\leqslant i\leqslant r$. Now Lemma~\ref{lem:average} yields $\rv(c_1'-c_1)=\rv(a_{q'}-a_q)$, and 
by Fact \ref{fact:res-vs-rv}, we have
\[
\dir(a_{q'}-a_q) = \dir(c_1'-c_1)\subseteq \bar{U}.  \qedhere
\]
\end{proof}

\subsection{Arc-wise analytic t-stratifications are \lt-statifications}
\label{sec:aats-to-lipts}

We continue assuming that $K$ is algebraically closed and that it is equipped with an analytic structure in the sense of \cite{CLip}.
As the title indicates, the purpose of this section is to show the following result.  

\begin{thm}\label{thm:aats-to-lipts} Let $\mathcal{S}$ be an arc-wise analytic t-stratification of an open ball $\mathbf{B}\subseteq K^n$. Then $\mathcal{S}$ is a \lt-stratification. 
\end{thm}

\begin{proof}
Fix an open or closed ball $B\subseteq \mathbf{B}$ disjoint from $S_{<d}$.
Through this proof we follow once more
Convention \ref{conv:balls} to treat the open and closed case in parallel; in particular, $B$ is closed$^*$. Using Lemma \ref{lem:arc-wise-strats}, let $\bar U$ be a $d$-dimensional~$\bar K$-vector space, $U$ be a lift of $U$ and $\varphi\colon B\to B$ be a risometry such that
\begin{enumerate}
    \item $(S_0,\ldots, S_n)$ is $U$-trivialized by $\varphi$, and 
    \item for an exhibition $\pi\colon K^n\to K^d$ of $U$, and for each $b\in B/U$, the map $\arc_b\colon \pi(B)\to K^{n-d}$ is analytic. 
\end{enumerate}
Remark \ref{rem:an-to-strict} ensures that $\arc_b$ is strictly $C^1$. By Proposition \ref{prop:SU-equivalence}, it remains to show 
\begin{equation}\label{eq:taylor2}
v(\arc_b(u_1)-\arc_b(u_0)-D\arc_{b_0}(u_0)(u_1-u_0))\lsnew\frac{v(u_1-u_0)\cdot\max\{v(u_0-u_1), v(b_0 - b)\}}{\radc^*(B)},
\end{equation}
for $b,b_0\in B/U$ and $u_0,u_1\in \pi(B)$, with $u_1\neq u_0$. We will achieve this by a couple of reductions: 

\medskip

\emph{Step  1:} We may suppose that $B$ is closed. Indeed, suppose for contradiction that \eqref{eq:taylor2} fails for some open ball $B$, i.e., for certain $b_0, b, u_1, u_0$, the left hand side of
\eqref{eq:taylor2} is strictly bigger than the right hand side; let $\delta$ be the quotient of the left hand side by the right hand side. Pick a closed ball $B' \subseteq B$ big enough such that $b_0, b \in B'/U$ and $u_1, u_2 \in \pi(B')$, and such that moreover, $\delta\cdot \radc B' > \rado B$.
Now the assumption that \eqref{eq:taylor2} holds in $B'$ (for the same $b_0, b, u_1, u_0$) yields a contradiction.

\medskip

\emph{Step  2:} We may suppose $B=\valring^n$ and $u_0=0$. Indeed, by scaling and translating, we may send $u_0$ to $0$ and $B$ to $\valring^n$. Apply this transformation to the entire setting (i.e., to $\mathbf B$ and $\mathcal S$),
use \eqref{eq:taylor2} on $\valring^n$ on the transformed setting; then 
the inverse transformation gives back the result on $B$.

\medskip

\emph{Step  3:} We may suppose $b_0=b$. Consider the analytic function $h(u)\coloneqq \arc_b(u)-\arc_{b_0}(u)$. By the definition of $\arc_b$ (see Definition \ref{def:g_b}), since $\varphi$ is a risometry, $\rv(h(u))$ is constant. In particular, 
\[
v(h(u))=v(\arc_b(u)-\arc_{b_0}(u))=v(b-b_0).
\]
By Lemma~\ref{lem:our5.8-multi} (1), we obtain $v(D h(0))<v(b-b_0)$. Assuming the result holds for $b$ alone, we have 
\begin{equation}\label{eq:step4.1}
    v(\arc_b(u_1)-\arc_b(0)-D\arc_{b}(0)(u_1))\lsnew v(u_1)^2.
\end{equation}
Thus, for $b_0\neq b$
\begin{align*}
    &v(\arc_b(u_1)-\arc_b(0)-D\arc_{b_0}(0)(u_1)) \\
    = & v(\arc_b(u_1)-\arc_b(0)-D\arc_{b}(0)(u_1) - (D\arc_{b_0}(0)-D\arc_{b}(0))(u_1)) \\
    = & v(\arc_b(u_1)-\arc_b(0)-D\arc_{b}(0)(u_1) - D h(0)(u_1))\\
    \leqslant & \max\{v(u_1)^2, v(D h(0))\cdot v(u_1)\}
    \lsnew v(u_1)\cdot\max\{v(u_1), v(b-b_0)\},
\end{align*}
which completes this step. 

\medskip

\emph{Step  4:} By the previous steps, \eqref{eq:taylor2} corresponds to 
\begin{equation}\label{eq:simplified-SUT1}
v(\arc_b(u_1)-\arc_b(0)-D\arc_b(0)(u_1))\lsnew v(u_1)^2.     
\end{equation}
Using that $\pi$ exhibits $\bar U$
and that $\varphi$ is a risometry, we obtain that the matrix $D\arc_b$ has coefficients in $\valring$ and its residue
$\res(D\arc_b)$ is constant on $\pi(B)$.
Thus the inequality \eqref{eq:simplified-SUT1} follows directly from Lemma \ref{lem:our5.8-multi} (2).   
\end{proof}

As a corollary of Theorems \ref{thm:arc-existence} and \ref{thm:aats-to-lipts}, we obtain that \lt-stratifications exist:

\begin{cor}\label{cor:t2-exist} 
Suppose $K$ is an algebraically closed valued field of equi-characteristic $0$ with analytic structure in the sense of \cite{CLip}. Let
$\chi\colon \mathbf B \to \RV\eq$ be an $A$-definable map, for some open ball $\mathbf{B}\subseteq K^n$ (possibly equal to $K^n$) and some set of parameters $A \subseteq K \cup \RV\eq$. Then there exists an $A$-definable \lt-stratification $\mathcal{S}=(S_i)_i$ reflecting $\chi$. \qed 
\end{cor}

\section{The critical value function of a t-stratificaton}\label{sec:exponent}

We will now associate an invariant to every t-stratification which we call the ``critical value function'', and which captures all kinds of asymptotic distances near singularities. As an example, consider the cusp curve, given by $x^3 = y^2$. When $x$ is close to $0$, the valuation of the vertical distance between the two branches is $v(x)^{3/2}$. This function $\lambda \mapsto \lambda^{3/2}$ is part of the critical value function (see Example~\ref{ex:contact}).
More generally, in the case of curves, it captures all such asymptotic distances between any two branches near every singularity. In higher dimension, one cannot speak of branches anymore, but it still captures similar kind of information, like the asymptotic behaviour of the radius of the ``trumpet'' defined by $z^2+x^2=x^3$; see Figure~\ref{fig:trumpet}.

The main result of this section (Theorem~\ref{thm:crit} and Corollary~\ref{cor:crit}) is that, under some reasonable assumptions about the language, all this asymptotic behaviour together is still described by a finite amount of information. More precisely,
the critical value function, which contains all that information, is a ``piecewise monomial geometric object'' in (a cartesian power of) the value group, i.e., it can be described by finitely many monomial functions
with rational exponents.

We start by defining the critical value function, which, in reality, is a ``multi-function'', i.e., potentially taking several values.
For the moment, we assume that we are in the setting of Section~\ref{sec:setting}; we will impose stronger conditions on the language only where necessary.
In the entire Section~\ref{sec:exponent}, let $\mathbf B \subseteq K^n$ be an open ball (possibly equal to $K^n$).

\begin{conv}
We let $\Gamma_\infty$ be $\Gamma\cup\{\infty\}$ where $\infty$ is an element bigger than every element in $\Gamma$. As a convention, for a subset $X\subseteq \mathbf B$, we set $v(X-\emptyset)=\infty$. In what follows we let 
\[
\Delta_n\coloneqq\{(\lambda_0,\ldots, \lambda_{n-1})\in \Gamma_\infty^n: \lambda_0\geqslant \lambda_1\geqslant\cdots\geqslant \lambda_{n-1}\}.
\]
\end{conv}

\begin{defn}\label{def:critpt} Let $\mathcal{S}$ be a t-stratification of $\mathbf B$.
\begin{enumerate}
    \item 
Given $c \in \mathbf B$, we define the set of \emph{critical values of $\cS$ at $c$} as 
\[
\Crit_{\cS}(c) \coloneqq \{v(x - c) :
x \in K^n \setminus \{c\}\text{ and $\cS$ is not $\dir(x-c)$-trivial on $B(x, <v(x-c))$}\}.
\]
(We will sometimes say that the critical value $v(x-c)$ is \emph{witnessed} by $x$.)
\item
The \emph{critical value function} of $\cS$ is the function $f^{\cS}\colon \Delta_n \to \mathcal{P}(\Gamma^\times)$ (taking values on the power set of $\Gamma^\times$) given by
\[
f^{\cS}(\lambda_0, \dots, \lambda_{n-1})\coloneqq
\bigcup_{\substack{c \in K^n\\\forall i< n\colon v(c-S_{\leqslant i})=\lambda_i}} \Crit_{\cS}(c)
\]
\end{enumerate}
\end{defn}

\begin{figure}
    \centering
    \includegraphics{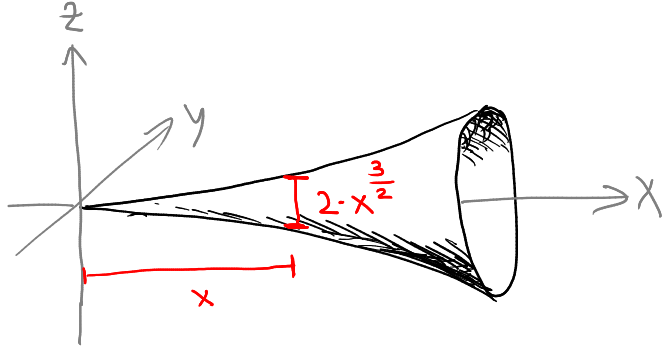}
    \caption{The critical value function of the the trumpet, defined by $y^2+z^2=x^3$, captures the (asymptotic) valuation of the diameter, as a function of the valuation of $x$: $\lambda \mapsto \lambda^{3/2}$}
    \label{fig:trumpet}
\end{figure}

Note that (1) has some flavour of Whitney's Condition (b): That condition imposes that if $c$ is a point in some stratum $S_d$ and $x$ is nearby, in some higher dimensional stratum $S_{d'}$, then the tangent space to $S_{d'}$ at $x$ should nearly contain the one-dimensional space spanned by $x-c$. The critical values of $c$ are those values $v(x-c)$ where this fails.

\begin{rem}\label{rem:all-points}
For every $c \in \mathbf B$, we have $v(c - S_0) \geqslant v(c - S_{\leqslant1}) \geqslant \dots \geqslant v(c - S_{\leqslant n-1})$,
so every critical value of every point appears as some value of $f^{\cS}$.
\end{rem}

\begin{rem}
There are some ``trivial'' critical values:
If, given $(\lambda_0, \dots, \lambda_{n-1}) \in \Delta_n$,
there exists any point $c \in \mathbf B$ satisfying $v(c - S_{\leqslant i}) = \lambda_i$ for each $i$, then
$f^{\cS}(\lambda_0, \dots, \lambda_{n-1})$ contains 
$\{\lambda_0, \dots, \lambda_{n-1}\} \cap \Gamma^\times$.
More precisely, if some point $c \in \mathbf B$ has distance $\lambda_d \in \Gamma^\times$ to $S_{\leqslant d}$, then $\lambda_d \in \Crit_{\cS}(c)$. Indeed,
set $B \coloneqq B(c, \leqslant \lambda_d)$,
set $d' \coloneqq \dim(\tsp_B(\cS))$
let $\pi\colon K^n \to K^{d'}$ be an exhibition of $\tsp_B(\cS)$, let $F \subseteq B$ be the $\pi$-fiber (intersected with $B$) containing $c$ and pick $x \in S_{d'} \cap F$. Then $v(c - x) = \lambda_d$, and using that $S_{d'} \cap F$ is finite, one obtains that $\cS$ is not $\dir(c-x)$-trivial on $B(x, <\lambda_d)$.
\end{rem}

\begin{defn}\label{def:orthogonality} Recall that one says that the value group $\Gamma$ and the residue field $\bar K$ are \emph{orthogonal}, if every definable subset $X\subseteq \bar K^m\times \Gamma^n$ is a finite union of sets of the form $A\times B$ for $A\subseteq \bar K^m$ and $B\subseteq \Gamma^n$ definable sets. 
\end{defn}

\begin{rem}\label{rem:orthogonality}
If $\Gamma$ and $\bar K$ are orthogonal, then every definable function $f\colon X\subseteq \bar K^m\to \Gamma^n$ has finite image. Assuming in addition that $\Lx$ contains an angular component map $\ac$, both $\Gamma$ and $\bar{K}$ are stably embedded. Indeed, using $\ac$ there is a definable bijection $\RV^\times \to \bar K^\times \times \Gamma^\times$, and the result follows from orthogonality and the fact that $\RV$ is stably embedded (see Section~\ref{sec:modth}). 

Note that in $\Lval$, $\Gamma$ and $\bar K$ are orthogonal, and assuming the existence of an angular component is harmless by passing to an elementary extension.
\end{rem}

We can now precisely state the main result of this section:

\begin{thm}\label{thm:crit}
Suppose that we are in the setting from Section~\ref{sec:setting}.
Suppose moreover that the language $\Lx$ we work with contains an angular component map $\ac$ and that the value group and the residue field are orthogonal. Let $\mathbf B \subseteq K^n$ be an open ball (possibly equal to $K^n$) and let $\cS$ be a $t$-stratification of $\mathbf B$. Then, there is a positive integer $N=N(\cS)$ such that for every $\lambda \in \Delta_n$, the cardinality of $f^{\cS}(\lambda)$ is bounded by $N$.
\end{thm}

To obtain that the critical value function is a piecewise monomial geometric object as promised at the beginning of this section, one needs to further assume that the value group $\Gamma^\times$ is a pure divisible ordered abelian group. Since the ``graph'' $G$ of $f^{\cS}$ is a definable subset of $\Delta_n \times \Gamma^\times$, $G$ is a union of finitely many graphs of piecewise monomial functions with rational exponents, as promised. Definition \ref{def:monomial} and Corollary \ref{cor:crit} make this claim precise.

\begin{defn}\label{def:monomial}
A \emph{monomial function} from a subset $A  \subseteq \Gamma_\infty^n$ to $\Gamma^\times$
is a function of the form $(\lambda_0, \dots, \lambda_{n-1}) \mapsto \mu\cdot \prod_{i \in I} \lambda_i^{r_i}$, for some $\mu \in \Gamma^\times$ and some $r_i \in \Q$, and where $I$ is the set of all those coordinates $k$ such that the projection of $A$ to the $k$-th coordinate is contained in $\Gamma^\times$ (so that the product is always well-defined).
\end{defn}

\begin{cor}\label{cor:crit}
Suppose that we are in the setting of Theorem~\ref{thm:crit} and that additionally, the value group $\Gamma^\times$ is a pure divisible ordered abelian group. (In particular, this is true if we work in a pure algebraically closed valued field or more generally in an algebraically closed valued field with analytic structure in the sense of Section~ \ref{sec:analytic}.)

The graph of $f^{\cS}$ can be written as a union of graphs of finitely many monomial functions
$f\colon A \to \Gamma^\times$,
where $A \subseteq \Delta_n$ is
a subset of $B := \{\infty\}^{n_1} \times (\Gamma^\times)^{n_2} \times \{0\}^{n_3}$ (where $n_1 + n_2 + n_3 = n$)\footnote{No permutation of the coordinates of $B$ is needed, since we are only looking at subsets of $\Delta_n$, where coordinates are ordered.}
given by a finite conjunction of conditions of the form
$g_{j}(\lambda) \leqslant 1$ and $g_{j}(\lambda) < 1$ for monomial functions $g_{j}$ on $B$.
\end{cor}

\begin{proof} 
Since $\Gamma$ is stably embedded (Remark \ref{rem:orthogonality}), the function $f^{\cS}$ is definable in the pure language of ordered groups, and the result
follows by quantifier elimination of divisible ordered abelian groups. 
\end{proof}

Before proving Theorem~\ref{thm:crit}, we give an example:

\begin{exam}\label{ex:contact} Let $K$ be an algebraically closed valued field,
fix some co-prime natural numbers $a > b>1$, and consider the definable set $X\subseteq K^2$ given by
\[
X\coloneqq \{(x,y) \in K^2 \mid x^a = y^b\}.
\]
Set $S_0\coloneqq\{(0,0)\}$, $S_1\coloneqq X \setminus S_0$ and $S_2\coloneqq K^2\setminus (S_0\cup S_1)$.
We let the reader check that $\mathcal{S}\coloneqq\{S_0,S_1, S_2\}$ is a t-stratification reflecting $X$ (even an arc-wise analytic one). Let us compute the valuative invariants $f_d^{\cS}$ for $d=0,1,2$. As a preliminary computation, note that for $(x,y) \in S_1$, we have $v(y) = v(x)^{a/b}$. From this, one deduces that
$v((x,y)) = v(x)$ if $v((x,y)) \leqslant 1$ and $v((x,y)) = v(y)$ if $v((x,y)) \geqslant 1$. 

We now start by computing $f^{\cS}(0,0)$.
Since $S_0 = \{(0,0)\}$, this is equal to $\Crit_{\cS}(0,0)$, so we need to determine all $\mu$ for which there exists $(x,y)\in K^2$ such that $\mu=v((x,y))$ and $S_1$ is not $\dir((x,y))$-trivial on $B((x,y), <\mu)$.

The point $(1,1)$ lies in $S_1$, and
using $a \ne b$, one can check that $S_1$ is not $\dir((1,1))$-trivial on $B((1,1)$, so $1$ is a critical value of $(0,0)$. We claim that this is the only critical value at $(0,0)$. To see this, we do not need to consider points $(x,y)$ for which $B((x,y),<\mu)$ is disjoint from $S_1$, since those are trivial in every direction. We can therefore restrict our attention to
$(x,y) \in S_1$ with $v((x,y)) = \mu \ne 1$. We then either have $\mu < 1$ and $\dir((x,y)) = \bar K \times \{0\}$ or $\mu > 1$ and $\dir((x,y)) = \{0\} \times \bar K$, and in both cases, one can verify that $S_1$ is $\dir((x,y))$-trivial on the ball. Thus indeed $f^{\cS}(0,0) = \{1\}$.

The most interesting part of the critical value function is $f^{\cS}(\lambda_0, 0)$ for $\lambda_0 > 0$, i.e., we wish to determine the (non-trivial) critical values of points $(x, y) \in S_1$;
satisfying $v(x,y) = \lambda_0$.
As suggested by the preliminary computation, we do a corresponding case distinction on whether $\lambda_0$ is less than $1$ or bigger than $1$.

Suppose that $\lambda_0 \leqslant 1$. Then $S_1$ also contains the point $(x, \zeta_b y)$, for some $b$-th root of unity $\zeta_b \ne 1$, and this point witnesses that $v(\zeta_b y - y) = v(y) =  \lambda_0^{a/b}$ is a critical value of $(x,y)$. Moreover, $1$ is also a critical value,
witnessed e.g.\ by the point $(1,1)$ as in the case $\lambda_0 = 0$ when $\res(x) \ne 1$.
If $\lambda_0 > 1$, then in a similar way, $(\zeta_a x, y)$ (for $\zeta_a \ne 1$ an $a$-th root of unity) witnesses that $v(x) = \lambda_0^{b/a}$ is a critical value of $(x,y)$. (In this case, $1$ is not a critical value.)
We leave it to the reader to verify that these are all critical values.

Finally, one should compute $f^{\cS}(\lambda_0, \lambda_1)$ for $\lambda_0 \geqslant \lambda_1 > 0$. This is a bit cumbersome but contains no interesting new information: There may be non-trivial critical values, but those are ``inherited'' from a nearby point on $S_1$. We therefore do not present the computation.

In total, we get, for $\lambda_0 \geqslant \lambda_1$:
\[
f^{\mathcal S}(\lambda_0,\lambda_1) =
\begin{cases}
\{1\}
&
\lambda_0 = \lambda_1 = 0
\\
\{1, \lambda_0, \lambda_0^{a/b}\}
&
0 < \lambda_0 < 1,
\lambda_1 = 0
\\
\{1,\lambda_0, \lambda_0^{a/b}, \lambda_1\}
&
0 < \lambda_0 < 1,
0 < \lambda_1 < \lambda_0^{a/b}
\\
\{1,\lambda_0,  \lambda_1\}
&
0 < \lambda_0 < 1,
\lambda_0^{a/b}\leqslant \lambda_1 \leqslant \lambda_0
\\
\{\lambda_0, \lambda_0^{b/a}\}
&
1 \leqslant \lambda_0 < \infty,
\lambda_1 = 0
\\
\{\lambda_0,  \lambda_0^{b/a}, \lambda_1\}
&
1 \leqslant \lambda_0 < \infty,
0 < \lambda_1<  \lambda_0^{b/a}
\\
\{\lambda_0,   \lambda_1\}
&
1 \leqslant \lambda_0 < \infty,
\lambda_0^{b/a}\leqslant \lambda_1 \leqslant \lambda_0
\\
\emptyset
&
\lambda_0 = \infty.
\end{cases}
\]
\end{exam}

\begin{exam}
Let $X= \{(x,y,z) \in K^3 : y^2+z^2=x^{3}\}$ be the trumpet from Figure~\ref{fig:trumpet}. The intersection of $X$ with the translate of the $yz$-plane by some $x \in \valring \setminus \{0\}$ is a circle of diameter $2x^{3/2}$; we claim that the function $v(x) \mapsto v(2x^{3/2})$ is captured by the critical value function, no matter which t-stratification we choose. Indeed,
if we pick one point on this circle, say $p_x := (x, 0, x^{3/2})$ then the opposite point $(x, 0, -x^{3/2})$ witnesses that $v(2x^{3/2})$ is a critical value at $p_x$. Any t-stratification reflecting $X$ has $(0,0,0) \in S_0$, and since $S_0$ is finite, for $x$ sufficiently small, we have $v(p_x - S_0) = v(x)$. Thus, 
$(\lambda_0,\lambda_1,0) \mapsto \lambda_0^{3/2}$ is part of the critical value function at least for small $\lambda_0$ (and for some suitable $\lambda_1$).
\end{exam}

Let us now get to the proof of Theorem~\ref{thm:crit}. We first need to do some preliminary work. Since some of the intermediate results might be of independent interest, we will not impose the assumptions of the theorem all the way through, but only assume what-ever is needed. (We do however always assume that we are in the setting of Section~\ref{sec:setting}.)

The first thing to note is that the set of critical values at a single point is finite. This is already known (and not very difficult):

\begin{lem}[{\cite[Theorem 7.4]{halupczok2014a}}]\label{lem:finim} Suppose $\bar K$ and $\Gamma$ are orthogonal. Then for any t-stratification $\cS$ of $\mathbf B$ and for any $c \in \mathbf B$, the set $\Crit_{\cS}(c)$ is finite.
\end{lem} 

Next, note that in Definition~\ref{def:critpt} (2), the union over all $c \in K^n$ in reality only runs over all rainbow fibers:

\begin{rem}\label{rem:critrfiber}
If $C$ is a rainbow fiber of a t-stratification $\cS$, then $\Crit_{\cS}(c)$ is the same for all $c \in C$:
By Lemma~\ref{lem:rainbow-charac}, given $c, c' \in C$ there exists a risometry $\varphi$ fixing each $S_i$ setwise and sending $c$ to $c'$. This implies that if $\mu$ is a critical value for $c$, witnessed by $x$ (with $v(x-c) = \mu$), then $\varphi(x)$ witnesses that $\mu$ is also a critical value for $c'$.
The same argument also shows that for every $i$, we have $v(c - S_{\leqslant i}) = v(c' - S_{\leqslant i})$, and this is therefore also equal to $v(C - S_{\leqslant i})$.
\end{rem}

This motivates the following definition.

\begin{defn}\label{def:fiber-d} Let $\mathcal{S}$ be a t-stratification of $\mathbf{B}$. For $\lambda=(\lambda_0,\ldots, \lambda_{n-1})\in\Delta_n$, we let $R_\lambda^\mathcal{S}$ be the collection of all rainbow fibers $C$ such that $v(C-S_{\leqslant i})=\lambda_i$ for each $i< n$.
\end{defn}

As in Remark~\ref{rem:all-points}, one sees that 
the sets $R_\lambda^{\mathcal{S}}$ form a partition of the set of all rainbow fibers, when $\lambda$ runs over $\Delta_n$.

As the last missing ingredient to prove Theorem~\ref{thm:crit}, we will show that each set $R_\lambda^{\mathcal{S}}$ can be parametrized by the residue field. For this ingredient, we will need the following slightly modified version of \cite[Lemma 2.1.3]{clu-hal-rid}.

\begin{lem}\label{lem:average-RF} Suppose the language includes an angular component map $\ac$. Let $X_q\subseteq K^n$ be a $\emptyset$-definable family of finite sets, where $q$ runs over some $\emptyset$-definable set $Q$ in an arbitrary possibly imaginary sort. Then there exists a $\emptyset$-definable family of injective maps $f_q\colon X_q \to \bar K^m$ (for some $m$) which moreover has the following property: if, for some $q,q'\in Q$ there exists a risometry $\varphi\colon K^n\to K^n$ sending $X_{q}$ to $X_{q'}$, then for every $x\in X_{q}$
\[
f_q(x) = f_{q'}(\varphi(x)). 
\]
\end{lem}

\begin{proof}
Using elimination of $\exists^\infty$ (see point (3) of Section \ref{sec:modth}), we can bound the cardinality $\#X_q$ and then assume that it is constant.
We do an induction over $\#X_q$.

If $X_q$ is always a singleton or empty, we can define $f_q$ to be always constant. Otherwise,
the lemma is obtained by repeatedly taking averages of the elements of $X_q$ and subtracting. More precisely, setting $a_q \coloneqq \frac1{\#X_q}\sum_{x \in X_q} x$, and $\lambda_q\coloneqq \max\{v(x-x') : x,x'\in X_q\}$, we get that the map
\[
\hat f_q\colon X_q\to \bar K, x \mapsto
\begin{cases}
\ac(x - a_q) & \text{if } v(x - a_q) = \lambda_q\\
0 & \text{if } v(x - a_q) < \lambda_q
\end{cases}
\]
is not constant on $X_q$.
Therefore, each fiber $\hat f^{-1}_q(\xi)$ of $\hat f_q$ (for $\xi \in \hat f_q(X_q)$) has cardinality less than $X_q$, so
by induction, we obtain a definable family of
injective maps $g_{q,\xi}\colon \hat f^{-1}_q(\xi) \to \bar K^m$. Now set $f_q(x) \coloneqq (\hat f_q(x), g_{q,\hat f_q(x)}(x))$. To show the last statement of the lemma, let $\varphi\colon K^n\to K^n$ be a risometry sending $X_{q}$ to $X_{q'}$, for $q,q'\in Q$. Without loss of generality, by scaling and translating, we may suppose that $\lambda_q=1$ and $a_q=a_{q'}=0$. Since $\varphi$ is a risometry, $\lambda_{q'}=1$. This implies that $X_q$ and $X_{q'}$ are contained in $\valring^n$, and that the restriction $\varphi_{|\valring^n}$ is a risometry on $\valring^n$. In particular, it induces a well-defined function $\bar\varphi\colon \bar K^n\to \bar K^n$ sending $\res(x)\to \res(\varphi(x))$, which in addition has to be a translation. 
An easy computation shows that $\bar\varphi$ sends $\res(a_q) = 0$ to $\res(a_{q'}) = 0$ so that it must be the identity; indeed:
\begin{align*}
\bar\varphi(\res(a_q))  & = \bar\varphi(\res(\frac1{\#X_q}\sum_{y \in X_q} y))) = \bar\varphi(\frac1{\#X_q}\sum_{y \in X_q} \res(y)) \\
                        & = \frac1{\#X_q}\sum_{y \in X_q} \bar\varphi(\res(y)) =\frac1{\#X_q} \sum_{y \in X_q} \res(\varphi(y)) \\
                        & = \frac1{\#X_{q'}}\sum_{y \in X_{q'}} \res(y) = \res(a_{q'}).  
\end{align*}
Therefore, for $x\in X_q$, 
\[
\hat f_q(x) = \res(x) = \bar\varphi(\res(x)) = \res(\varphi(x)) =  \hat f_{q'}(\varphi(x)).
\]
This shows moreover, that given $\xi\in \bar K^n$, $\varphi$ restricted to $\hat f_q^{-1}(\xi)$ is a bijection onto $\hat f_{q'}^{-1}(\xi)$. Thus, by induction, $g_{q,\hat f_q(x)}(x) = g_{q',\hat f_{q'}(\varphi(x))}(\varphi(x))$, which shows that $f_q(x) =f_{q'}(\varphi(x))$. 
\end{proof}

The following lemma is the promised statement about $R_\lambda^{\cS}$ being residue-field parametrized. We need to formulate a somewhat stronger statement to make the inductive proof work.
In the statement, we write ``$\bar K\eq$'' for all those sorts from $\RV\eq$ which are quotients of $\emptyset$-definable subsets of $\bar K^m$ by $\emptyset$-definable equivalence relations (for any $m$).

\begin{lem}\label{lem:RF-paramtrization} Suppose the language includes an angular component map $\ac$
and that $\bar K$ is stably embedded.
Let $\mathcal{S}_q$ be a $\emptyset$-definable family of t-stratifications of a $\emptyset$-definable family of balls $\mathbf{B}_q \subseteq K^n$ (where $q$ ranges over a $\emptyset$-definable subset $Q$ of a -possibly imaginary- sort). Let $\lambda=(\lambda_0,\ldots, \lambda_{n-1})\in \Delta_n$. Then there is a $\lambda$-definable family $E_q\subseteq \bar K\eq$ and a $\lambda$-definable family of definable bijections $f_q\colon R_\lambda^{\mathcal{S}_q}\to E_q$ which moreover satisfies the following:
if, for some $q,q'\in Q$, there exists a definable risometry $\varphi\colon K^n\to K^n$ sending $S_{q,i}$ to $S_{q',i}$ for each $i$, then for every $C \in R_\lambda^{\cS_q}$
\[
f_q(C) = f_{q'}(\varphi(C)). 
\]
\end{lem}

\begin{proof}
If $\lambda_0 = 0$, then $R_\lambda^{\mathcal{S}_q}$ consists of singletons of elements of $S_{q,0}$, so the present lemma holds by Lemma \ref{lem:average-RF}. So from now on suppose that $\lambda_0 > 0$.

Let $B_{\mathcal{S}_q}({\lambda_0})$ be the set of all maximal balls $B\subseteq S_{q,\geqslant 1}$ such that $v(B-S_{q,0})=\lambda_0$.
(Each rainbow fiber in $R^{\cS_q}_\lambda$ is contained in one of those balls.)
Let us show that there are a $\lambda$-definable family $D_q\subseteq \bar K\eq$ and a  $\lambda$-definable family of definable bijections $h_q\colon B_{\mathcal{S}_q}({\lambda_0})\to D_q$.
If $S_{q,0}$ is empty, then $B_{\mathcal{S}_q}({\lambda_0})$ is either empty or only contains $\mathbf{B}$ (depending on whether $\lambda_0 = \infty$), so now assume that $S_{q,0}$ is non-empty.
Using Lemma \ref{lem:average-RF}, let $e_q\colon S_{q,0}\to H_q\subseteq \bar K^m$ be a definable family of definable bijections (for some $m$). Define for $x\in \bigcup B_{\mathcal{S}_q}({\lambda_0})$ a map $f_x\colon H_{q} \to \bar K^n$ given by 
$f_x(e_q(y)) = \ac(x-y)$ for $y \in S_{q,0}$. Next, define $g\colon \bigcup B_{\mathcal{S}_q}({\lambda_0}) \to \bar K\eq$ sending $x$ to the code of $f_x$. (The code can indeed be take in in $\bar K\eq$, since $\bar K$ is stably embedded.) Each ball $B \in B_{\mathcal{S}_q}({\lambda_0})$ is a union of fibers of $g$. Quotient the image of $g$ by the equivalence relation relating two fibers if there is a ball in $B_{\mathcal{S}_q}({\lambda_0})$ containing both; we obtain the desired definable family $h_q$. Note that if $\varphi\colon K^n\to K^n$ is a risometry sending $\mathcal S_{q}$ to $\mathcal S_{q'}$, then $\varphi$ induces a bijection between the sets $B_{\mathcal{S}_q}({\lambda_0})$ and $B_{\mathcal{S}_{q'}}({\lambda_0})$, and moreover, $h_{q}(B)=h_{q'}(\varphi(B))$ for every $B\in B_{\mathcal{S}_q}({\lambda_0})$. Indeed, the last part of Lemma \ref{lem:average-RF} implies both that $H_q=H_{q'}$ and that for any $x\in \bigcup B_{\mathcal{S}_q}({\lambda_0})$, the functions $f_x$ and $f_{\varphi(x)}$ are equal. This shows that the codes of $f_x$ and $f_{\varphi(x)}$ are equal, which implies that $h_{q}(B)=h_{q'}(\varphi(B))$ for any $B\in B_{\mathcal{S}_q}({\lambda_0})$. 

If $n = 1$, then $R^{\cS_q}_\lambda = B_{\cS_q}(\lambda_0)$ (by $1$-triviality on each of those balls) and we are done. Otherwise, we continue using an induction on $n$, as follows:

For $B\in B_{\mathcal{S}_q}(\lambda_0)$, let
$(R_\lambda^{\mathcal{S}_q})_{|B}=\{C\in R_\lambda^{\mathcal{S}_q} : C\subseteq B\}$. Let us build a definable family of definable bijections $f_{q,B}\colon (R_\lambda^{\mathcal{S}_q})_{|B}\to G_{q,B}$ for a definable family of sets $G_{q,B}\subseteq \bar K\eq$.
So fix $B\in B_{\mathcal{S}_q}(\lambda_0)$. 
We may assume that $(R_\lambda^{\mathcal{S}_q})_{|B}$ is non-empty.
By 1-triviality on $B$,
there exists a coordinate projection $\pi = \pi_{q,B}\colon K^n \to K$ such that $\bar\pi(\tsp_B(\cS_q)) = \bar K$. If several such $\pi$ are possible, we pick the first one (so that $\pi_{q,B}$ is definable uniformly in $q,B$).
Let $\bar{U} \subseteq \tsp_B(\cS_q)$ be a one-dimensional sub-vector space exhibited by $\pi$. For $t\in \pi(B)$, set $\mathbf{F}_t \coloneqq \pi^{-1}(t) \cap B$ and $C_t \coloneqq C \cap \mathbf{F}_t$ for $C\in(R_\lambda^\mathcal{S})_{|B}$.
By $\bar U$-triviality, 
on each fiber $\mathbf{F}_t$, we have an induced t-stratification $\mathcal{S}_{q,B,t}$, and
the map $C \mapsto C_t$ is a bijection
$(R_\lambda^{\mathcal{S}_q})_{|B} \to R_{\lambda'}^{\mathcal{S}_{q,B,t}}$ for
some suitable $\lambda'=(\lambda'_1,\ldots, \lambda'_{n-1})\in \Delta_{n-1}$, namely
$\lambda'_i =  \lambda_i$ if $\lambda_i < \lambda_0$
and $\lambda'_i = \infty$ otherwise.

By induction, we have a definable family of definable bijections $g_{q,B,t}\colon R_{\lambda'}^{\mathcal{S}_{q,B,t}}\to E_{q,B,t}$ for a definable family of subsets $E_{q,B,t}\subseteq \bar K\eq$. 
Given any two $t, t' \in \pi(B)$, a translater witnessing $\bar U$-triviality on $B$ yields a risometry
$\alpha_{t,t'}\colon F_t\to F_{t'}$ sending $S_{q,B,t,i}$ to $\mathcal{S}_{q,B,t',i}$ for each $i$, so by the ``moreover'' part of the statement, given $C\in R_{\lambda}^{\mathcal{S}_q}$, $g_{q,B,t}(C_t)$ does not depend on $t$. Set  
\[
G_{q,B} \coloneqq E_{q,B,t} \text{ and }
f_{q,B}(C) \coloneqq g_{q,B,t}(C_t) \text{ for some (any) $t\in \pi(B)$ }.  
\]
To conclude, for $C\in R_\lambda^{\mathcal{S}_q}$, let $B_{C}$ be the unique element in $B_{\mathcal{S}_q}({\lambda_0})$ such that $C\subseteq B_C$ and set $E_q\coloneqq \bigcup_{\xi\in D_q} \{\xi\}\times G_{q,h_q^{-1}(\xi)}$, and $f_q\colon R_\lambda^{\mathcal{S}_q}\to E_q$ by 
\[
C \mapsto (h_q(B_C),f_{q,B_C}(C)).  
\]
By construction, $f_q$ is a bijection. It remains to show the last statement, so let $\varphi\colon K^n\to K^n$ be a risometry sending $S_{q,i}$ to $S_{q',i}$ for each $i$ and let $C \in R_\lambda^{\mathcal S_q}$ and $C' \coloneqq \varphi(C)$. By the paragraph above, $h_q(C)=h_{q'}(C')$.
Using that $\varphi$ restricts to a risometry from $B_C$ to $B_{C'}$, we obtain that the coordinate projections $\pi_{q,B_C}$ and $\pi_{q',B_{C'}}$ project to the same coordinate. After that,
Remark~\ref{rem:fibers_riso} ensures 
that the ``moreover'' part of the induction hypothesis (as used to define $g_{q,B_C,t}$) can be applied to
$(q,B_C,t)$ and $(q', B_{C'}, t')$ (for any $t, t'$ in the right balls),
so that we obtain
\[
f_{q,B_C}(C)=
g_{q,B_C,t}(C_t) = g_{q',B_{C'},t'}(C'_{t'})
=f_{q',B_{C'}}(C'), 
\]
and hence $f_q(C) = f_{q'}(C')$.  
\end{proof}

\begin{proof}[Proof of Theorem~\ref{thm:crit}]
By Lemma~\ref{lem:finim},
$\Crit_{\cS}(c)$ is finite for every $c$. By
Remark~\ref{rem:critrfiber}, $\Crit_{\cS}(c)$ only depends on the rainbow fiber containing $c$, so for each rainbow fiber $C$, we have a well-defined (finite) set $\Crit_{\cS}(C)$
(equal to $\Crit_{\cS}(c)$ for any $c \in C$), and
we have
\[
f^{\cS}(\lambda) = 
\bigcup_{C \in R_{\lambda}^{\mathcal{S}}} \Crit_{\cS}(C).
\]
Using that $\RV$ is stably embedded (see Section~\ref{sec:modth}), that we have an $\ac$-map and that $\bar K$ and $\Gamma$ are orthogonal, one deduces that $\bar K$ is stably embedded (see Remark~\ref{rem:orthogonality}), so that we can apply Lemma~\ref{lem:RF-paramtrization} and obtain that $R_{\lambda}^{\mathcal{S}}$ is in definable bijection to a definable set in $\bar K\eq$. Now orthogonality of $\bar K$ and $\Gamma$ implies that for $C$ running over $R_{\lambda}^{\mathcal{S}}$, there are only finitely many different sets $\Crit_{\cS}(C)$; hence the union is finite. The uniform bound follows by compactness. 
\end{proof}

\section{Future directions and conjectures}\label{sec:discussion}

In this section we discuss various directions for further work. Of particular interest for us is the potential relation with the Nash-Semple conjecture which we discuss in Section \ref{sec:nash}. 

\subsection{A hierarchy of examples}\label{sec:hierarchy}

The definition of \lt-stratification provided in Section \ref{sec:lipschitz}, involves a condition about approximations by the first Taylor polynomial with an error term of degree 2. This can best be seen in the equivalent form given by 
part (i) of Proposition \ref{prop:SU-equivalence}:
If we restrict Condition \eqref{eq:SUT1} to
the case $b = b_0$ and $u_2 = u_0$, that condition can be rewritten as  
\begin{equation}\label{eq:SUT1new}
v(\arc_b(u_1)-T^{\leqslant 1}_{\arc_b,u_0}(u_1-u_0))\lsnew\frac{v(u_1-u_0)^2}{\radc^*(B)},
\end{equation}
where $T^{\leqslant 1}_{\arc_b,u_0}(x) = \arc_b(u_0) + D\arc_b(u_0)(x)$ denotes the first Taylor polynomial of $\arc_b$ around $u_0$. It seems natural that one could impose similar conditions of higher order: For any $r \geqslant 2$, a 
t$^r$-stratification should impose that the functions $\arc_b$ satisfy 
\begin{equation}\label{eq:SUT1newr}
v(\arc_b(u_1)-T^{\leqslant r-1}_{\arc_b,u_0}(u_1-u_0))\lsnew\frac{v(u_1-u_0)^r}{\radc^*(B)}.
\end{equation}
We do not try to give a full definition of t$^r$-stratifications (for that, one would need to find the right condition when $b_0 \ne b$), but we provide some examples giving some evidence that these stratifications form a strict hierarchy.
Note also that an arc-wise analytic t-stratification should be a t$^r$-stratification for every $r$, essentially by Lemma~\ref{lem:basic-analyticII} (2).

\begin{exam}
Suppose $K$ is an algebraically closed valued field.
Fix $a_0 \in \maxid$ and
a ball $B_0 \coloneqq B(0, < \lambda_0) \subsetneq \maxid$ (with $0 < \lambda_0 < 1$).
Let $X \subseteq K^2$ be the graph of $f\colon K \to K$ given by
\[
f(x) = \begin{cases}
a_0x & x \in B_0\\
0 & x \notin B_0
\end{cases}
\]
We consider the t-stratification
$\mathcal{S}=\{S_0,S_1,S_2\}$ with $S_0 \coloneqq \{(1,0)\}$, $S_1 \coloneqq X \setminus S_0$ and $S_2 \coloneqq K^2 \setminus S_{\leqslant 1}$ and ask ourselves for which $r \geqslant 2$ it should be a t$^r$-stratification.

It suffices to consider what kind of $1$-triviality we have on the ball $B := \maxid^2$. As $1$-trivializing map, we can choose $\varphi\colon B \to B, (x,y) \mapsto (x,y+f(x))$, so that $\arc_0(x) = f(x)$ for $x \in \maxid$ (where we take $\pi\colon K^2 \to K$ to be the projection to the first coordinate). Let us verify for which $r$ this map satisfies \eqref{eq:SUT1newr}. Since $B$ is an open ball of radius $1$, the condition becomes
\[
v(f(u_1)-T^{\leqslant r-1}_{f,u_0}(u_1-u_0))\leqslant v(u_1-u_0)^r.
\]
Clearly, this can fail only if $u_0 \in B_0$ and $u_1 \in \maxid \setminus B_0$ or vice versa.
We have
\[
f(u_1)-T^{\leqslant r-1}_{f,u_0}(u_1-u_0)
=
\begin{cases}
0 - a_0u_1 & \text{in the case $u_0  \in B_0$}\\
a_0u_1 - 0 & \text{in the case $u_1  \in B_0$},
\end{cases}
\]
so in both cases, the condition becomes $v(a_0)v(u_1) \leqslant v(u_1-u_0)^r$.
Using that $v(u_1 - u_0) \geqslant \max\{v(u_1), \lambda_0\}$, we obtain that this is most difficult to satisfy when $v(u_1) = \lambda_0$, and in that case, so that the condition holds for all $u_0$ and $u_1$ if and only if $v(a_0)\leqslant \lambda_0^{r-1}$. In particular, for any $r$, we can pick $a_0$ and $\lambda_0$ is such a way that $\cS$ should be a t$^r$-stratification but not a t$^{r+1}$-stratification.
Note also that in the case $r = 2$, the condition $v(a_0)\leqslant \lambda_0$ implies the full condition \eqref{eq:SUT1} from Proposition~\ref{prop:SU-equivalence}, so that $\cS$ is really an \lt-stratification.
\end{exam}

In the above example, the stratification is definable only using the valuation. We give a second example with a ring-definable stratification. 

\begin{exam} Let $K$ be the real closed valued field $\R(\!(t^{\Q})\!)$ and consider $f(x) \coloneqq \frac1{1+x^2}$. It satisfies: 
\begin{enumerate}
    \item $v(f(x)) \leqslant 1$ for all $x \in K$; 
    \item $\max_{x \in \valring} v(f(x)) = 1$ (since there exists an $x \in \R$ with $f(x) \ne 0$).
\end{enumerate}
The $r$-th derivative $f^{(r)}$ is of the form $\frac{h(x)}{(1+x^2)^{r+1}}$ for some polynomial $h$ of degree at most $r$, so that one obtains the above two properties for $f^{(r)}$ in the same way as for $f$.

Fix some positive $\alpha \in \Q$ and define $g\colon K \to K, x \mapsto  t^\alpha f(t^{-1}x)$. Now set $S_0 = \{(0,1)\}$, let $S_1$ be the graph of $g$, and let $S_2$ be the remainder of $K^2$. Again, we want to determine for which $r$ this is a t$^r$-stratification. As in the previous example, the only relevant ball is $B = \maxid^2$ and we can pick $\varphi\colon B \to B, (x,y) \mapsto (x, y+g(x))$, so that $\arc_0(x) = g(x)$ for $x \in \maxid$. Hence we need to verify whether for all $u_0, u_1 \in \maxid$, we have
\[
v(g(u_1)-T^{\leqslant r-1}_{g,u_0}(u_1-u_0))\leqslant v(u_1-u_0)^r.
\]
We claim that this condition is equivalent to the condition that $v(g^{(r)}(u_0)) \leqslant 1$ holds for every $u_0 \in \maxid$. Indeed, $\Leftarrow$ holds by the usual Taylor approximation theorem (in the mean value form, where the error term is expressed in terms of $g^{(r)}$ at some point between $u_0$ and $u_1$), and $\Rightarrow$ holds sine for any $u_0$, and for $u_1$ sufficiently close to $u_1$, we have $v(g(u_1)-T^{\leqslant r-1}_{g,u_0}(u_1-u_0)) = v(g^{(r)}(u_0))v(u_1 - u_0)^r$.

Now notice that using the above properties (1) and (2) of $f^{(r)}$, and using the chain rule repeatedly to express $g^{(r)}$ in terms of $f^{(r)}$, we obtain that $\max_{x \in \maxid} v(g^{(r)}(x)) = v(t^{\alpha - r})$. Thus the condition of t$^r$-stratifications is satisfied if and only if $\alpha \geqslant r$.

Note that for $\alpha < 1$, $\varphi$ is not even a risometry, and indeed, $S_1$ is not $1$-trivial on $B$. On the other hand,
it is a risometry on $B$ if $\alpha \geqslant 1$ (so that we do obtain a t-stratification), and if $\alpha \geqslant 2$ then as in the previous example, one really obtains a \lt-stratification.
\end{exam}

\subsection{The Nash-Semple conjecture and t-stratifications}\label{sec:nash}

An important motivation for searching for stronger notions of stratifications is that they might help proving the Nash-Semple conjecture. We now sketch a program how arc-analytic t-stratifications might help. We first recall the conjecture.

Let $K$ be an algebraically closed field of characteristic 0. Let $X$ be a closed subscheme (reduced, separated, of finite type) of $\mathbb{A}_K^n$ of pure dimension $r$, let $X_{\mathrm{sing}}$ be its singular locus and let $X_{\mathrm{reg}}$ be the complement of $X_{\mathrm{sing}}$ in $X$. Write $G_{n,r}$ for the Grassmanian, i.e., $G_{n,r}(K)$ parametrizes the $r$-dimensional sub-vector spaces of $K^n$. 
Let $\eta\colon X_{\mathrm{reg}}\to X\times G_{n,r}$ be the morphism sending a closed point $x\in X_{\mathrm{reg}}$ to the pair $(x,T_{x}(X))$ where $T_{x}(X) \in G_{n,r}$ is the tangent space of $X$ at $x$. Let $X^*$ be the closure of $\eta(X_{\mathrm{reg}})$ in $X\times G_{n,r}$ and $\pi\colon X^*\to X$ be the coordinate projection to $X$. The pair $(X^*, \pi)$ is called the \emph{Nash modification} (or \emph{Nash blow up}) of $X$.

One can show that the construction of the pair $(X^*, \pi)$ is actually independent of the chosen immersion $X \subseteq \mathbb{A}_K^n$, so it naturally generalizes to general reduced separated schemes of finite type of pure dimension $r$.

For $n>0$, we inductively define $X^{(*n)}$, the $n$-th iteration of Nash modifications, as $X^{(*1)}\coloneqq X^*$ and $X^{(*(n+1))}\coloneqq (X^{(*n)})^*$. The Nash-Semple conjecture states that the iterative procedure of taking Nash modifications eventually provides a resolution of $X$, that is, there is some $n$ such that $X^{(*n)}$ is smooth.

Early results of A. Nobile \cite{nobile} show that $\pi\colon X^*\to X$ is an isomorphism if and only if $X$ is non-singular and that the conjecture holds for curves. 
Already for surfaces, the situation is a lot more difficult. Spivakovsky \cite{spivakovsky} proved (building on results of Hironaka \cite{Hir.nash}) that a variant of the conjecture holds, where one alternatingly applies Nash modifications and normalizations. For toric varieties, there is a reduction of the conjecture to a combinatorial problem about semi-groups by P. Gonz\'{a}lez P\'{e}rez and B. Teissier in \cite{perez-teissier}. Besides these partial results, the Nash-Semple conjecture remains widely open. 

To make a link between this conjecture and arc-wise analytic t-stratifications, the first observation one has to make is that one can always replace $K$ by a larger algebraically closed field (by extending scalars).
In this way, we may assume that $K$ admits a non-archimedean valuation $v\colon K\to \Gamma$, so that the notions of stratifications from this paper make sense.

Our feeling is that from the critical value function (Definition~\ref{def:critpt}) of an arc-wise analytic t-stratification of $K^n$ reflecting $X(K)$, one should be able to extract an upper bound for the number of Nash modifications needed to make $X$ smooth.
We did not yet check enough examples to specify this bound concretely, but let us at least formulate a weak statement precisely: We conjecture that for any function $f\colon \Delta_n \to \mathcal{P}(\Gamma^\times)$, there exists a bound $N(f) \in \N$ such that if $\cS$ is any arc-wise analytic t-stratification whose critical value function $f^{\cS}$ is equal to $f$, and if $\cS$ reflects $X(K)$, then 
the number of Nash modifications needed for $X$ is bounded by $N(f)$.

Note that for the cusp curves from Example~\ref{ex:contact} defined by $x^a = y^b$,
the critical value function ``knows'' the rational number $\frac ab$, and from that number, one can indeed determine the number of blow-ups needed to resolve the cusp singularity. Something similar seems to also work in higher dimension.

Assuming that one has a good candidate for $N(f^{\cS})$, proving that it works should be reasonably doable, e.g.\ as follows.

Firstly, one should show that an arc-wise analytic t-stratification $\cS$ reflecting $X(K)$ induces an arc-wise analtyic t-stratification $\cS^*$ reflecting $X^*(K)$. (Note that this needs generalizing the notion of arc-wise analytic t-stratifications to non-affine $X^*$.)
This crucially uses that we work with arc-wise analytic t-stratifications and not just t- or \lt-stratifications.
Indeed, one can think of the Nash modification as locally introducing derivatives as new coordinates.
Replacing a function $\arc_b$ by $(\arc_b, D\arc_b)$ preserves neither $d$-triviality nor 2-Taylor $d$-triviality, but it does preserve analyticity. 

Once $\cS^*$ is defined, one should be able to relate the critical value functions $f^{\cS}$ and $f^{\cS^*}$, and thereby prove $N(f^{\cS^*}) < N(f^{\cS})$, unless $N(f^{\cS})$ was already $0$.
Assuming that $N(f^{\cS}) = 0$ implies that $X$ is smooth, this completes the proof.\footnote{A technical issue is that the critical value function does not seem to detect normal crossing singularities, so what we really expect is that one can prove $N(f^{\cS^*}) < N(f^{\cS})$ only if $N(f^{\cS}) \geqslant 2$; to complete the proof, one would then prove that $N(f^{\cS}) \leqslant 1$ implies that one Nash modification suffices.}

\subsection{Open questions}
\label{seq:ques}

Currently, the only way in which we obtain the existence of \lt-stratifications is via the arc-wise analytic t-stratifications. Such a proof can only work in fields with analytic structure, and currently, we even additionally need that $K$ is algebraically closed.

\begin{ques}\label{q:hensel} Suppose that the theory of $K$ is $1$-h-minimal as defined in \cite{clu-hal-rid} (see Section~\ref{sec:modth}). Let $\chi\colon K^n\to \RV\eq$ be a definable map. Do \lt-stratifications reflecting $\chi$ always exist?
\end{ques}

We believe that the answer to this question is yes, and more generally that the t$^r$-stratifications from Section~\ref{sec:hierarchy} should exists in this generality for every integer $r \geqslant 1$.

\begin{ques}\label{q:converse} Does the converse of Theorem \ref{thm:sts-to-Lipstrats} hold, namely, are valuative Lipschitz stratifications always Lipschitz t-stratifications?
\end{ques}

Also here, we believe that the answer is yes. Note that this would imply that we get a truly equivalent way to define Lipschitz stratifications over $\C$ and $\R$: A ring-definable partition $\cS$ of $\C^n$ or $\R^n$ is a Lipschitz stratification if and only if the partition defined by the same formula in a non-standard model is a \lt-stratification.

\begin{ques}\label{q:arx-wise} The current definition of arc-wise analytic t-stratification only applies when the underlying valued field $K$ is algebraically closed. Can one extend this definition to other fields with analytic structure as defined in \cite{CLip} and prove their existence? Of special interest to us is the case of real closed valued fields with analytic structure (see also \cite{cubi-haskell}).  
\end{ques}

Note that adapting Definition~\ref{def:arc-wise-strats} in a too naive way would not yield a good notion; in particular, the implication to \lt-stratifications might fail.

\medskip

In Section~\ref{sec:hierarchy}, we suggested that there exists an entire hierarchy of notions of stratifications and gave two examples providing evidence that this hierarchy should be strict: one in an algebraically closed valued field, and the other one with a ring-definable stratification. However, we were not able to combine these two features.
This raises the following somewhat surprising question:

\begin{ques}
If $\cS$ is a ring-definable t-stratification in an algebraically closed valued field, is it then automatically a t$^r$-stratification for some $r \geqslant 2$, or for every $r \geqslant 2$, or even an arc-wise analytic t-stratification?
\end{ques}

\bibliographystyle{siam}

\end{document}